\documentclass[11pt,a4paper]{article}
\usepackage[a4paper]{geometry}
\usepackage{amssymb,latexsym,amsmath,amsfonts,amsthm}
\usepackage{graphicx}
\usepackage{bigcircle}

\usepackage{parskip}
\usepackage{epsfig}
\usepackage{comment}
\usepackage{subfigure}
\usepackage{graphicx,ebezier}
\usepackage{overpic}

\DeclareMathOperator{\diag}{diag} 

\newcommand{\er}{\mathbb{R}}
\newcommand{\cee}{\mathbb{C}}
\newcommand{\zet}{\mathbb{Z}}

\newcommand{\lam}{\lambda}
\newcommand{\Lam}{\Lambda}
\newcommand{\bol}{\hfill\square\\}
\newcommand{\til}{\tilde}
\newcommand{\wtil}{\widetilde}

\newcommand{\dee}{{\mathbb D}}

\renewcommand{\Re}{\mathrm{Re}\,}
\renewcommand{\Im}{\mathrm{Im}\,}

\newcommand{\err}{\mathcal{R}}

\newcommand{\hh}{\widehat}
\newcommand{\squ}{_{\mathrm{sq}}}
\newcommand{\al}{\alpha}
\newcommand{\ud}{\,\mathrm{d}}
\newcommand{\Aa}{\mathcal A}
\newcommand{\what}{\widehat}

\newtheorem{theorem}{Theorem}[section]
\newtheorem{lemma}[theorem]{Lemma}
\newtheorem{proposition}[theorem]{Proposition}

\newtheorem{corollary}[theorem]{Corollary}

\newtheorem{rhp}[theorem]{RH problem}

\theoremstyle{definition}

\theoremstyle{remark}

\newtheorem{remark}[theorem]{Remark}

\numberwithin{equation}{section}

\title{Non-intersecting squared Bessel paths at a hard-edge tacnode}

\author{Steven Delvaux\footnotemark[1]}
\date{\today}

\begin{document}

\maketitle
\renewcommand{\thefootnote}{\fnsymbol{footnote}}
\footnotetext[1]{Department of Mathematics, University of Leuven (KU Leuven),
Celestijnenlaan 200B, B-3001 Leuven, Belgium. email:
steven.delvaux\symbol{'100}wis.kuleuven.be. The author is a Postdoctoral Fellow
of the Fund for Scientific Research - Flanders (Belgium). \\
}

\begin{abstract} The squared Bessel process is a $1$-dimensional diffusion process related
to the squared norm of a higher dimensional Brownian motion. We study a model
of $n$ non-intersecting squared Bessel paths, with all paths starting at the
same point $a>0$ at time $t=0$ and ending at the same point $b>0$ at time
$t=1$. Our interest lies in the critical regime $ab=1/4$, for which the paths
are tangent to the hard edge at the origin at a critical time $t^*\in (0,1)$.
The critical behavior of the paths for $n\to\infty$ is studied in a scaling
limit with time $t=t^*+O(n^{-1/3})$ and temperature $T=1+O(n^{-2/3})$. This
leads to a critical correlation kernel that is defined via a new
Riemann-Hilbert problem of size $4\times 4$. The Riemann-Hilbert problem gives
rise to a new Lax pair representation for the Hastings-McLeod solution to the
inhomogeneous Pain\-lev\'e~II equation $q''(x) = xq(x)+2q^3(x)-\nu,$ where
$\nu=\alpha+1/2$ with $\alpha>-1$ the parameter of the squared Bessel process.
These results extend our recent work with Kuijlaars and Zhang \cite{DKZ} for
the homogeneous case $\nu = 0$.\smallskip


\end{abstract}

\setcounter{tocdepth}{2} \tableofcontents

\section{Introduction}

The motivation of this paper is the recent surge of interest in a model of
non-intersecting Brownian motions at a tacnode. The model is illustrated in the
third picture of Figure~\ref{fig:3cases0}, where we have two groups of Brownian
motions which are asymptotically supported inside two touching ellipses in the
time-space plane. The interest lies in the microscopic behavior of the paths
near the touching point of these two ellipses, i.e., near the tacnode.

The tacnode model was recently studied via different methods by several groups
of authors. The model is studied in a discrete symmetric setting by Adler,
Ferrari and Van Moerbeke \cite{AFvM11}. A different approach in this case is
due to Johansson \cite{Joh11}. Further developments are the study of a double
Aztec diamond model by Adler, Johansson and van Moerbeke \cite{AJvM}, and the
non-symmetric tacnode by Ferrari and Veto \cite{FV}.
Another model with a
tacnode but with very different properties is discussed in \cite{BorodinDuits}.

In a joint work with Kuijlaars and Zhang \cite{DKZ} we study the continuous
non-symmetric version of the tacnode model. Our results in \cite{DKZ} express
the critical correlation kernel in terms of a certain Riemann-Hilbert problem
(RH problem) of size $4\times 4$, related to a new Lax pair representation for
the Hastings-McLeod solution to the Painlev\'e~II equation. Recall that the
standard RH problem for the Painlev\'e~II equation has only size $2\times 2$
\cite{FN}; see also \cite{JKT} for a Lax pair with $3\times 3$ matrices.

Recently Duits and Geudens \cite{DG} use the $4\times 4$ RH problem from
\cite{DKZ} to describe a new critical phenomenon in the two-matrix model.
Interestingly the kernels in \cite{DKZ,DG} are built from (basically) the same
$4\times 4$ RH problem but in an essentially different way.


\begin{figure}[t]
\begin{center}
\subfigure{\label{figlargesep0}}\includegraphics[scale=0.3]{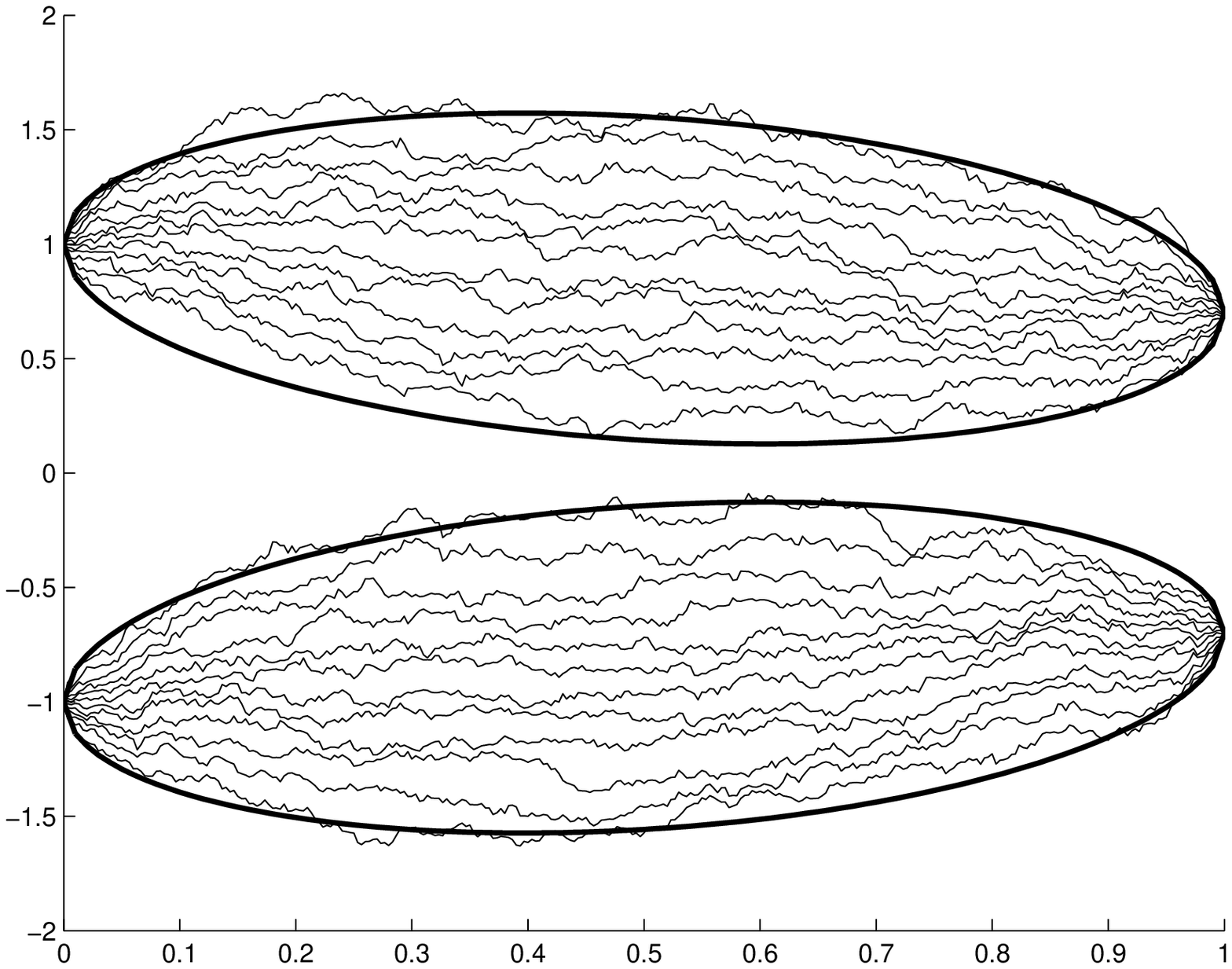}\hspace{5mm}
\subfigure{\label{figsmallsep0}}\includegraphics[scale=0.3]{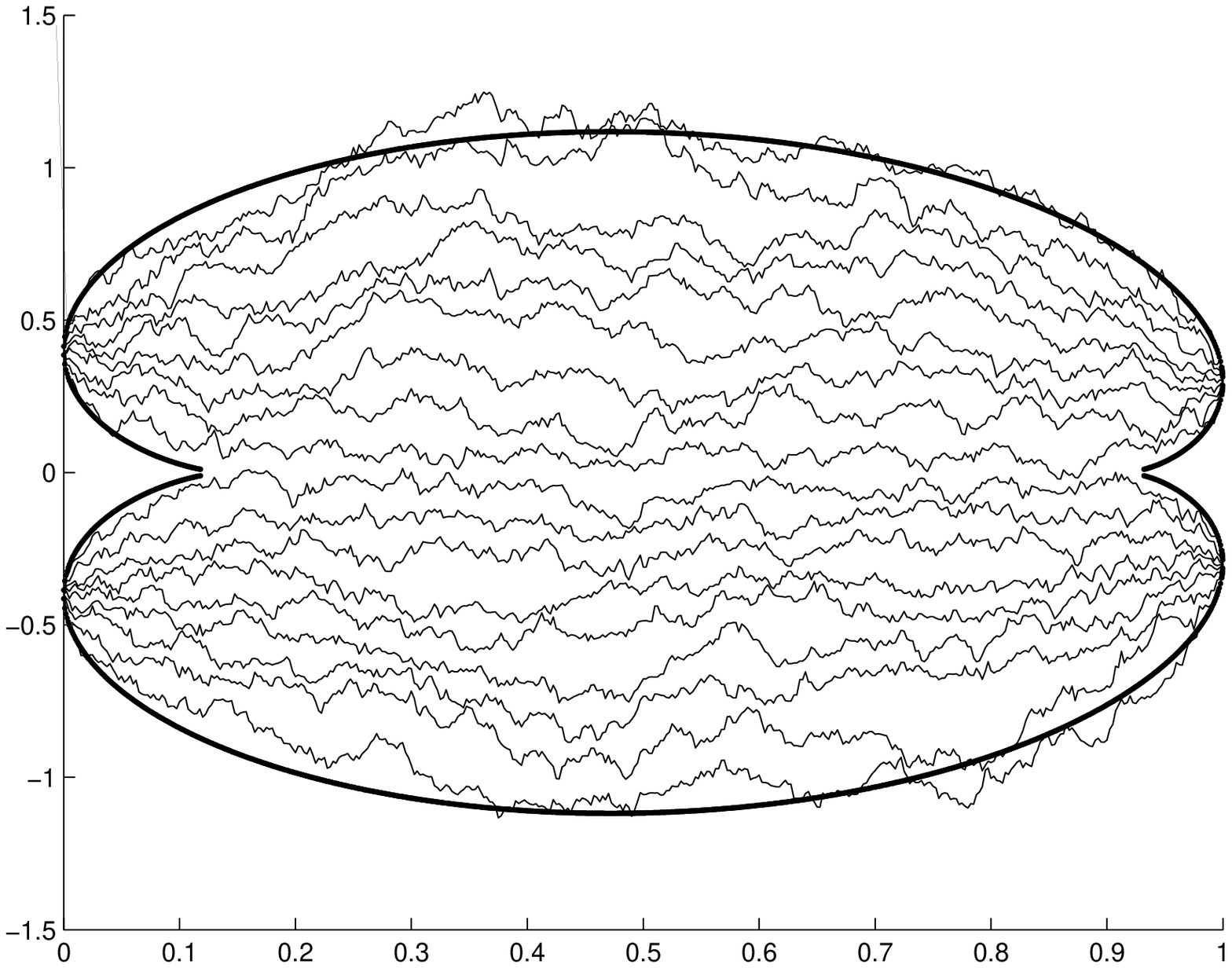}
\subfigure{\label{figcriticalsep0}}\includegraphics[scale=0.3]{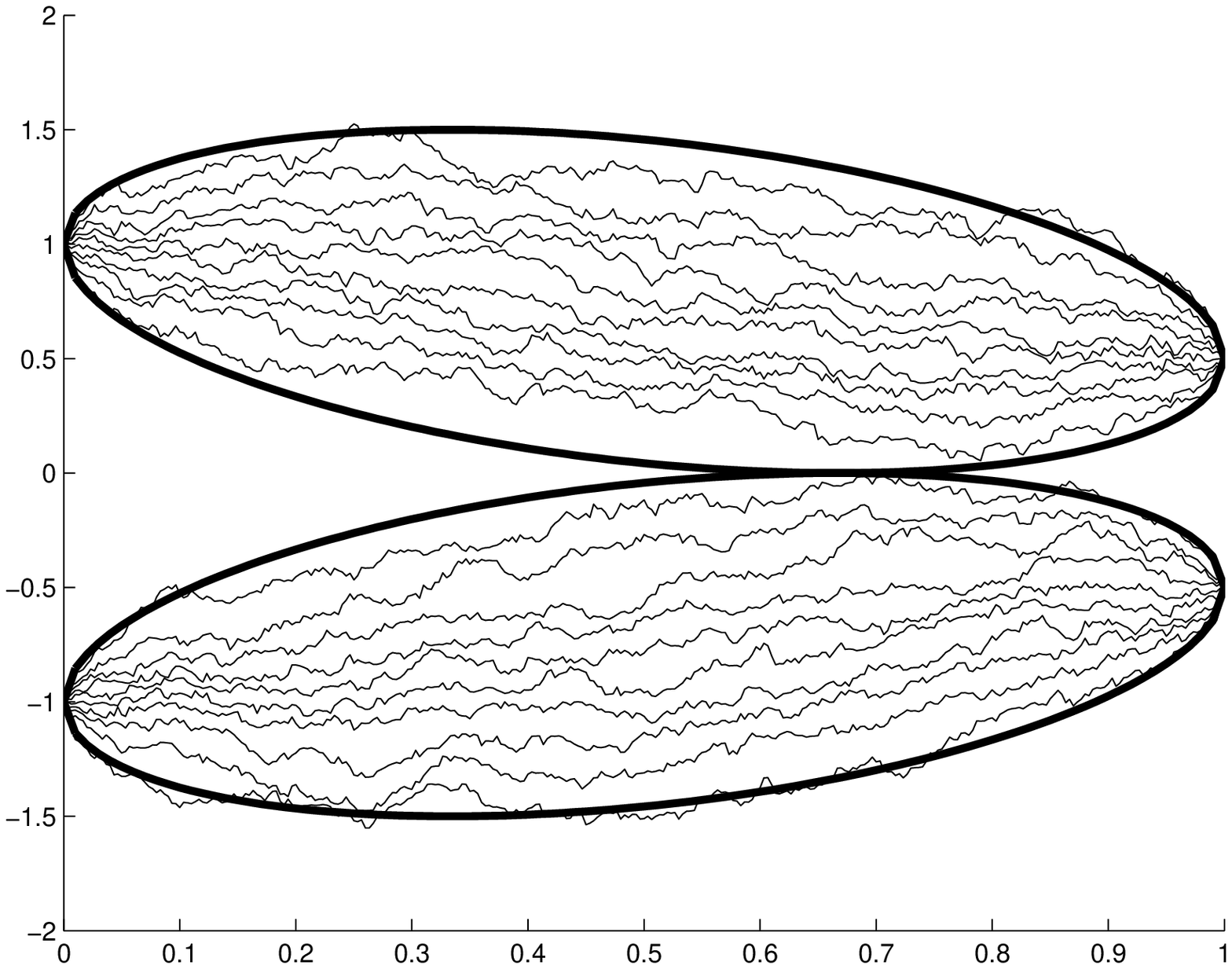}
\end{center}
\caption{$n=20$ non-intersecting Brownian motions at temperature $T=1$ with two
starting and two ending positions in the case of (a) large, (b) small, and (c)
critical separation between the endpoints. The horizontal axis stands for the
time, $t\in[0,1]$, and the vertical axis shows the positions of the Brownian
motions at time $t$. For $n\to\infty$ the Brownian motions fill a prescribed
region in the time-space plane which is bounded by the boldface lines in the
figure. In the case (c), the limiting support consists of two touching ellipses
which touch each other at a critical point which is a tacnode.}
\label{fig:3cases0}
\end{figure}

The goal of this paper is to extend the approach of \cite{DKZ} to a hard-edge
situation. To this end we consider a model of non-intersecting squared Bessel
paths \cite{DKRZ}, see also \cite{KMW,KMW2,KIK08,KT04,KT11}. In the critical
case we can get a situation where the limiting hull of the paths is touching
the hard edge at a certain critical time $t^*$. This is shown in the third and
fourth picture of Figure~\ref{fig:3cases}. We will refer to the touching point
as a \emph{hard-edge tacnode}.
 (This is an analogue of the \emph{higher-order tacnode} in the classification
scheme in \cite[p.~72]{Mir}.) We will see that the microscopic behavior of the
paths near the hard-edge tacnode is again expressed by a limiting kernel
defined in terms of a $4\times 4$ RH problem, which is now related to the
Hastings-McLeod solution to the \emph{inhomogeneous} Painlev\'e~II equation.

Our scaling limit near the hard-edge tacnode will be more general than the one
in \cite{DKZ}. In the latter paper we used a $O(n^{-2/3})$ scaling of the
endpoints, or equivalently of the temperature $T$. In the present paper we will
use an extra $O(n^{-1/3})$ scaling of the \emph{time} $t$, similarly to
\cite{AFvM11,DG,Joh11}. At the level of the RH problem this leads to an extra
parameter $\tau$, as in \cite{DG}. While the extension of the RH problem with
the parameter $\tau$ was already known to the authors of \cite{DKZ}, the
solvability of the extended RH problem is a quite non-trivial fact. This
solvability was recently established by Duits and Geudens in the symmetric case
\cite{DG}. We will make use of their ideas. We note that the solvability of the
extended RH
problem  in the general non-symmetric case has not been settled yet.\\

\begin{figure}[t]
\begin{center}
\subfigure{\label{figlargesep}}\includegraphics[scale=0.3]{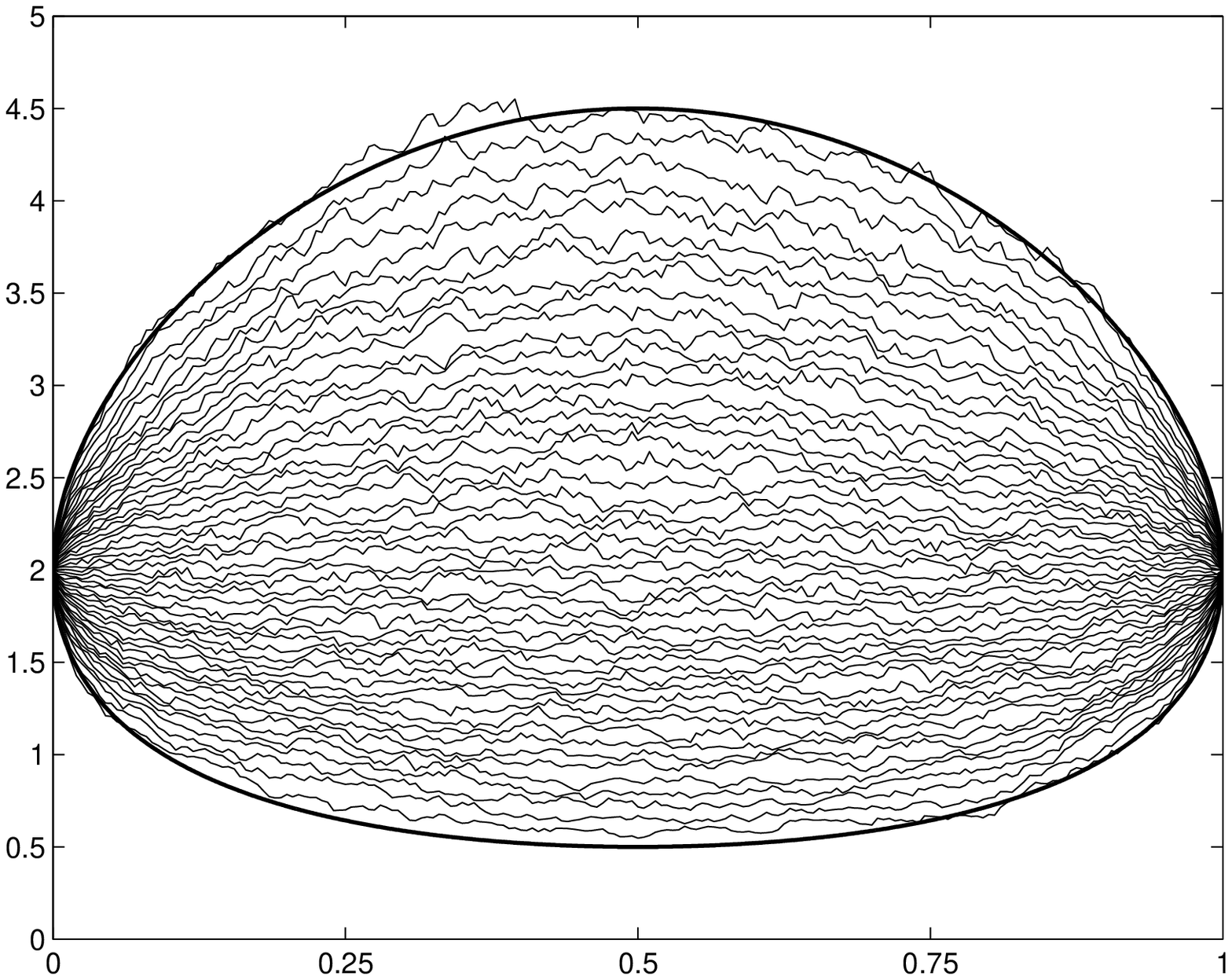}\hspace{5mm}
\subfigure{\label{figsmallsep}}\includegraphics[scale=0.3]{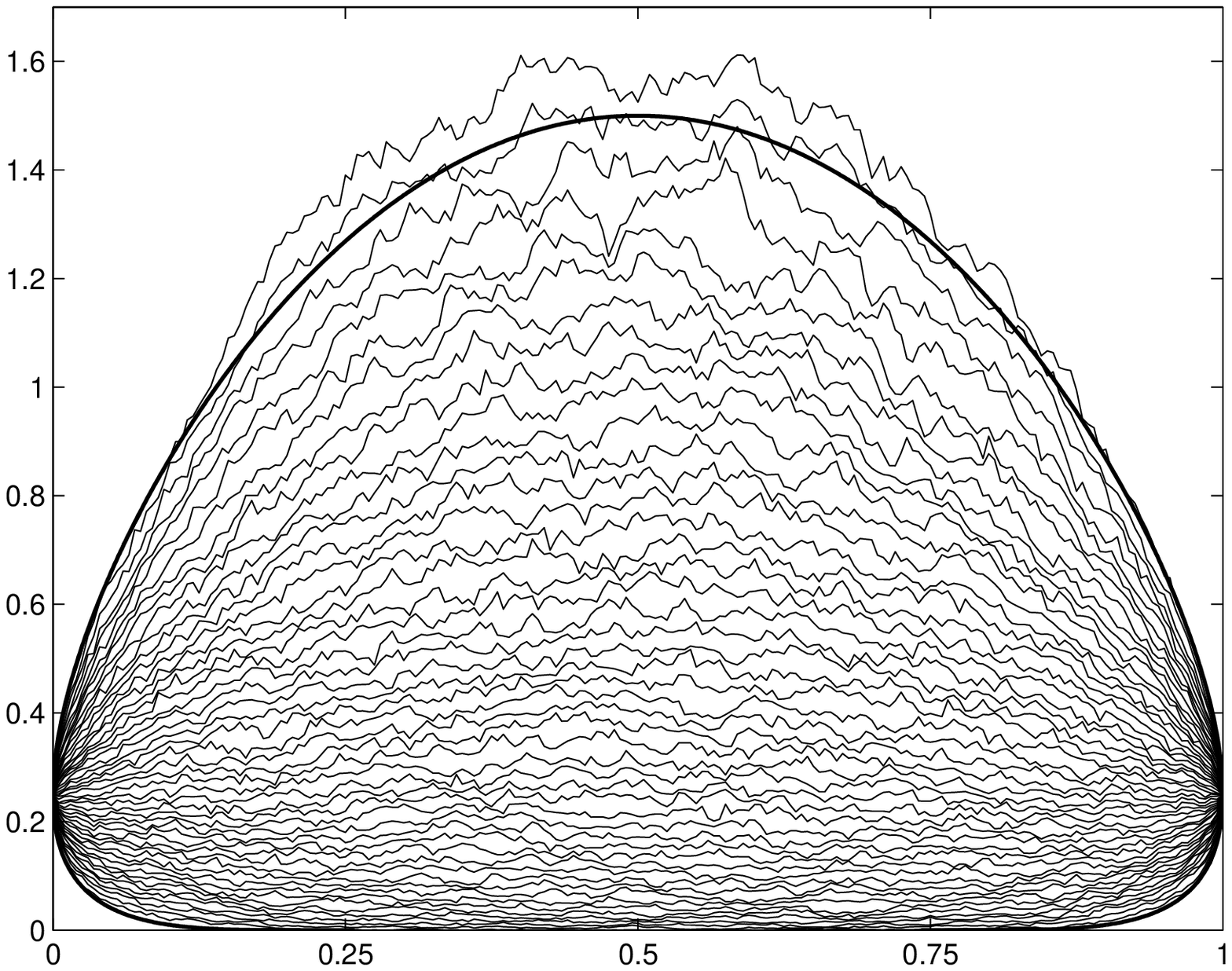}
\subfigure{\label{figcriticalsep}}\includegraphics[scale=0.3]{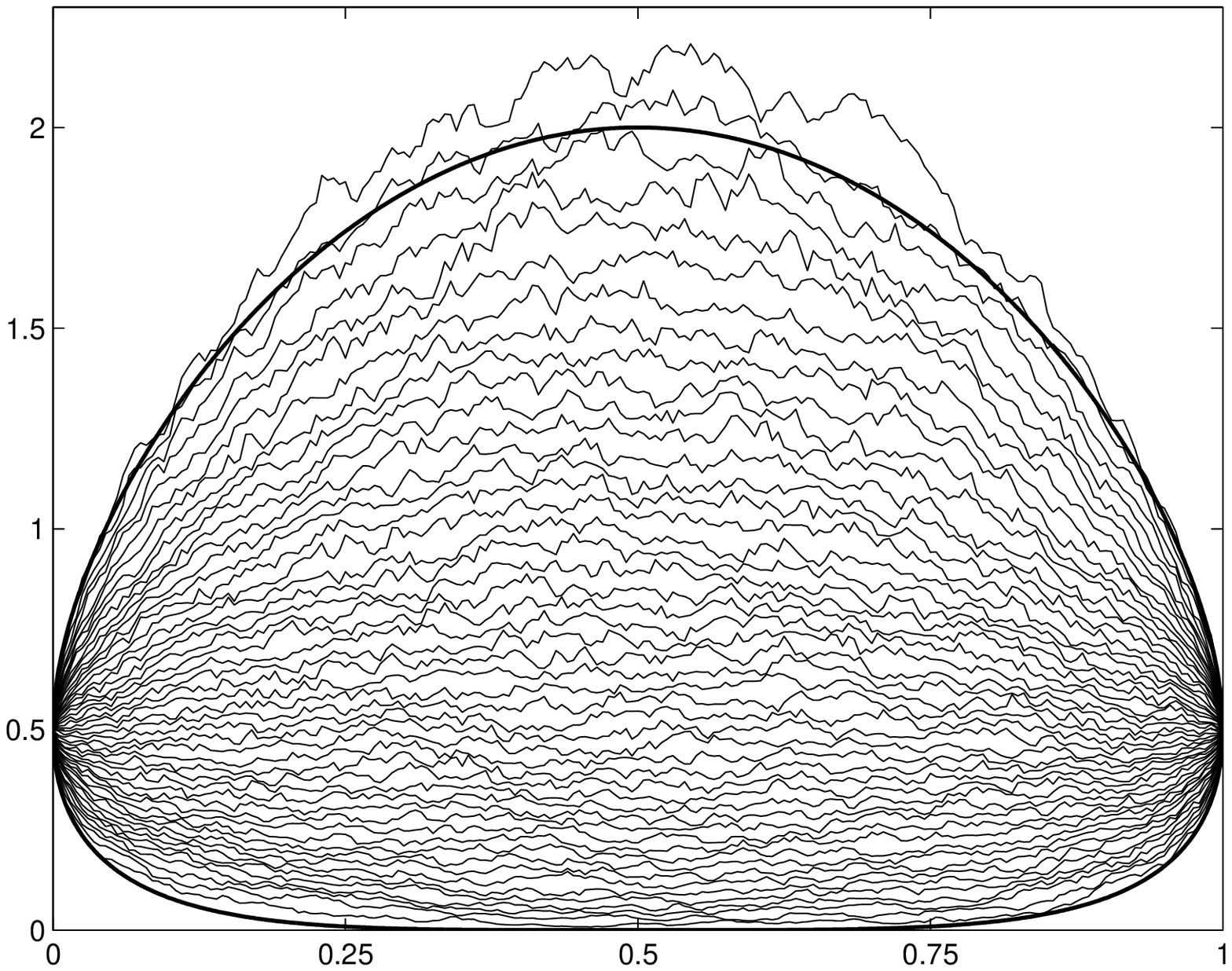}
\subfigure{\label{figcriticalzoom}}\includegraphics[scale=0.36]{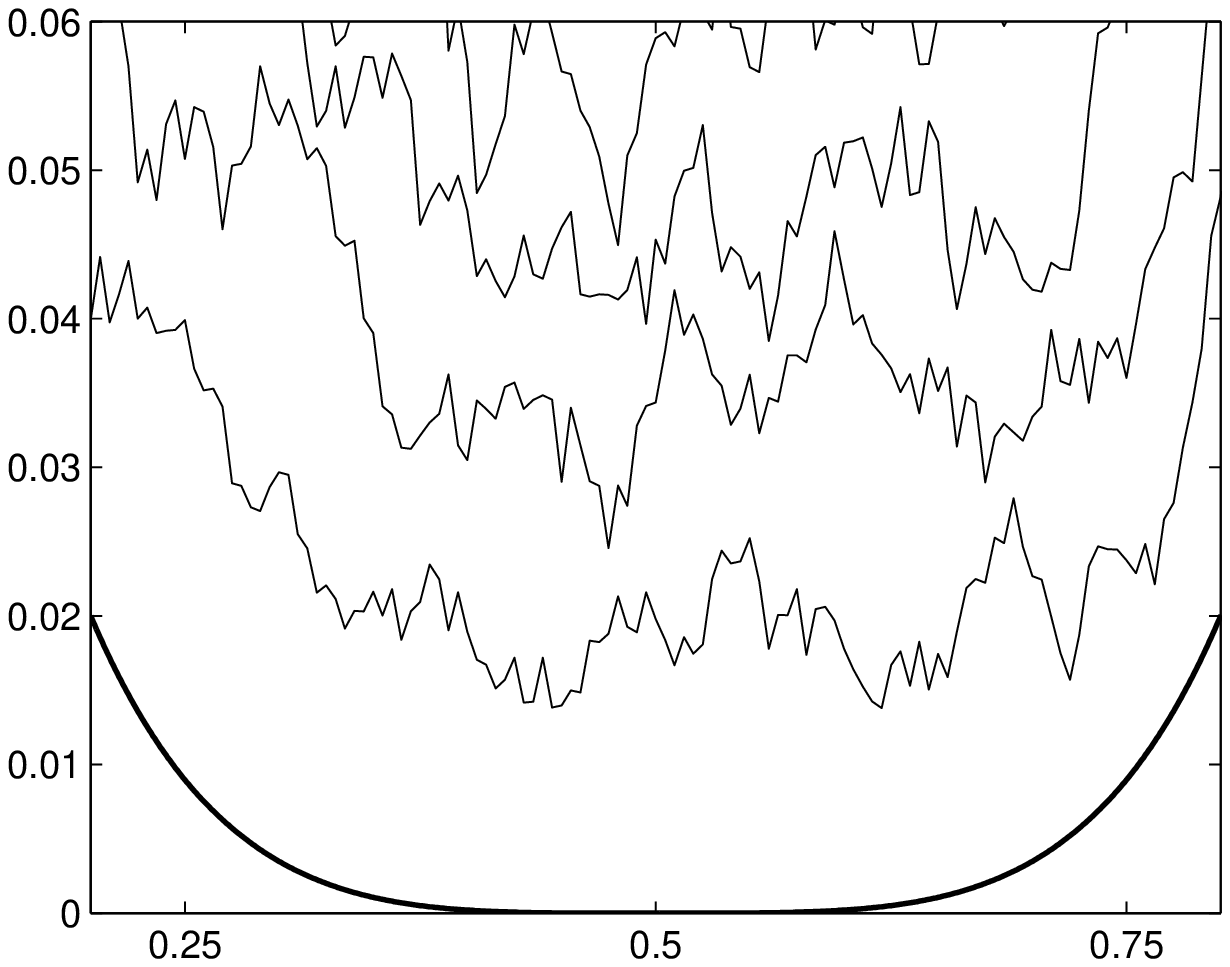}
\end{center}
\caption{$n=50$ non-intersecting squared Bessel paths at temperature $T=1$
starting in $a>0$ and ending in $b>0$ with (a) $a=b=2$, (b) $a=b=0.25$, and (c)
$a=b=0.5$, illustrating the case of large, small and critical ending positions,
respectively. In the case (c), the paths are tangent to the hard edge at time
$t^*=1/2$ (hard-edge tacnode) with a fourth order contact. A magnified picture
near the hard-edge tacnode is shown in (d).} \label{fig:3cases}
\end{figure}

Now we describe our model in more detail, following \cite{DKRZ}. Let $X(\tau)$
be a squared Bessel process with parameter $\alpha>-1$, i.e., a diffusion
process (a strong Markov process on $\er$ with continuous sample paths) with
transition probability density \cite{KarShr,KOC}
\begin{align}
p_{\tau}^{\alpha}(x,y)&=\frac{1}{2\tau}\left(\frac{y}{x}\right)^{\alpha/2}e^{-\frac{x+y}{2\tau}}
I_{\alpha}\left(\frac{\sqrt{xy}}{\tau}\right), &\qquad x,y>0,\label{probdensity:alpha} \\
p_{\tau}^{\alpha}(0,y)&=\frac{y^{\alpha}}{(2\tau)^{\alpha+1}\Gamma(\alpha+1)}e^{-\frac{y}{2\tau}},
&\qquad y>0,\label{probdensity:alpha0}
\end{align}
where $\tau$ stands for the time, $x,y$ for the position and where
\begin{equation}
I_{\alpha}(z)=\sum_{k=0}^{\infty}\frac{(z/2)^{2k+\alpha}}{k!\Gamma(k+\alpha+1)}
\end{equation}
is the modified Bessel function of the first kind of order $\alpha$; see
\cite[p.375]{AS}. If $d=2(\alpha+1)$ is an integer, the squared Bessel process
behaves like the square of the distance to the origin of a $d$-dimensional
Brownian motion. Some applications of this diffusion process in mathematical
finance and other fields can be found in \cite{KarShr}, see also
\cite{GJY,KT04,KT11,Platen}.

We will consider $n$ independent squared Bessel processes $X_j(\tau)$,
$j=1,\ldots,n$. It is then convenient to perform a rescaling of the time
$\tau\mapsto t$ by \begin{equation}\label{time:rescaling} \tau = \frac{T}{2n}t.
\end{equation}
The proportionality parameter $T>0$ can be interpreted as the temperature. It
has a natural physical interpretation, in the sense that high temperature
corresponds to a large variance of the squared Bessel process and vice versa.

We consider a model of $n$ independent squared Bessel processes, starting in
$a>0$ at time $t=0$, ending in $b>0$ at time $t=1$, and conditioned (in the
sense of Doob) not to intersect in the whole time interval $t\in (0,1)$. This
model was studied by Kuijlaars, Mart\'inez-Finkelshtein and Wielonsky
\cite{KMW,KMW2} for $b=0$ and for general $a,b>0$ in the non-critical case in
\cite{DKRZ}. We also mention the large deviation principle recently established for a
closely related model in
\cite{HK}.


Let us fix the temperature $T=1$. 
There are three cases to distinguish, see Figure~\ref{fig:3cases}. If $ab>1/4$
then the paths do not reach the hard edge at the origin. If $ab<1/4$ then they
hit the hard edge and stick to it during a certain time interval $[t_1,t_2]$.
Finally, for the critical separation $ab=1/4$ the paths are touching to the
hard edge at one particular time $t=t^*\in [0,1]$, given by
\begin{equation}\label{def:tcrit}
t^* = \frac{\sqrt{a}}{\sqrt{a}+\sqrt{b}}.
\end{equation}
These three cases are illustrated in Figure~\ref{fig:3cases}(a)--(c).

\begin{figure}[tbp]
\hspace{-15mm} \centering \begin{overpic}[scale=0.27]%
{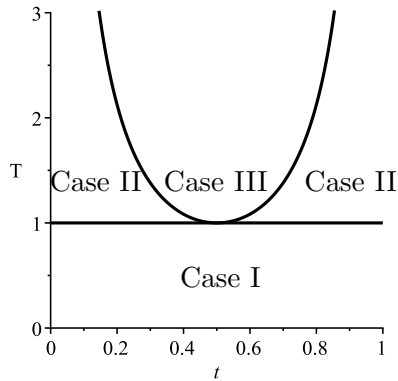}%
      \put(42,50){Case III}
      \put(13.5,50){Case II}
      \put(77,50){Case II}
      \put(46,26){Case I}
\end{overpic}
\caption{Phase diagram in the $(t,T)$ plane for the non-intersecting squared
Bessel paths with $a=b=1/2$. The critical tacnode point is located at
$(t,T)=(1/2,1)$.} \label{fig:phasediagram}
\end{figure}

In this paper we find it convenient to take \emph{fixed} endpoints $a,b$
satisfying $ab=1/4$. We consider the time $t\in [0,1]$ and temperature $T>0$ to
be varying. Figure~\ref{fig:phasediagram} shows the phase diagram in the
$(t,T)$-plane if $a=b=1/2$. The diagram shows three regions. Case~I corresponds
to temperature $T<1$. Then the paths do not reach the hard edge and at each
fixed time $t\in (0,1)$ they are asymptotically distributed on the interval
$[p,q]=[p(t),q(t)]$ with \cite[Sec.~4.3]{DKRZ}
\begin{align}\label{MP:endpoint1}
\sqrt{p}:=&(1-t)\sqrt{a}+t\sqrt{b}-\sqrt{2t(1-t)T}>0,\\
\label{MP:endpoint2} \sqrt{q}:=&(1-t)\sqrt{a}+t\sqrt{b}+\sqrt{2t(1-t)T}.
\end{align} The limiting density of the paths on $[p,q]$ is the
Marcenko-Pastur density \cite[Remark~1]{DKRZ}.

In Cases~II and III we have $T>1$, with the paths reaching the hard edge in
Case~III and staying away from it in Case~II. The limiting density is a
transformation of the one in \cite{DKV}; it is different from the
Marcenko-Pastur density. We note that the limiting density and support are
always independent of~$\alpha$.

The critical curves in the phase diagram are exactly like in the corresponding
model of non-intersecting Brownian motions \cite[Sec.~1.9]{DelKui1}, thanks to
\cite[Remark~1]{DKRZ}. But the nature of the phase transitions is different due
to the presence of the hard edge.

At the transition of Cases~I and II we expect a phase transition in terms of
the inhomogeneous Painlev\' e~II equation \eqref{def:Painleve2} with
$\nu=\alpha+1/2$. This transition may not be felt in the correlation kernel but
should certainly manifest itself in the critical limits of the recurrence
coefficients of the multiple orthogonal polynomials, as in \cite{DelKui1}. At
the transition of Cases~II and III we expect the same phase transition as in
\cite{KMW2}. We note that the transition curve is given by the
equation~\cite[Sec.~1.9]{DelKui1}
\begin{equation}\label{vglboundarycurve} T = \frac{a(1-t)^2+bt^2}{t(1-t)}.
\end{equation}
The topic of the present paper is the multi-critical tacnode point
$(t,T)=(t^*,1)$ which lies at the transition of each of the Cases~I, II and III
in the phase diagram.



\section{Statement of results}

\subsection{A new $4\times 4$ Riemann-Hilbert problem}

We introduce a new $4\times 4$ RH problem, which is a variant of the RH problem
in \cite{DKZ,DG}. The RH problem generalizes the one in \cite{DKZ,DG} in the
sense that it contains non-trivial \lq Stokes multipliers\rq\ $e^{\nu\pi i}$
and $e^{-\nu\pi i}$. As in \cite{DG} we also have an extra parameter $\tau$.

The RH problem will have jumps on a contour in the complex plane consisting of
$10$ rays emanating from the origin. More precisely, let us fix two numbers
$\varphi_1,\varphi_2$ such that
\begin{equation}\label{def:angles}
0<\varphi_1<\varphi_2<\pi/2.
\end{equation}
Then we define the half-lines $\Gamma_k$, $k=0,\ldots,9$, by
\begin{equation}\label{def:rays1}
\Gamma_0=\er_+,\quad \Gamma_1=e^{i\varphi_1}\er_+,\quad
\Gamma_2=e^{i\varphi_2}\er_+,\quad \Gamma_3=e^{i(\pi-\varphi_2)}\er_+,\quad
\Gamma_4=e^{i(\pi-\varphi_1)}\er_+,
\end{equation}
and
\begin{equation}\label{def:rays2}
\Gamma_{5+k}=-\Gamma_k,\qquad k=0,\ldots,4.
\end{equation}
All rays $\Gamma_k$, $k=0,\ldots,9$, are oriented towards infinity, as shown in
Figure~\ref{fig:modelRHP}. We also denote by $\Omega_k$ the region in $\cee$
which lies between the rays $\Gamma_k$ and $\Gamma_{k+1}$, for $k=0,\ldots,9$,
where we identify $\Gamma_{10}:=\Gamma_0$.

Using this definition of the rays $\Gamma_k$, we now consider the following RH
problem.

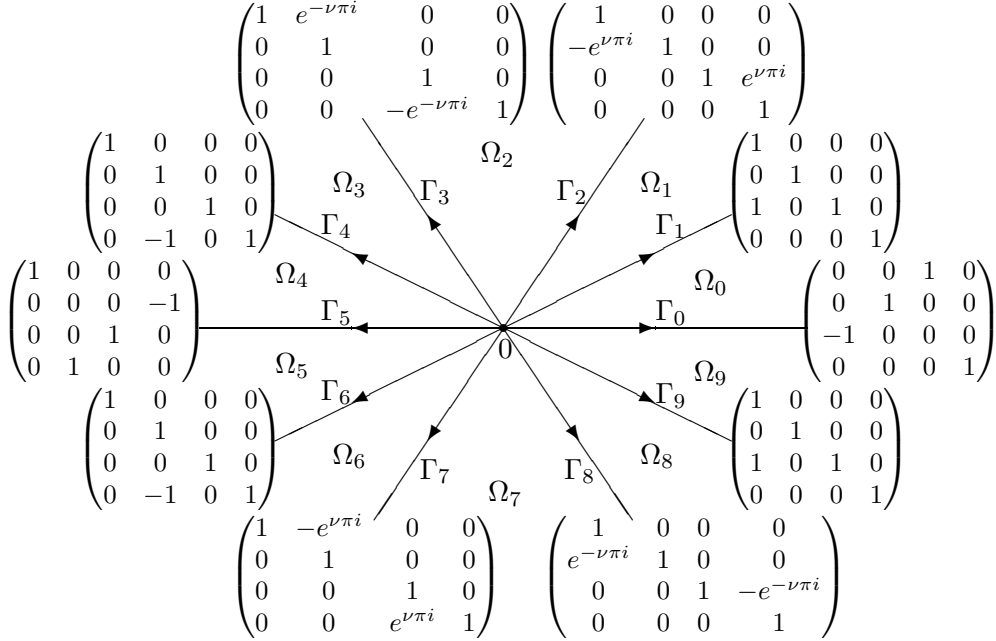
\begin{figure}[t]
\vspace{14mm}
\begin{center}
   \setlength{\unitlength}{1truemm}
   \begin{picture}(100,70)(-5,2)
       \put(40,40){\line(1,0){40}}
       \put(40,40){\line(-1,0){40}}
       \put(40,40){\line(2,1){30}}
       \put(40,40){\line(2,-1){30}}
       \put(40,40){\line(-2,1){30}}
       \put(40,40){\line(-2,-1){30}}
       \put(40,40){\line(2,3){18.5}}
       \put(40,40){\line(2,-3){17}}
       \put(40,40){\line(-2,3){18.5}}
       \put(40,40){\line(-2,-3){17}}
       \put(40,40){\thicklines\circle*{1}}
       \put(39.3,36){$0$}
       \put(60,40){\thicklines\vector(1,0){.0001}}
       \put(20,40){\thicklines\vector(-1,0){.0001}}
       \put(60,50){\thicklines\vector(2,1){.0001}}
       \put(60,30){\thicklines\vector(2,-1){.0001}}
       \put(20,50){\thicklines\vector(-2,1){.0001}}
       \put(20,30){\thicklines\vector(-2,-1){.0001}}
       \put(50,55){\thicklines\vector(2,3){.0001}}
       \put(50,25){\thicklines\vector(2,-3){.0001}}
       \put(30,55){\thicklines\vector(-2,3){.0001}}
       \put(30,25){\thicklines\vector(-2,-3){.0001}}

       \put(60,41){$\Gamma_0$}
       \put(60,52.5){$\Gamma_1$}
       \put(47,57){$\Gamma_2$}
       \put(29,57){$\Gamma_3$}
       \put(16,52.5){$\Gamma_4$}
       \put(16,41){$\Gamma_5$}
       \put(16,30.5){$\Gamma_6$}
       \put(29,20){$\Gamma_7$}
       \put(48,20){$\Gamma_8$}
       \put(60,30){$\Gamma_{9}$}

       \put(65,45){$\small{\Omega_0}$}
       \put(58,58){$\small{\Omega_1}$}
       \put(37,62){$\small{\Omega_2}$}
       \put(17.5,58){$\small{\Omega_3}$}
       \put(10,46){$\small{\Omega_4}$}
       \put(10,34){$\small{\Omega_5}$}
       \put(17.5,22){$\small{\Omega_6}$}
       \put(38,17){$\small{\Omega_7}$}
       \put(58,22){$\small{\Omega_8}$}
       \put(65,33){$\small{\Omega_{9}}$}

       \put(78.5,40){$\small{\begin{pmatrix}0&0&1&0\\ 0&1&0&0\\ -1&0&0&0\\ 0&0&0&1 \end{pmatrix}}$}
       \put(69,57){$\small{\begin{pmatrix}1&0&0&0\\ 0&1&0&0\\ 1&0&1&0\\ 0&0&0&1 \end{pmatrix}}$}
       \put(45,74){$\small{\begin{pmatrix}1&0&0&0\\ -e^{\nu\pi i}&1&0&0\\ 0&0&1&e^{\nu\pi i}\\ 0&0&0&1 \end{pmatrix}}$}
       \put(4,74){$\small{\begin{pmatrix}1&e^{-\nu\pi i}&0&0\\ 0&1&0&0\\ 0&0&1&0\\ 0&0&-e^{-\nu\pi i}&1 \end{pmatrix}}$}
       \put(-16,57){$\small{\begin{pmatrix}1&0&0&0\\ 0&1&0&0\\ 0&0&1&0\\ 0&-1&0&1\end{pmatrix}}$}
       \put(-26,40){$\small{\begin{pmatrix}1&0&0&0\\ 0&0&0&-1\\ 0&0&1&0\\ 0&1&0&0 \end{pmatrix}}$}
       \put(-16,23){$\small{\begin{pmatrix}1&0&0&0\\ 0&1&0&0\\ 0&0&1&0\\ 0&-1&0&1 \end{pmatrix}}$}
       \put(4,6){$\small{\begin{pmatrix}1&-e^{\nu\pi i}&0&0\\ 0&1&0&0\\ 0&0&1&0\\ 0&0&e^{\nu\pi i}&1 \end{pmatrix}}$}
       \put(45,6){$\small{\begin{pmatrix}1&0&0&0\\ e^{-\nu\pi i}&1&0&0\\ 0&0&1&-e^{-\nu\pi i}\\ 0&0&0&1 \end{pmatrix}}$}
       \put(69,23){$\small{\begin{pmatrix}1&0&0&0\\ 0&1&0&0\\ 1&0&1&0\\ 0&0&0&1\end{pmatrix}}$}

  \end{picture}
   \vspace{0mm}
   \caption{The figure shows the jump contours $\Gamma_k$ in the complex $\zeta$-plane and the corresponding jump
   matrix
   $J_k$ on $\Gamma_k$, $k=0,\ldots,9$, in the RH problem for  $M = M(\zeta)$. We denote by $\Omega_k$ the region between
    the rays $\Gamma_k$ and $\Gamma_{k+1}$.}
   \label{fig:modelRHP}
\end{center}
\end{figure}

\begin{rhp}\label{rhp:modelM} We look for a $4\times 4$ matrix valued function
$M(\zeta)$ (which also depends parametrically on $\nu>-1/2$ and on the complex
parameters $r_1,r_2,s,\tau\in\cee$) satisfying
\begin{itemize}
\item[(1)] $M(\zeta)$ is analytic for $\zeta\in\cee\setminus\left(\bigcup_{k=0}^{9}
\Gamma_k\right)$.
\item[(2)] For $\zeta\in\Gamma_k$, the limiting values
\[ M_+(\zeta) = \lim_{z \to \zeta, \,  z\textrm{ on $+$-side of }\Gamma_k} M(z), \qquad
    M_-(\zeta) = \lim_{z \to \zeta, \, z\textrm{ on $-$-side of }\Gamma_k} M(z) \]
exist, where the $+$-side and $-$-side of $\Gamma_k$ are the sides which lie on
the left and right of $\Gamma_k$, respectively, when traversing $\Gamma_k$
according to its orientation. These limiting values satisfy the jump relation
\begin{equation}\label{jumps:M}
M_{+}(\zeta) = M_{-}(\zeta)J_k(\zeta),\qquad k=0,\ldots,9,
\end{equation}
where the jump matrix $J_k(\zeta)$ for each ray $\Gamma_k$ is shown in
Figure~\ref{fig:modelRHP}.
\item[(3)] As $\zeta\to\infty$ we have
\begin{multline}
\label{M:asymptotics} M(\zeta) =
\left(I+\frac{M_1}{\zeta}+\frac{M_2}{\zeta^2}+O\left(\frac{1}{\zeta^3}\right)\right)
\diag((-\zeta)^{-1/4},\zeta^{-1/4},(-\zeta)^{1/4},\zeta^{1/4})
\\
\times
\Aa\diag\left(e^{-\theta_1(\zeta)+\tau\zeta},e^{-\theta_2(\zeta)-\tau\zeta},e^{\theta_1(\zeta)+\tau\zeta},e^{\theta_2(\zeta)-\tau\zeta}\right),
\end{multline}
where the coefficient matrices $M_1,M_2,\ldots$ are independent of $\zeta$, and
with
\begin{equation}\label{mixing:matrix}
\Aa:=\frac{1}{\sqrt{2}}\begin{pmatrix} 1 & 0 & -i & 0 \\
0 & 1 & 0 & i \\
-i & 0 & 1 & 0 \\
0 & i & 0 & 1 \\
\end{pmatrix},
\end{equation}
\begin{equation}\label{def:theta1}
\theta_1(\zeta) = \frac{2}{3}r_1(-\zeta)^{3/2}+2s(-\zeta)^{1/2},\qquad
\theta_2(\zeta) = \frac{2}{3}r_2\zeta^{3/2}+2s\zeta^{1/2}.
\end{equation}
Here we use the principal branches of the fractional powers, i.e., $\zeta^{a}:=|\zeta|^{a}e^{i a\arg\zeta}$ with argument function $|\arg\zeta|<\pi$, for any $a\in\er$.
Note that $\zeta^a$ and $(-\zeta)^a$ have a branch cut for $\zeta$ along the negative and positive real line respectively.
\item[(4)] As $\zeta\to 0$ we have
\begin{align}\label{Mzero:smaller} & M(\zeta) = O(\zeta^{\nu}),\quad M^{-1}(\zeta) =
O(\zeta^{\nu}),&\qquad \hbox{if $\nu\leq 0$},
\end{align}
and \begin{align}\label{Mzero:greater} \left\{\begin{array}{ll}
M(\zeta)\diag(\zeta^{-\nu},\zeta^{\nu},\zeta^{\nu},\zeta^{-\nu})=O(1),& \quad
\zeta\in \Omega_1\cup\Omega_8,\\
M(\zeta)\diag(\zeta^{\nu},\zeta^{-\nu},\zeta^{-\nu},\zeta^{\nu})=O(1),&\quad
\zeta\in \Omega_3\cup\Omega_6,\\
M(\zeta)=O(\zeta^{-\nu}),&\quad\textrm{elsewhere,}
\end{array}\right. \hbox{if
$\nu\geq 0$},
\end{align}
where the $O$-symbol is defined entrywise.
\end{itemize}
\end{rhp}

More detailed information on the behavior of $M(\zeta)$ for $\zeta\to 0$ is
given in Proposition~\ref{prop:behaviorMzero}.

It follows from standard arguments (e.g.~\cite{Dei}) that the solution to the
RH problem~\ref{rhp:modelM} is unique if it exists. The existence of the
solution is non-trivial and is stated in the following theorem, whose proof
will be given in Section~\ref{section:vanishing}.

\begin{theorem}(Existence).\label{theorem:solvability} Assume $\nu>-1/2$,
$r_1=r_2=:r>0$, $s,\tau\in\er$. Then the RH problem~\ref{rhp:modelM} for
$M(\zeta)$ is uniquely solvable. The same is true when $r_1,r_2,s,\tau$ belong
to a sufficiently small complex neighborhood of the above values.
\end{theorem}

The next theorem provides a connection with the \emph{inhomogeneous
Painlev\'e~II equation}
\begin{equation}\label{def:Painleve2}
    q''(x) = xq(x)+2q^3(x)-\nu,
\end{equation}
where the prime denotes the derivative with respect to $x$. The
\emph{Hastings-McLeod solution} \cite{FIKN,HML} is the special solution $q(x)$
of \eqref{def:Painleve2} which is real for real $x$ and satisfies
\begin{align}\label{def:HastingMcLeod1}
q(x)\sim \frac{\nu}{x}\qquad x\to +\infty, \\
\label{def:HastingMcLeod2} q(x)\sim \sqrt{\frac{-x}{2}}\qquad x\to -\infty.
\end{align}
We also define the \emph{Hamiltonian} $u(x)$ by
\begin{equation} \label{def:Hamiltonean}
    u(x) := (q'(x))^2-xq^2(x)-q^4(x)+2\nu q(x)
\end{equation}
and note that
\begin{equation} \label{Hamiltonean:der}
    u'(x) = -q^2(x).
\end{equation}

\begin{theorem}\label{theorem:Painleve2modelrhp} ($M_1$ vs.~the Painlev\'e~II
equation).
Let the parameters $\nu>-1/2$, $r_1=r_2=1$, $s,\tau\in\er$ in
\eqref{M:asymptotics}--\eqref{def:theta1} be fixed. The residue matrix $M_1$ in
\eqref{M:asymptotics} takes the form
\begin{equation}\label{M1:explicit}
M_1 = \begin{pmatrix} a & b & ic & id \\
-b & -a & id & ic \\
ie & if & -g & h \\
if & ie & -h & g
\end{pmatrix}
\end{equation}
where $a,b,c,\ldots$ are real valued constants depending on $\nu$, $s$ and
$\tau$. We have
\begin{align}
\label{d:Painleve2} d &= 2^{-1/3} q\left(2^{2/3} (2s-\tau^2)\right),
\\
\label{c:Hamiltonean} c &= -2^{-1/3}u(2^{2/3} (2s-\tau^2))+s^2,
\end{align}
with $q$ the Hastings-McLeod solution to the inhomogeneous Painlev\'e~II
equation \eqref{def:Painleve2}--\eqref{def:HastingMcLeod2}, and with $u$ the
Hamiltonian in \eqref{def:Hamiltonean}.
\end{theorem}

Theorem~\ref{theorem:Painleve2modelrhp} will be proved in
Section~\ref{section:proofTheoremP2}. In the homogeneous case $\nu=0$ the
theorem recovers some results from \cite{DKZ}, see also \cite{DG}. In the proof
we also obtain identities relating the other entries of $M_1$ in terms of the
entries $c$ and $d$, see e.g.\ \eqref{DG:entries}. We trust that the notations
$a,b$ in \eqref{M1:explicit} will not lead to confusion with the endpoints of
the squared Bessel paths.

The most difficult part of Theorem~\ref{theorem:Painleve2modelrhp} is proving
that $q(x)$ is precisely the \emph{Hastings-McLeod solution} to the Painlev\'
e~II equation. This will require asymptotic results that we derive in
Sections~\ref{section:asyM} and \ref{section:asyM:minusinfty}.


Finally, we transform the RH matrix $M(\zeta)$ into a new matrix $\what M(\zeta)$ as follows
\begin{equation}\label{Mhat}
\what M(\zeta) := \diag(\zeta^{1/4},\zeta^{-1/4},\zeta^{1/4},\zeta^{-1/4})\diag\left(\begin{pmatrix} 1 & -1 \\
1 & 1\end{pmatrix},\begin{pmatrix} 1 & 1 \\
-1 & 1\end{pmatrix}\right)M(\zeta^{1/2}).
\end{equation}
The transformed matrix $\what M(\zeta)$ depends on the same parameters $r_1,r_2,s,\tau,\nu$ as $M(\zeta)$.
The matrix $\what M(\zeta)$ satisfies a RH problem by itself but we will not state
it here.

\subsection{Critical asymptotics of squared Bessel paths}

Now we describe the critical asymptotics of the non-intersecting squared Bessel
paths. We will work under the triple scaling limit where the endpoints $a$, $b$
are fixed such that
\begin{equation}
    \label{doublescaling:ab} ab =
    1/4.
\end{equation}
We will let the time $t$ and the temperature $T$ depend on $n$ such that
\begin{eqnarray}
    \label{doublescaling:t} t &=& \frac{\sqrt{a}}{\sqrt{a}+\sqrt{b}} + K n^{-1/3},\\
    \label{doublescaling:T} T &=& 1 + L n^{-2/3},\end{eqnarray}
where $K,L$ are arbitrary real constants.

Denote again by $\alpha$ the parameter of the non-intersecting squared Bessel
paths. It turns out that under the scaling limit
\eqref{doublescaling:ab}--\eqref{doublescaling:T},  the analysis of the squared
Bessel paths for large $n$ leads to the RH problem~\ref{rhp:modelM}, with the
parameters $r_1=r_2=1$ and with $\nu$, $s=s^*$ and $\tau=\tau^*$ given by
\begin{align}
\nu &=\alpha+1/2,\\ \label{doublescaling:s} s^* &= \frac{K^2(\sqrt{a}+\sqrt{b})^4 -L}{2},\\
\label{doublescaling:tau} \tau^* &= -K(\sqrt{a}+\sqrt{b})^2.
\end{align}

For $u,v>0$, we define the \emph{tacnode kernel} $K^{tacnode}(u,v; s^*,\tau^*,\alpha)$ as
\begin{equation}\label{tacnode:kernel}
    K^{tacnode}(u,v; s^*,\tau^*,\alpha)
    = \frac{1}{2\pi i(u-v)} \begin{pmatrix} -1 & 0 & 1 & 0 \end{pmatrix}
     \what M_+^{-1}(v)  \what M_+(u) \begin{pmatrix} 1 & 0 & 1 & 0 \end{pmatrix}^T,
\end{equation}
where the superscript ${}^T$ stands for the transpose. Here $\what M(\zeta)$ is defined in \eqref{Mhat} where the RH matrix
$M(\zeta)$ is defined with respect to the parameters $r_1=r_2=1$, $s=s^*$,
$\tau=\tau^*$ and $\nu:=\alpha+1/2.$

The squared Bessel paths at each fixed time $t\in (0,1)$ are a determinantal
point process with correlation kernel $K_n(x,y)$, see
Section~\ref{subsection:aboutproof} below. Now we state our main result.

\begin{theorem} \label{theorem:kernelpsi}
(Asymptotics of the correlation kernel). Consider $n$ non-intersecting squared
Bessel paths on the time interval $[0,1]$ with transition probability density
\eqref{probdensity:alpha}--\eqref{time:rescaling} with given starting point
$a>0$ and given endpoint $b
> 0$. Assume that
\eqref{doublescaling:ab}--\eqref{doublescaling:T} hold. Then the correlation
kernel $K_n$ for the positions of the paths at time $t$ has the following
scaling limit as $n \to \infty$ with $n$ even,
\begin{align}\label{kernel at tacnode}
    \lim_{n \to \infty} \frac{1}{\kappa^2 n^{4/3}}
    K_n \left( \frac{u}{\kappa^2 n^{4/3}}, \frac{v}{\kappa^2 n^{4/3}}\right)
        = u^{\alpha/2}v^{-\alpha/2} K^{tacnode}(u,v; s^*,\tau^*,\alpha),
\end{align}
for any fixed $u,v>0$, where $s^*$ and $\tau^*$ are given in
\eqref{doublescaling:s}--\eqref{doublescaling:tau} and with
\begin{equation}\label{def:kappa}\kappa =
2(\sqrt{a}+\sqrt{b}).\end{equation}
\end{theorem}


We expect the conclusion of Theorem~\ref{theorem:kernelpsi} to remain valid if $n\to\infty$ with $n$ odd.

Note that the factor $u^{\alpha/2}v^{-\alpha/2}$ in \eqref{kernel at tacnode}
has no influence on the correlation functions of the determinantal point
process induced by $K^{tacnode}$. Moreover, when comparing $K^{tacnode}$ to the
kernel in \cite[Eq.~(2.48)]{DKZ}, note that the latter has a typo: it has \lq
$M_+^{-1}(u) M_+(v)$\rq\ instead of \lq $M_+^{-1}(v) M_+(u)$\rq.

\begin{remark} (Varying endpoints).
Fix $a^*,b^*>0$ such that $a^* b^* = 1/4$ and assume that the endpoints $a,b$ vary with $n$ as
\begin{align}
\label{doublescaling:a} a &= a^*(1+2L_1 n^{-2/3}),\\
\label{doublescaling:b} b &= b^*(1+2L_2 n^{-2/3}),
\end{align}
for certain real constants $L_1,L_2$. Assume again that the time $t$ and
temperature $T$ vary with $n$ as in
\eqref{doublescaling:t}--\eqref{doublescaling:T}, but with $a,b$ replaced by
$a^*,b^*$. Then the conclusion of Theorem~\ref{theorem:kernelpsi} remains
valid, with now \eqref{doublescaling:s} replaced by
\begin{align}
\label{doublescaling:sbis} s^* &= \frac{K^2(\sqrt{a}+\sqrt{b})^4
-L+L_1+L_2}{2}.
\end{align}
\end{remark}

\begin{remark}\label{Bessel:sqrt} (Bessel process).
Theorem~\ref{theorem:kernelpsi} was formulated for the \emph{squared} Bessel
process \eqref{probdensity:alpha}--\eqref{time:rescaling}. By taking square
roots, one obtains a similar statement for the (ordinary) \emph{Bessel
process}. The correlation kernel $\wtil K_n(x,y)$ for the positions of $n$
non-intersecting Bessel processes on $[0,1]$, starting at $\sqrt{a}$ and ending
at $\sqrt{b}$, is expressed in terms of the above $K_n$ by $$\wtil
K_n(x,y)=2\sqrt{xy}K_n(x^2,y^2),$$ see \cite[Remark~2.10]{KMW}. The analogue of
\eqref{kernel at tacnode} in this case is
\begin{align}
    \lim_{n \to \infty} \frac{1}{\kappa n^{2/3}}
\wtil K_n \left( \frac{u}{\kappa n^{2/3}}, \frac{v}{\kappa n^{2/3}}\right)=
2u^{\alpha}v^{-\alpha}\sqrt{uv}K^{tacnode}(u^2,v^2; s^*,\tau^*,\alpha).
\end{align}
\end{remark}

\begin{remark} (Chiral 2-matrix model). The RH problem~\ref{rhp:modelM} for $M(\zeta)$ was recently used to study a critical phenomenon in the chiral 2-matrix model \cite{DGZ}. This leads to a critical correlation kernel which is essentially \emph{different} from the kernel $K^{tacnode}$ in \eqref{tacnode:kernel} 
and which is a hard edge analogue of the Duits-Geudens kernel~\cite{DG}.
\end{remark}

\subsection{About the proof}\label{subsection:aboutproof}

The positions of the non-intersecting squared Bessel paths at each particular
time $t\in (0,1)$ constitute a determinantal point process. It is a multiple
orthogonal polynomial (MOP) ensemble with respect to the weight functions
\cite{DKRZ}
\begin{align}\begin{array}{ll}\label{w11}
w_{1,1}(x)=x^{\frac{\alpha}{2}}e^{-\frac{nx}{Tt}}I_{\alpha}\left(\frac{2n\sqrt{ax}}{Tt}\right),
&  w_{1,2}(x)=x^{\frac{\alpha+1}{2}}e^{-\frac{nx}{Tt}}
I_{\alpha+1}\left(\frac{2n\sqrt{ax}}{Tt}\right),
\\ 
w_{2,1}(x)=x^{-\frac{\alpha}{2}}e^{-\frac{nx}{T(1-t)}}
I_{\alpha}\left(\frac{2n\sqrt{bx}}{T(1-t)}\right), &
w_{2,2}(x)=x^{-\frac{\alpha-1}{2}}e^{-\frac{nx}{T(1-t)}}I_{\alpha-1}
\left(\frac{2n\sqrt{bx}}{T(1-t)}\right).
\end{array}\end{align}
Here $T>0$ is the temperature.

This MOP ensemble is related to the following RH problem \cite{DKRZ}.

\begin{rhp} \label{rhp:Y}
We look for a $4 \times 4$ matrix valued function $Y$ satisfying
\begin{enumerate}
  \item[(1)] $Y$ is analytic in $\cee \setminus \er_{+}$.

  \item[(2)] For $x >0 $, $Y$ possesses continuous boundary values
  $Y_{+}$ (from the upper half plane) and $Y_{-}$ (from the lower half
  plane), which satisfy
  \begin{equation}\label{jump:Y}
  Y_+(x)=Y_-(x)
  \begin{pmatrix}
  I_2 & W(x) \\
  0 & I_2
  \end{pmatrix},
  \end{equation}
  where $I_k$ denotes the  identity matrix of size $k$ and $W(x)$ is the rank-one matrix
  \begin{align*}
  W(x) =\begin{pmatrix} w_{1,1}(x)& w_{1,2}(x) \end{pmatrix}^T \begin{pmatrix} w_{2,1}(x)&
  w_{2,2}(x)\end{pmatrix}.
  \end{align*}

  \item[(3)] As $z\to \infty$, $z\in\cee\setminus \er_{+}$, we have
  \begin{equation*}
  Y(z)=(I+O(1/z)) ~\diag(z^{n_1}, z^{n_2}, z^{-n_1},
  z^{-n_2}),
  \end{equation*}
  where $n_1:=\lceil n/2\rceil$ and $n_2:=n-n_1$.

  \item[(4)] $Y(z)$ has the following behavior near the origin:
  \begin{equation*}
  Y(z)\diag(1,1,h^{-1}(z),h^{-1}(z))=O(1),\quad
  Y^{-T}(z)\diag(h^{-1}(z),1,1,1)=O(1),
  \end{equation*}
  as $z\to 0$, $z\in\cee\setminus\er_+$, where the $O$-symbol is defined
  entrywise, where
  the superscript ${}^{-T}$ denotes the inverse transpose and with
  \begin{equation}\label{def:h}
  h(z)=\left\{
         \begin{array}{ll}
           |z|^{\alpha}, & \quad  \hbox{if $-1<\alpha<0$,} \\
           \log|z|, & \quad\hbox{if $\alpha=0$,} \\
           1, & \quad \hbox{if $\alpha>0$}.
         \end{array}
       \right.
  \end{equation}
\end{enumerate}
\end{rhp}

There is a unique solution $Y$ to the above RH problem. The matrix $Y$ is constructed using
multiple orthogonal polynomials of mixed type with respect to the modified Bessel
weights \eqref{w11}; see \cite{DK,DKRZ}. Due to the singularity of the weight matrix near the origin,
we need the condition $(4)$ to ensure the uniqueness of the solution.

The correlation kernel admits the following representation in terms of $Y$
\cite{DK,DKRZ}:
\begin{equation}\label{kernel:Y:0}
K_{n}(x,y)=\frac{1}{2\pi i(x-y)}\begin{pmatrix}0 &0 & w_{2,1}(y)&
w_{2,2}(y)\end{pmatrix} Y_{+}^{-1}(y)Y_{+}(x)
\begin{pmatrix}
w_{1,1}(x) & w_{1,2}(x) & 0 & 0
\end{pmatrix}^T.
\end{equation}

%

\subsection*{Outline of the paper} The remainder of this paper is organized as
follows. In Sections~\ref{section:asyM} and~\ref{section:asyM:minusinfty} we
calculate the large~$s$ asymptotics of the model RH problem for $M$.
Section~\ref{section:proofTheoremP2} proves
Theorem~\ref{theorem:Painleve2modelrhp} on the relation of $M(\zeta)$ with the
inhomogeneous Painlev\'e~II equation. Section~\ref{section:vanishing} proves
Theorem~\ref{theorem:solvability} on the solvability of the model RH problem.
In the final Section~\ref{section:steepest:DKRZ} we analyze the
non-intersecting squared Bessel paths near the hard-edge tacnode and we prove
Theorem~\ref{theorem:kernelpsi}.


\section{Asymptotics of $M(\zeta)$ for $s\to+\infty$}
\label{section:asyM}

In this and the next section we will analyze the model RH problem
\ref{rhp:modelM} for $M(\zeta)$ if $r_1=r_2=1$ and $\tau=0$ in the limit $s \to
\infty$. We will prove the solvability of the RH problem for $s\in\er$ with
$|s|$ sufficiently large, and also establish the large $s$ asymptotics for the
quantities $c$ and $d$ in \eqref{M1:explicit}. More precisely, we will prove
the following proposition.

\begin{proposition}\label{prop:large s solvability}
    Let $\nu>-1/2$ be fixed and suppose that $r_1=r_2=1$ and $\tau=0$.
    Then for $s \in \er$ with $|s|$ large enough, the RH problem for
    $M( \cdot)$ is uniquely solvable.

    Moreover, the numbers $d = - i \left(M_1\right)_{1,4}$ and $c = -i \left(M_1\right)_{1,3}$ in \eqref{M1:explicit}
    have the large $s$ asymptotics
    \begin{eqnarray} \label{eq:asy of d}
    d &=& \frac{\nu}{4s}+O(s^{-5/2}), \qquad s\to +\infty,
       \\ \label{eq:asy of c} c &=&
        s^2 + O(s^{-1}),\qquad s\to +\infty,
        \\ \label{eq:asy of d:bis}
     d &=& \sqrt{-s}+O(s^{-1}),\qquad s\to -\infty.
\end{eqnarray}

\end{proposition}

Proposition~\ref{prop:large s solvability} will be needed in the proofs of
Theorems \ref{theorem:solvability} and \ref{theorem:Painleve2modelrhp}. Note
that in \cite{DKZ} we had a stronger variant of \eqref{eq:asy of d} with
$\nu=0$. This stronger variant allowed us to conclude directly that $q$ in
\eqref{d:Painleve2} is the \emph{Hastings-McLeod solution} to Painlev\'e~II.
In contrast, the condition \eqref{eq:asy of d} is insufficient to characterize
the Hastings-McLeod solution, for any parameter $\nu>-1/2$, which is why we
also need the asymptotics \eqref{eq:asy of d:bis} for $s\to -\infty$. Actually,
it \emph{is} possible to characterize the Hastings-McLeod solution from its
behavior for $s\to +\infty$ alone (without the need for $s\to -\infty$
asymptotics) \cite[Sec.~11.7]{FIKN}; but the latter requires an extremely
detailed asymptotic analysis that is hard to perform in our setting.

The proof of Proposition~\ref{prop:large s solvability} is based on the
Deift-Zhou steepest descent analysis of the RH problem for $M(\zeta)$. In this
section we will perform the analysis for $s\to +\infty$, in the next section we
will consider the case where $s\to -\infty$.


We now turn to the analysis for $s\to +\infty$. We will apply a series of
transformations $M \mapsto A \mapsto B \mapsto C \mapsto D$, so that the
matrix-valued function $D$ tends uniformly to the identity matrix as $s \to
+\infty$. The transformations will be very similar to \cite[Sec.~3]{DKZ} and
therefore we only give a brief description. The main difference with \cite{DKZ}
is in the construction of a local parametrix at the origin with the help of
modified Bessel functions; see Section~\ref{subsection:localPII}.

\subsection{First transformation: $M \mapsto A$}
\label{subsection:MtoA}

The first three transformations will be almost identical as in
\cite[Sec.~3]{DKZ}. The first transformation $M\mapsto A$ is a rescaling of the
RH problem for $M$:
\begin{equation}\label{M to A}
A(\zeta)=\diag(s^{1/4},s^{1/4},s^{-1/4},s^{-1/4})M(s\zeta),\qquad \zeta\in\cee
\setminus\bigcup_{k=0}^9 \Gamma_k.
\end{equation}
The matrix $A$ satisfies a RH problem with exactly the same jumps and behavior
near the
origin as $M(\zeta)$. The 
large $\zeta$ asymptotics 
in \eqref{M:asymptotics} changes to
\begin{multline}
\label{A:asymptotics} A(\zeta) =
\left(I+\frac{A_1}{\zeta}+O\left(\frac{1}{\zeta^2}\right)\right)
\diag((-\zeta)^{-1/4},\zeta^{-1/4},(-\zeta)^{1/4},\zeta^{1/4})
\\ \times \Aa
\diag\left(e^{-\lam\til\theta_1(\zeta)},e^{-\lam\til\theta_2(\zeta)},e^{\lam\til\theta_1(\zeta)},e^{\lam\til\theta_2(\zeta)}\right)
\end{multline}
with $\Aa$ as in \eqref{mixing:matrix} and with
\begin{align}\label{tilde theta}
\lambda = s^{3/2},\qquad  \tilde
\theta_1(\zeta)=\frac{2}{3}(-\zeta)^{3/2}+2(-\zeta)^{1/2},\qquad\quad \tilde
\theta_2(\zeta)=\frac{2}{3}\zeta^{3/2}+2\zeta^{1/2}.
  \end{align}

\subsection{Second transformation: $A \mapsto B$}
\label{subsection:AtoB}

\begin{figure}[t]
\vspace{-15mm}
\begin{center}
   \setlength{\unitlength}{1truemm}
   \begin{picture}(100,70)(-5,2)
       \put(40,40){\line(1,0){50}}
       \put(40,40){\line(-1,0){50}}
       \put(60,40){\line(2,1){30}}
       \put(60,40){\line(2,-1){30}}
       \put(20,40){\line(-2,1){30}}
       \put(20,40){\line(-2,-1){30}}
       \put(40,40){\line(2,3){10}}
       \put(40,40){\line(2,-3){10}}
       \put(40,40){\line(-2,3){10}}
       \put(40,40){\line(-2,-3){10}}
       \put(40,40){\thicklines\circle*{1}}
       \put(20,40){\thicklines\circle*{1}}
       \put(60,40){\thicklines\circle*{1}}
       \put(39.3,36){$0$}
       \put(15.5,36){$-2$}
       \put(59.3,36){$2$}
       \put(75,40){\thicklines\vector(1,0){.0001}}
       \put(5,40){\thicklines\vector(1,0){.0001}}
       \put(52,40){\thicklines\vector(1,0){.0001}}
       \put(30,40){\thicklines\vector(1,0){.0001}}
       \put(80,50){\thicklines\vector(2,1){.0001}}
       \put(80,30){\thicklines\vector(2,-1){.0001}}
       \put(0,50){\thicklines\vector(2,-1){.0001}}
       \put(0,30){\thicklines\vector(2,1){.0001}}
       \put(45.2,48){\thicklines\vector(2,3){.0001}}
       \put(45.3,32){\thicklines\vector(2,-3){.0001}}
       \put(34.8,48){\thicklines\vector(-2,3){.0001}}
       \put(34.7,32){\thicklines\vector(-2,-3){.0001}}

       \put(80,52.5){$\til\Gamma_1$}
       \put(50,52){$\Gamma_2$}
       \put(26,52){$\Gamma_3$}
       \put(0.5,50){$\til \Gamma_4$}
       \put(0,27){$\til \Gamma_6$}
       \put(26,25){$\Gamma_7$}
       \put(50,25){$\Gamma_8$}
       \put(80,25){$\til\Gamma_{9}$}
  \end{picture}
  \vspace{-22mm}
   \caption{Jump contour $\Sigma_B$ for the RH problem for $B$.
   Note the reversion of the orientation of the rays $\til\Gamma_4$, $\til\Gamma_6$ and $(-\infty,0)$.}
   \label{fig:ContourB}
\end{center}
\end{figure}
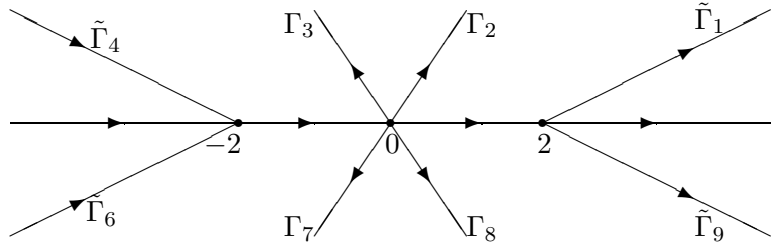

In the second transformation we apply contour deformations.
The four rays $\Gamma_k$, $k=1,4,6,9$, emanating from the
origin are replaced by their parallel lines emanating from some special points
on the real line. More precisely, we replace $\Gamma_1$ and $\Gamma_9$ by their
parallel rays $\tilde\Gamma_1$ and $\tilde\Gamma_9$ emanating from the
point~$2$,
and replace $\Gamma_4$ and $\Gamma_6$ by their parallel rays $\tilde\Gamma_4$
and $\tilde\Gamma_6$ emanating from the point~$-2$.
See Figure~\ref{fig:ContourB}.

Denote by $E_{j,k}$ the $4 \times 4$ elementary matrix with entry $1$ at the
$(j,k)$th position and all other entries equal to zero. We define
\begin{equation}
     B(\zeta)=\left\{
                \begin{array}{ll}
                  A(\zeta)(I+E_{3,1}), & \text{for $\zeta$ between $\Gamma_1$, $[0,2]$ and $\til\Gamma_1$}, \\
                  A(\zeta)(I-E_{3,1}), & \text{for $\zeta$  between $\Gamma_9$, $[0,2]$ and $\til\Gamma_9$}, \\
                  A(\zeta)(I+E_{4,2}), & \text{for $\zeta$ between $\Gamma_4$, $[-2,0]$ and $\til\Gamma_4$}, \\
                  A(\zeta)(I-E_{4,2}), & \text{for $\zeta$  between $\Gamma_6$, $[-2,0]$ and $\til\Gamma_6$}, \\
                  A(\zeta), & \text{elsewhere}.
                \end{array}.
              \right.
\end{equation}

The large $\zeta$ asymptotics of $B(\zeta)$ are the same as for $A(\zeta)$, see
\eqref{A:asymptotics}. The new jumps of $B$ on the interval $(-2,2)$ are given
by
\begin{align*}
  B_{+}(\zeta)=B_-(\zeta)(I+E_{1,3}),&\qquad \zeta\in (0,2),\\
  B_{+}(\zeta)=B_-(\zeta)(I+E_{2,4}),&\qquad \zeta\in (-2,0).
  \end{align*}
The jumps of $B$ on the remaining contours are the same as those for $A$,
provided that $\Gamma_j$ is replaced by $\til\Gamma_j$, $j\in\{1,4,6,9\}$, and
that the change of the orientation of $\til \Gamma_4$, $\til\Gamma_6$ and
$(-\infty,0)$ is taken into account. (The reversion of a contour implies that
the jump matrix is replaced by its inverse). The jump contour $\Sigma_B$ for
$B(\zeta)$ is shown in Figure~\ref{fig:ContourB}.

\subsection{Third transformation: $B \mapsto C$}

The third transformation is a normalization of the RH problem at infinity.
Define the \lq $g$-functions\rq
\begin{align}\label{g1}
g_1(\zeta)=\frac{2}{3}(2-\zeta)^{3/2},\qquad & \zeta\in\cee \setminus
[2,\infty),
\\ \label{g2}
g_2(\zeta)=\frac{2}{3} (2+\zeta)^{3/2}, \qquad & \zeta\in\cee \setminus
(-\infty, -2],
\end{align}
with the principal branch of the power $3/2$. We define the transformation
$B\mapsto C$ as
\begin{equation}\label{B to C} C(\zeta)= (I + i \lambda (E_{3,1}-E_{4,2}))
B(\zeta)\diag(e^{\lambda g_1(\zeta)}, e^{\lambda g_2(\zeta)},e^{-\lambda
g_1(\zeta)},e^{-\lambda g_2(\zeta)}),
\end{equation}
where again $\lam:=s^{3/2}$. As in \cite[Sec.~3.4]{DKZ}, one shows that $C$
satisfies the following RH problem.

\begin{rhp} \label{rhp:C} We look for a $4\times 4$ matrix valued function $C$ such that
\begin{enumerate}
  \item[(1)] $C(\zeta)$ is analytic for $\zeta\in\cee\setminus
  \Sigma_B$.
  \item[(2)] $C$ has the following jumps on $\Sigma_B$:
  \[ C_+(\zeta) = C_-(\zeta) J_C(\zeta), \]
   where $J_C$ is given by
  \begin{align}
  J_C(\zeta)& =
  E_{2,2}+E_{4,4}+E_{1,3}-E_{3,1},
  \qquad \text{for $\zeta \in (2,\infty)$}, \nonumber
  \\ \nonumber
  J_C(\zeta) &=
  (I+e^{2\lambda g_1(\zeta)}E_{3,1}), \qquad \text{for $\zeta\in \tilde \Gamma_1\cup
  \tilde \Gamma_9$},
  \\   \nonumber
  J_C(\zeta) &=
  (I+e^{-2\lambda g_1(\zeta)}E_{1,3}),
  \qquad \text{for $\zeta\in(0,2)$},\\
  \label{jumps:C:G2}
  J_C(\zeta) &= (I + e^{\lambda (g_1 - g_2)(\zeta)}e^{\nu\pi i} (E_{3,4}-E_{2,1})),
  \qquad \text{for $\zeta\in\Gamma_2$},\\
  \label{jumps:C:G8}
  J_C(\zeta) &= (I - e^{\lambda (g_1 - g_2)(\zeta)}e^{-\nu\pi i} (E_{3,4}-E_{2,1})),
  \qquad \text{for $\zeta\in\Gamma_8$},\\
  J_C(\zeta) &= (I - e^{-\lambda (g_1 - g_2)(\zeta)}e^{-\nu\pi i}(E_{4,3}-E_{1,2})),
  \qquad \text{for $\zeta\in\Gamma_3$},
  \label{jumps:C:G3} \\
  J_C(\zeta) &= (I + e^{-\lambda (g_1 - g_2)(\zeta)}e^{\nu\pi i}(E_{4,3}-E_{1,2})),
  \qquad \text{for $\zeta\in\Gamma_7$},
  \label{jumps:C:G7}
  \\
  \nonumber
  J_C(\zeta) &= (I+e^{-2\lambda g_2(\zeta)}E_{2,4}),
  \qquad \text{for $\zeta\in(-2,0)$},
  \\
  \nonumber
  J_C(\zeta) &= (I+e^{2\lambda g_2(\zeta)}E_{4,2}), \qquad \text{for $\zeta\in \til\Gamma_4 \cup \til\Gamma_6$},
  \\
  \nonumber 
  J_C(\zeta) &=
  E_{1,1}+E_{3,3}+E_{2,4}-E_{4,2},
  \qquad \text{for $\zeta \in (-\infty, -2)$}.
  \end{align}

  \item[(3)] As $\zeta\to\infty$, we have
  \begin{align}\label{asy of C}
  C(\zeta)=&\left(I+\frac{C_1}{\zeta}+
  O\left(\frac{1}{\zeta^2}\right)\right)\diag((-\zeta)^{-1/4},\zeta^{-1/4},
  (-\zeta)^{1/4},\zeta^{1/4}) \Aa,
  \end{align}
  where
 \begin{equation}\label{C1}
 C_1 = (I + i \lambda (E_{3,1}-E_{4,2}))
  \left[A_1 (I - i \lambda (E_{3,1}-E_{4,2}))
  + \begin{pmatrix}
  - \frac{\lambda^2}{2} & 0 & -i\lambda  & 0\\
  0 & \frac{\lambda^2}{2} & 0 & -i\lambda \\
  \frac{i\lam^3-2i\lam}{6} & 0 & -\frac{\lambda^2}{2} & 0\\
  0 & \frac{i\lam^3-2i\lam}{6} & 0 & \frac{\lambda^2}{2}
  \end{pmatrix}\right],
 \end{equation}
 with $A_1$ the residue matrix in \eqref{A:asymptotics}.
 \item[(4)] The behavior of $C(\zeta)$ for $\zeta\to 0$ is the same as for
 $M(\zeta)$. More precisely,
 \begin{equation}\begin{array}{ll}
 \label{Czero} C(\zeta) = O(\zeta^{\nu}),\quad C^{-1}(\zeta) = O(\zeta^{\nu}), \qquad
\zeta\to 0,&\qquad \textrm{if $\nu\leq 0$},\\
C(\zeta)\diag(\zeta^{-\nu},\zeta^{\nu},\zeta^{\nu},\zeta^{-\nu})=O(1),\qquad
\zeta\to 0,&\qquad \textrm{if $\nu\geq 0$}, \end{array}\end{equation} where in
the second line we assume that $\zeta\to 0$ to the right of $\Gamma_2$ and
$\Gamma_8$.
\end{enumerate}
\end{rhp}

As in \cite{DKZ} we obtain the following expressions for the numbers $c$ and
$d$ in \eqref{M1:explicit}:
\begin{equation}\label{cctild:D}
c = -i\sqrt{s}(C_1)_{1,3}+s^2, \qquad d = -i\sqrt{s}(C_1)_{1,4}.
\end{equation}

\subsection{Global and local parametrices}

\subsubsection*{Global parametrix}

Away from the  points $2$, $-2$ and $0$, we will approximate $C$ by the \lq
global parametrix\rq\
\begin{equation}\label{C infty}
  C^{(\infty)}(\zeta)
  =\diag\left((2-\zeta)^{-1/4}, (2+\zeta)^{-1/4},
  (2-\zeta)^{1/4}, (2+\zeta)^{1/4}\right)\Aa,
  \end{equation}
with $\Aa$ as in \eqref{mixing:matrix}, and where, as usual, we take the principal
branches of $(2+\zeta)^{1/4}$ and $(2-\zeta)^{1/4}$ with cuts along
$(-\infty,-2]$ and $[2,\infty)$, respectively. See \cite[Sec.~3.6]{DKZ}.

\subsubsection*{Local parametrices at the points $-2$ and $2$}

In small neighborhoods of the points $-2$ and $2$ we construct local
parametrices $C^{(-2)}$, $C^{(2)}$ to the RH problem, respectively. These
parametrices are constructed with the help of Airy functions, exactly as in
\cite[Sec.~3.7]{DKZ}.

\subsubsection*{Local parametrix at the origin} \label{subsection:localPII}

The next step, which has no analogue in \cite[Sec.~3]{DKZ}, is the construction
of a local parametrix $C^{(0)}$ near the origin. The construction will rely on
the one in \cite[Sec.~5.6]{DKRZ}. The local parametrix $C^{(0)}$ will be
defined in the disk $B_\rho$ of radius $\rho$ around the origin in $\cee$,
with $\rho>0$ fixed and sufficiently small. It will solve the following RH
problem.

\begin{rhp}\label{rhp:C0} We look for a $4\times 4$ matrix valued function $C^{(0)}$ such that
\begin{itemize}
\item[(1)] $C^{(0)}(\zeta)$ is analytic for $\zeta\in B_\rho\setminus
\Sigma_B$.
\item[(2)] For $\zeta\in B_\rho\cap
(\Gamma_2\cup\Gamma_3\cup\Gamma_7\cup\Gamma_8)$, we have
$C^{(0)}_+(\zeta)=C^{(0)}_-(\zeta)J_C(\zeta)$ with $J_C$ given in
\eqref{jumps:C:G2}--\eqref{jumps:C:G7}. For $\zeta\in (0,\rho)$ we have
\begin{equation}\label{jumps:Cinfty:mod1} C^{(0)}_+(\zeta) =
C^{(0)}_-(\zeta)\times\left\{
\begin{array}{ll}
I+e^{-2\lambda g_1(\zeta)}E_{1,3},& \textrm{if $\nu-1/2\not\in\zet_{\geq 0}$},\\
I,& \textrm{if $\nu-1/2\in\zet_{\geq 0}$},
\end{array}\right. \end{equation}
and for $\zeta\in(-\rho,0)$ we have \begin{equation}\label{jumps:Cinfty:mod2}
C^{(0)}_+(\zeta) = C^{(0)}_-(\zeta)\times\left\{
\begin{array}{ll}
I+e^{-2\lambda g_2(\zeta)}E_{2,4},& \textrm{if $\nu-1/2\not\in\zet_{\geq 0}$},\\
I,& \textrm{if $\nu-1/2\in\zet_{\geq 0}$}.
\end{array}\right. \end{equation}
\item[(3)] Uniformly on the circle $|\zeta|=\rho$ we have for $\lam\to
+\infty$ that
\begin{equation}\label{matching:C0}
C^{(0)}(\zeta)=C^{(\infty)}(\zeta)(I+O(\lam^{-1})).\end{equation}
 \item[(4)] The behavior of $C^{(0)}(\zeta)$ for $\zeta\to 0$ is the same as for
 $C(\zeta)$, see \eqref{Czero}.
 \item[(5)] $C^{(0)}$ satisfies the symmetry relation
\begin{equation}\label{symmetry:C0}C^{(0)}(-\zeta) =
\begin{pmatrix} J & 0 \\ 0 & -J \end{pmatrix}
C^{(0)}(\zeta)  \begin{pmatrix} J & 0 \\ 0 & -J
\end{pmatrix},\qquad \textrm{with }J:=\begin{pmatrix} 0 & 1 \\ 1 & 0 \end{pmatrix}. \end{equation}
\end{itemize}
\end{rhp}

The idea behind \eqref{jumps:Cinfty:mod1}--\eqref{jumps:Cinfty:mod2} is that we
can get the correct jumps $J_C$ as long as $\nu-1/2\not\in\zet_{\geq 0}$. If
the latter fails then we approximate $J_C$ by the identity matrix for
$\zeta\in(-\rho,\rho)$, but we will see that the resulting approximation error
can be controlled.

\begin{figure}[t]
\vspace{-13mm}
\begin{center}
   \setlength{\unitlength}{1truemm}
   \begin{picture}(100,70)(-5,2)
       \put(40,40){\line(1,0){20}}
       \put(40,40){\line(-1,0){20}}
       \put(40,40){\line(2,3){10}}
       \put(40,40){\line(2,-3){10}}
       \put(40,40){\line(-2,3){10}}
       \put(40,40){\line(-2,-3){10}}
       \put(40,40){\thicklines\circle*{1}}
       \put(39.3,36){$0$}
       \put(52,40){\thicklines\vector(1,0){.0001}}
       \put(30,40){\thicklines\vector(1,0){.0001}}
       \put(45.2,48){\thicklines\vector(2,3){.0001}}
       \put(45.3,32){\thicklines\vector(2,-3){.0001}}
       \put(34.8,48){\thicklines\vector(-2,3){.0001}}
       \put(34.7,32){\thicklines\vector(-2,-3){.0001}}

       \put(60,40.5){$\er$}
       \put(50,55){$\Gamma_2$}
       \put(26,55){$\Gamma_3$}
       \put(26,25){$\Gamma_7$}
       \put(50,25){$\Gamma_8$}
       \put(56,48){$\omega_0$}
       \put(38,56){$\omega_1$}
       \put(21,48){$\omega_2$}
       \put(21,30){$\omega_3$}
       \put(38,23){$\omega_4$}
       \put(56,30){$\omega_5$}
  \end{picture}
  \vspace{-22mm}
   \caption{Jump contour and enclosed sectors $\omega_0,\ldots,\omega_5$ for the local parametrix~$C^{(0)}$.}
   \label{fig:Contour:local}
\end{center}
\end{figure}
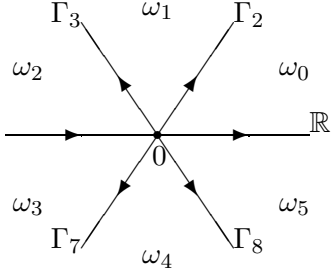

We will see below that the solution matrix $C^{(0)}$ indeed exists. For now, we
stress the symmetry relation \eqref{symmetry:C0} and note that we have the same
type of relation between the jump matrices on opposite rays. Then similarly as
in Proposition~\ref{prop:behaviorMzero} in Section~\ref{section:proofTheoremP2}
we can find the detailed local behavior of $C^{(0)}$ near the origin. The
outcome is that if $\nu-1/2\not\in\zet_{\geq 0}$, then there exist an analytic
matrix valued function $E$ and constant matrices $A_k$ such that, for the
sectors $\omega_k$ in Figure~\ref{fig:Contour:local}, we have
\begin{equation}\label{singularbehzero:1:0}
C^{(0)}(\zeta) =
E(\zeta)\diag(\zeta^{\nu},\zeta^{\nu},\zeta^{-\nu},\zeta^{-\nu})A_k\diag(e^{\lambda
g_1(\zeta)}, e^{\lambda g_2(\zeta)},e^{-\lambda g_1(\zeta)},e^{-\lambda
g_2(\zeta)}), 
\end{equation}
as $\zeta\in\omega_k$, $k=0,\ldots,5$, with in particular
\begin{equation} \label{singularbehzero:2bis:0}
A_0 =
\begin{pmatrix} 0&0&-1 & 2\cos \nu\pi\\ 2\cos\nu\pi&1&e^{-\nu\pi i}&0 \\ 0&1&0&0 \\
0&0&1&0\end{pmatrix}.
\end{equation}
Note that \eqref{singularbehzero:2bis:0} is the same as the matrix $A_1$ in
\eqref{singularbehzero:2bis}, up to an irrelevant diagonal factor $D$. In fact
the calculations in both cases are virtually the same.

On the other hand, if $\nu-1/2\in\zet_{\geq 0}$, then there exist an analytic
matrix valued function $E$ and constant matrices $A_k$ such that
\begin{equation}\label{singularbehzero:1zet:0}
C^{(0)}(\zeta) = E(\zeta)K(\zeta)A_k\diag(e^{\lambda g_1(\zeta)}, e^{\lambda
g_2(\zeta)},e^{-\lambda g_1(\zeta)},e^{-\lambda g_2(\zeta)}),
\end{equation}
as $\zeta\in\omega_k$, $k=0,\ldots,5$, where now $K(\zeta)$ is defined in
\eqref{singularbehzero:1zet:K} and
\begin{equation} \label{singularbehzero:2biszet:0}
A_0 = E_{1,1}+E_{2,4}+E_{3,2}+E_{4,3}.
\end{equation}In all cases, we have for certain constants $*$ that
(see Remark~\ref{remark:E0:pattern})
\begin{equation}\label{E0:pattern:0} E(0) =
\begin{pmatrix} 1&0&-1&0\\1&0&1&0\\0&1&0&1\\0&-1&0&1
\end{pmatrix}
\begin{pmatrix}*&*&0&0 \\
*&*&0&0 \\
0&0&*&* \\
0&0&*&*
\end{pmatrix}.
\end{equation}

\smallskip
Now we turn to the construction of $C^{(0)}$. It will be convenient to consider
the following \lq square root version\rq\ $C^{(0)}\squ$ of $C^{(0)}$:
\begin{multline}\label{squareroot:transfo} C^{(0)}\squ(\zeta) =
\diag(\zeta^{1/4},\zeta^{-1/4},\zeta^{1/4},\zeta^{-1/4})
\frac{1}{\sqrt{2}}\diag\left(\begin{pmatrix}1 & -1 \\ 1 &
1\end{pmatrix},\begin{pmatrix}1 & 1
\\ -1 & 1\end{pmatrix}\right)C^{(0)}(\zeta^{1/2})\\
\times \diag(J,1,1)\diag(-i,1,1,-i),
\end{multline}
where as usual we take the principal branches of all the powers, and with $J$
as in \eqref{symmetry:C0}. Similarly we define $C^{(\infty)}\squ$, by replacing
$C^{(0)}$ by $C^{(\infty)}$ in \eqref{squareroot:transfo}, recall \eqref{C
infty}. We also set
\begin{equation}\label{alpha:nu:0}\alpha:=\nu-1/2,\qquad \Delta^+ := \{z^2\mid z\in\Gamma_2\},\qquad
\Delta^- := \{z^2\mid z\in\Gamma_8\},\end{equation} with $\Delta^+$ and
$\Delta^-$ oriented towards the origin. Then  $C^{(0)}\squ$ solves the
following RH problem.

\begin{rhp} We look for a $4\times 4$ matrix valued function $C^{(0)}\squ$ such
that
\label{rhp:C0sq}
\begin{itemize}
\item[(1)] $C^{(0)}\squ(\zeta)$ is analytic for $\zeta\in
B_{\rho^2}\setminus\left(\er\cup\Delta^+\cup\Delta^-\right)$. \item[(2)]
$C^{(0)}\squ$ has the following jumps on
$B_{\rho^2}\cap\left(\er\cup\Delta^+\cup\Delta^-\right)$:
\begin{equation}\label{jumps:C0sq} (C^{(0)}\squ)_+(\zeta) = (C^{(0)}\squ)_-(\zeta)
J\squ(\zeta),
\end{equation}
   where $J\squ$ is given by
   \begin{align}
  \nonumber
  J\squ(\zeta) &=
\left\{
\begin{array}{ll}
I+e^{-2\lambda g_1(\zeta^{1/2})}E_{2,3},& \textrm{if $\alpha\not\in\zet_{\geq 0}$},\\
I,& \textrm{if $\alpha\in\zet_{\geq 0}$},
\end{array}
\right.
\qquad \text{for $\zeta\in(0,\rho^2)$},\\
  \nonumber
  J\squ(\zeta) &= (I - e^{\lambda (g_1 - g_2)(\zeta^{1/2})}e^{\alpha\pi i} (E_{3,4}+E_{1,2})),
  \qquad \text{for $\zeta\in\Delta^+\cap B_{\rho^2}$},\\
  \nonumber
  J\squ(\zeta) &= (I - e^{\lambda (g_1 - g_2)(\zeta^{1/2})}e^{-\alpha\pi i} (E_{3,4}+E_{1,2})),
  \qquad \text{for $\zeta\in\Delta^-\cap B_{\rho^2}$},\\ \label{jumps:C0sq:Rmin}
  J\squ(\zeta) &= \diag\left(\begin{pmatrix}0 & 1 \\ -1 & 0\end{pmatrix},\begin{pmatrix}0 &
  1
\\ -1 & 0\end{pmatrix}\right), \qquad \text{for $\zeta\in
  (-\rho^2,0)$}.
  \end{align}
   \item[(3)] $C^{(0)}\squ(\zeta)$ behaves near the origin $\zeta\to 0$ as follows:
\begin{align}\label{squareroot:zero1} C^{(0)}\squ(\zeta) = O(\zeta^{\alpha/2}),
\quad (C^{(0)}\squ)^{-1}(\zeta) = O(\zeta^{\alpha/2}),\qquad
\alpha<0,\\
C^{(0)}\squ(\zeta) = O(\log(|\zeta|)),\quad (C^{(0)}\squ)^{-1}(\zeta) =
O(\log(|\zeta|)),\qquad \alpha=0,\\
\label{squareroot:zero2}
C^{(0)}\squ(\zeta)\diag(\zeta^{\alpha/2},\zeta^{-\alpha/2},\zeta^{\alpha/2},\zeta^{-\alpha/2})=O(1),\qquad
\alpha>0,
\end{align}
where in \eqref{squareroot:zero2} we assume that $\zeta\to 0$ in the region to
the right of $\Delta^+$ and $\Delta^-$.
\item[(4)] Uniformly on the circle $|\zeta|=\rho^2$ 
we have as $\lam \to
+\infty$ that
 \begin{align}\label{matching condition Q}
C^{(0)}\squ(\zeta)=C\squ^{(\infty)}(\zeta)\left(I+O\!\left(\lam^{-1}\right)\right).
  \end{align}
  \end{itemize}
  \end{rhp}

\begin{proof}
The jumps of $C^{(0)}\squ$ on $\Delta^+$, $\Delta^-$ and $\er_+$  easily follow
from $\alpha=\nu-1/2$, \eqref{squareroot:transfo} and the jumps of $C^{(0)}$.
Now we check the jump on $\er_-$. With the negative real axis oriented from
left to right as usual, we have
\begin{align*}
\left[\diag(\zeta^{1/4},\zeta^{-1/4},\zeta^{1/4},\zeta^{-1/4})
\diag\left(\begin{pmatrix}1 & -1 \\ 1 & 1\end{pmatrix},\begin{pmatrix}1 & 1
\\ -1 & 1\end{pmatrix}\right) C^{(0)}(\zeta^{1/2}) \right]_+ \\ =
\left[\diag(\zeta^{1/4},\zeta^{-1/4},\zeta^{1/4},\zeta^{-1/4})
\diag\left(\begin{pmatrix}1 & -1 \\ 1 & 1\end{pmatrix},\begin{pmatrix}1 & 1
\\ -1 & 1\end{pmatrix}\right)\right]_-
i\diag\left(-J,J\right)[C^{(0)}(\zeta^{1/2})]_+.
\end{align*}
On account of the symmetry \eqref{symmetry:C0}, we obtain
$$ =
\left[\diag(\zeta^{1/4},\zeta^{-1/4},\zeta^{1/4},\zeta^{-1/4})
\diag\left(\begin{pmatrix}1 & -1 \\ 1 & 1\end{pmatrix},\begin{pmatrix}1 & 1
\\ -1 & 1\end{pmatrix}\right)C^{(0)}(\zeta^{1/2})\right]_-
i\diag\left(-J,J\right).
$$
By using this in \eqref{squareroot:transfo}, it is then straightforward to
obtain the jump \eqref{jumps:C0sq}--\eqref{jumps:C0sq:Rmin}.

The behavior of $C^{(0)}\squ(\zeta)$ as $\zeta\to 0$ follows from
$\alpha=\nu-1/2$, \eqref{squareroot:transfo} and a careful inspection of
\eqref{singularbehzero:1:0}--\eqref{E0:pattern:0}.
\end{proof}


The properties of $C^{(0)}\squ$ and $C^{(\infty)}\squ$ are exactly like those
for \lq $T$\rq\ and \lq $N_{\alpha}$\rq\ in \cite[Sec.~5.6]{DKRZ},
respectively, where we are dealing with \lq Case~I\rq. We use the construction
in that paper to build the matrix $C^{(0)}\squ$ solving the RH
problem~\ref{rhp:C0sq}, see also \cite[Sec.~5.6]{DGZ}.

Note that for $\alpha\in\zet_{\geq 0}$ the RH problem for $C^{(0)}\squ$
basically decouples into two RH problems of size $2\times 2$ which can be
solved using modified Bessel functions. On the other hand, if
$\alpha\not\in\zet_{\geq 0}$ then we have a genuine $4\times 4$ RH problem and
its solution requires some nontrivial modifications \cite[Sec.~5.6]{DKRZ}. In
the latter paper we use these modifications only if $\alpha<0$; but the same
construction works as long as $\sin(\pi\alpha)\neq 0$, i.e.,
$\alpha\not\in\zet$.

\smallskip
Finally, we lift $C^{(0)}\squ$ back to the original setting by inverting
\eqref{squareroot:transfo}: we put
\begin{multline}\label{squareroot:transfo:inverse} C^{(0)}(\zeta)
= \frac{1}{\sqrt{2}} \diag\left(\begin{pmatrix}1 & 1 \\ -1 &
1\end{pmatrix},\begin{pmatrix}1 & -1
\\ 1 & 1\end{pmatrix}\right)\diag(\zeta^{-1/2},\zeta^{1/2},\zeta^{-1/2},\zeta^{1/2})
C^{(0)}\squ(\zeta^2)\\
\times \diag(i,1,1,i)\diag(J,1,1),\end{multline} if $\Re\zeta>0$, and
$$ C^{(0)}(\zeta) = \diag(J,-J)C^{(0)}(-\zeta)\diag(J,-J),\qquad \textrm{if }\Re\zeta<0,
$$
recall \eqref{symmetry:C0}.
We claim that this matrix satisfies the RH problem~\ref{rhp:C0} for $C^{(0)}$.
It is easily seen that it has the right jumps. It also satisfies
\eqref{matching:C0} and \eqref{symmetry:C0}. It remains to check the behavior
of $C^{(0)}$ near the origin. Recall that $C^{(0)}(\zeta)$ should have the same
behavior for $\zeta\to 0$ as $C(\zeta)$ in \eqref{Czero}. Let us check this if
$\nu<0$. It is clear from \eqref{squareroot:transfo:inverse} that $C^{(0)}(\zeta)=O(\zeta^{\nu-1})$. Then we can again
derive the formulas \eqref{singularbehzero:1:0}--\eqref{E0:pattern:0}, with now
each term $\zeta^{\nu}$ or $\zeta^{-\nu}$ replaced by $\zeta^{\nu-1}$ or
$\zeta^{-\nu-1}$ respectively. Using these formulas together with
\eqref{squareroot:transfo} and \eqref{squareroot:zero1} we see that all the
entries of $E(\zeta)$ have a zero at $\zeta=0$ and so each term $\zeta^{\nu-1}$
or $\zeta^{-\nu-1}$ can be replaced by $\zeta^{\nu}$ or $\zeta^{-\nu}$
respectively, which yields the desired conclusion. Similar arguments apply if
$\nu\geq 0$.

\subsection{Fourth transformation: $C \mapsto D$}
We define the fourth transformation
\begin{equation}\label{def:R:Lun}
D(\zeta)=\left\{
       \begin{array}{ll}
         C(\zeta)(C^{(-2)})^{-1}(\zeta), & \text{if  $|\zeta+2|<\rho$,} \\
         C(\zeta)(C^{(2)})^{-1}(\zeta), & \text{if  $|\zeta-2|<\rho$,} \\
         C(\zeta)(C^{(0)})^{-1}(\zeta), & \text{if  $|\zeta|<\rho$,} \\
         C(\zeta)(C^{(\infty)})^{-1}(\zeta), & \text{elsewhere.}
       \end{array}
     \right.
\end{equation}

Then $D$ satisfies the following RH problem.
\begin{rhp} We look for a $4\times 4$ matrix valued function $D$ such that
\begin{enumerate}
\item[(1)] $D$ is analytic in $\cee\setminus \Sigma_D$, where $\Sigma_D$ is
shown in Figure \ref{fig:ContourD}.

\item[(2)] $D$ has jumps $D_+(\zeta)=D_-(\zeta)J_D(\zeta)$ for
 $\zeta\in \Sigma_D$, where
\begin{equation}\label{jump for D}
J_D=\left\{
      \begin{array}{ll}
        C^{(-2)} \left(C^{(\infty)} \right)^{-1}, & \text{on $|\zeta+2| = \rho$,} \\
        C^{(0)} \left(C^{(\infty)} \right)^{-1}, & \text{on $|\zeta| = \rho$,} \\
        C^{(2)} \left(C^{(\infty)} \right)^{-1}, & \text{on $|\zeta - 2| = \rho$,} \\
        C^{(\infty)} J_C \left(C^{(\infty)}\right)^{-1}, & \text{on the rest of $\Sigma_D$}.
      \end{array}
    \right.
\end{equation}
\item[(3)] As $\zeta \to
\infty$, we have
\begin{equation}\label{eq: D asy}
D(\zeta)=I+\frac{D_1}{\zeta}+O\left(\frac{1}{\zeta^2}\right).
\end{equation}
\item[(4)] $D(\zeta)$ is bounded in a neighborhood of $\zeta=0$.
\end{enumerate}
\end{rhp}

\begin{figure}[t]
\vspace{-16mm}
\begin{center}
   \setlength{\unitlength}{1truemm}
   \begin{picture}(100,70)(-5,2)
       \put(42.5,40){\line(1,0){15}}
       \put(37.5,40){\line(-1,0){15}}
       \put(62.1,41){\line(2,1){28}}
       \put(62.1,39){\line(2,-1){28}}
       \put(17.9,41){\line(-2,1){28}}
       \put(17.9,39){\line(-2,-1){28}}
       \put(41.4,42.1){\line(2,3){8}}
       \put(41.4,37.9){\line(2,-3){8}}
       \put(38.6,42.1){\line(-2,3){8}}
       \put(38.6,37.9){\line(-2,-3){8}}
       \put(40,40){\thicklines\circle*{1}}
       \put(20,40){\thicklines\circle*{1}}
       \put(60,40){\thicklines\circle*{1}}
       \put(39.3,34){$0$}
       \put(15.5,34){$-2$}
       \put(59.3,34){$2$}
       \put(30,40){\thicklines\vector(1,0){.0001}}
       \put(53,40){\thicklines\vector(1,0){.0001}}
       \put(80,50){\thicklines\vector(2,1){.0001}}
       \put(80,30){\thicklines\vector(2,-1){.0001}}
       \put(0,50){\thicklines\vector(2,-1){.0001}}
       \put(0,30){\thicklines\vector(2,1){.0001}}
       \put(45.3,48){\thicklines\vector(2,3){.0001}}
       \put(45.3,32){\thicklines\vector(2,-3){.0001}}
       \put(34.6,48){\thicklines\vector(-2,3){.0001}}
       \put(34.7,32){\thicklines\vector(-2,-3){.0001}}

       \put(21.2,42.3){\thicklines\vector(1,0){.0001}}
       \put(41.2,42.3){\thicklines\vector(1,0){.0001}}
       \put(61.2,42.3){\thicklines\vector(1,0){.0001}}

       \put(20,40){\circle{5}}
       \put(40,40){\circle{5}}
       \put(60,40){\circle{5}}

       \put(80,52.5){$\til\Gamma_1$}
       \put(50,52){$\Gamma_2$}
       \put(26,52){$\Gamma_3$}
       \put(0.5,50){$\til \Gamma_4$}
       \put(0,27){$\til \Gamma_6$}
       \put(26,25){$\Gamma_7$}
       \put(50,25){$\Gamma_8$}
       \put(80,25){$\til\Gamma_{9}$}
  \end{picture}
  \vspace{-22mm}
   \caption{Contour  $\Sigma_D$ for the RH problem for $D$ in the case $\nu-1/2\not\in\zet_{\geq 0}$.
   If $\nu-1/2\in\zet_{\geq 0}$ then there is an additional jump contour on the line segment $(-\rho,\rho)$.}
   \label{fig:ContourD}
\end{center}
\end{figure}
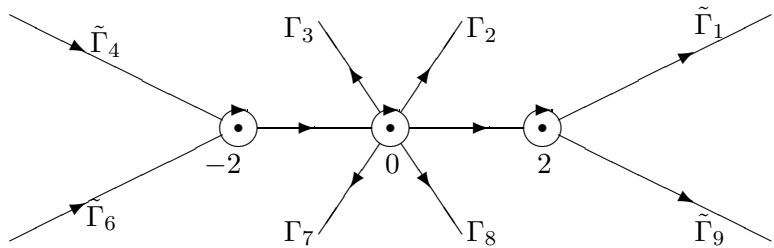

Note that, if $\nu-1/2\not\in\zet_{\geq 0}$, then $D$ has no jumps in the disk
around the origin. From the asymptotics of $C(\zeta)$ and $C^{(0)}(\zeta)$ at
the origin we find that $D(\zeta)=O(1)$ as $\zeta\to 0$ (if $\nu\geq 0$) or
$D(\zeta)=O(\zeta^{2\nu})$ as $\zeta\to 0$ (if $\nu\leq 0$). In both cases we
conclude that the singularity at the origin is removable and so $D(\zeta)$ is
analytic at $\zeta=0$.

On the other hand, if $\nu-1/2\in\zet_{\geq 0}$, then $D$ has a jump on
$(-\rho,\rho)$ and so it is not analytic in the disk around the origin. But
then we conclude as in the previous paragraph that $D(\zeta)=O(1)$ as $\zeta\to
0$ and so $D(\zeta)$ is bounded near the origin.

The jump matrix for $D$ satisfies
\begin{equation}\label{jumpD:small}
J_D(\zeta)=I+O(\lambda^{-1}), \qquad \text{as $\lambda \to +\infty$},
\end{equation}
uniformly for $\zeta$ on the circles $|\zeta|=\rho$, $|\zeta+2| = \rho$ and
$|\zeta-2| = \rho$, and the jumps on the remaining contours of $\Sigma_D$ are
uniformly bounded and exponentially converging to the identity matrix as
$\lam\to +\infty$.
(This holds in particular for the jump matrix on $(-\rho,\rho)$ if
$\nu-1/2\in\zet_{\geq 0}$.) By standard arguments \cite{Dei,DKMVZ1} we then
conclude that
\begin{equation}\label{RLun:small}
D(\zeta)=I+O\left(\frac{1}{\lambda(|\zeta|+1)}\right),
\end{equation}
as $\lambda\to +\infty$, uniformly for $\zeta\in\cee \setminus \Sigma_D$. Moreover, the matrix $D_1$ in \eqref{eq: D asy}
satisfies
\begin{equation}\label{R1Lun:small}
D_1=K\lam^{-1}+O(\lam^{-2}),\quad \lam\to\infty,
\end{equation}
for a certain constant matrix $K$.

It follows from \eqref{cctild:D} and the above constructions that
\begin{equation}\label{cctild:R} c =
-i\sqrt{s}(D_1)_{1,3}+s^2,\qquad d =-i\sqrt{s}(D_1)_{1,4},
\end{equation}
where we used \eqref{def:R:Lun} for large $\zeta$ and the fact that the large
$\zeta$ expansion of $C^{(\infty)}$ in \eqref{C infty} has a diagonal form.
Consequently, using \eqref{R1Lun:small} and $\lam=s^{3/2}$ we get
for a certain constant $\kappa\in\er$,
\begin{equation}\label{cctild:R:bis} c =
s^2+O(s^{-1}),\qquad d =\frac \kappa s +O(s^{-5/2}), \qquad s\to +\infty.
\end{equation}

\subsection{Proof of Proposition~\ref{prop:large s solvability}}

From the above results we obtain the solvability of the RH problem for
$M(\zeta)$ for $s>>0$ sufficiently large. We also obtain \eqref{eq:asy of c}.

Eq.~\eqref{d:Painleve2} shows that $d = 2^{-1/3} q\left(2^{5/3}s\right)$ with
$q=q(s)$ a certain solution to the Painlev\'e~II equation. The asymptotics of
$d=d(s)$ for $s\to+\infty$ in \eqref{cctild:R:bis} transform into similar
asymptotics for $q(s)$. Substituting these asymptotics in the Painlev\'e~II
equation \eqref{def:Painleve2}, we find that the constant $\kappa$ in
\eqref{cctild:R:bis} must necessarily equal $\nu/4$. For example, if $\kappa>\nu/4$
then \eqref{def:Painleve2} would imply that $q''(s)$ is positive and bounded
away from zero for all $s\in\er$ large enough. This would imply that $q(s)\to
\infty$ as $s\to\infty$, which contradicts the fact that $q(s)=O(1/s)$.
Similarly one can rule out the case where $\kappa<\nu/4$. Hence we have
$\kappa=\nu/4$ and we get \eqref{eq:asy of d}.

Finally, the proof of \eqref{eq:asy of d:bis} will be given in the next
section.

\section{Asymptotics of $M(\zeta)$ for $s\to -\infty$}
\label{section:asyM:minusinfty}

In this section we analyze the RH problem for $M(\zeta)$ if $r_1=r_2=1$ and
$\tau=0$ in the limit where $s\to -\infty$, thereby completing the proof of
Proposition~\ref{prop:large s solvability}. We will apply a series of
transformations $M \mapsto \hh A \mapsto \hh B \mapsto \hh C \mapsto \hh D$, so
that the matrix-valued function $\hh D$ uniformly tends to the identity matrix
as $s \to -\infty$. The analysis will be markedly different from
Section~\ref{section:asyM}. This holds in particular for the contour
deformation, the definition of the $g$-functions and the construction of the
global and local parametrices.

\subsection{First transformation: $M \mapsto \hh A$}

The first transformation is again a rescaling of the RH problem for $M$. Define
\begin{equation}\label{M to A:min}
\hh
A(\zeta)=\diag((-s)^{1/4},(-s)^{1/4},(-s)^{-1/4},(-s)^{-1/4})M(-s\zeta),\qquad
\zeta\in\cee \setminus\bigcup_{k=0}^9 \Gamma_k,
\end{equation}
where we assume that $s<0$. Then $\hh A$ satisfies
\begin{rhp} We look for a $4\times 4$ matrix valued function $\hh A$ satisfying
\begin{enumerate}
  \item[(1)] $\hh A(\zeta)$ is analytic for $\zeta\in\cee\setminus
  \bigcup_{k=0}^9 \Gamma_k$.
  \item[(2)] $\hh A$ has the same jump matrix $J_k$ on $\Gamma_k$ as $M$.
  \item[(3)] As $\zeta\to\infty$, we have
  \begin{align}\label{Aasy:min}
  \hh A(\zeta)=&\left(I+\frac{\hh A_1}{\zeta}+
  O\left(\frac{1}{\zeta^2}\right)\right)\diag((-\zeta)^{-1/4},\zeta^{-1/4},
  (-\zeta)^{1/4},\zeta^{1/4}) \nonumber \\
  &\times \Aa\diag(e^{-\hh\lambda \hh\theta_1(\zeta)},e^{-\hh\lambda \hh\theta_2(\zeta)} ,e^{\hh\lambda \hh\theta_1(\zeta)},
  e^{\hh\lambda \hh\theta_2(\zeta)}),
  \end{align}
  with
  \begin{align}\label{tilde theta:min}
  \hh\lambda = (-s)^{3/2},\qquad
  \hh \theta_1(\zeta)=\frac{2}{3}(-\zeta)^{3/2}-2(-\zeta)^{1/2},\qquad \hh \theta_2(\zeta)=\frac{2}{3}\zeta^{3/2}-2\zeta^{1/2}.
  \end{align}
  \item[(4)] The behavior of $\hh A$ near the origin is the same as for $M$.
\end{enumerate}
\end{rhp}

The number $d$ in \eqref{M1:explicit} now satisfies:
\begin{equation}\label{d in A:min}
d =-i\sqrt{-s}(\hh A_1)_{1,4}.
\end{equation}

\subsection{Second transformation: $\hh A \mapsto \hh B$}

\begin{figure}[t]
\vspace{-5mm}
\begin{center}
   \setlength{\unitlength}{1truemm}
   \begin{picture}(100,70)(-5,2)
       \put(40,40){\line(1,0){30}}
       \put(40,40){\line(-1,0){30}}
       \put(40,40){\line(0,1){10}}
       \put(40,40){\line(0,-1){10}}
       \put(40,40){\line(2,1){30}}
       \put(40,40){\line(2,-1){30}}
       \put(40,40){\line(-2,1){30}}
       \put(40,40){\line(-2,-1){30}}
       \put(40,50){\line(2,3){10}}
       \put(40,30){\line(2,-3){10}}
       \put(40,50){\line(-2,3){10}}
       \put(40,30){\line(-2,-3){10}}
       \put(40,40){\thicklines\circle*{1}}
       \put(40,30){\thicklines\circle*{1}}
       \put(40,50){\thicklines\circle*{1}}
       \put(40.1,36){$0$}
       \put(70,41){$\er$}
       \put(41,30){$-2i$}
       \put(41,50){$2i$}
       \put(71,40){\thicklines\vector(1,0){.0001}}
       \put(25,40){\thicklines\vector(1,0){.0001}}
       \put(40,47){\thicklines\vector(0,1){.0001}}
       \put(40,35){\thicklines\vector(0,1){.0001}}
       \put(60,50){\thicklines\vector(2,1){.0001}}
       \put(60,30){\thicklines\vector(2,-1){.0001}}
       \put(20,50){\thicklines\vector(2,-1){.0001}}
       \put(20,30){\thicklines\vector(2,1){.0001}}
       \put(45.2,58){\thicklines\vector(2,3){.0001}}
       \put(45.3,22){\thicklines\vector(2,-3){.0001}}
       \put(34.8,58){\thicklines\vector(-2,3){.0001}}
       \put(34.7,22){\thicklines\vector(-2,-3){.0001}}

       \put(60,52.5){$\Gamma_1$}
       \put(50,62){$\til\Gamma_2$}
       \put(26,62){$\til\Gamma_3$}
       \put(20.5,50){$\Gamma_4$}
       \put(20,27){$\Gamma_6$}
       \put(26,15){$\til\Gamma_7$}
       \put(50,15){$\til\Gamma_8$}
       \put(60,25){$\Gamma_{9}$}
  \end{picture}
  \vspace{-14mm}
   \caption{Contour $\Sigma_{\hh B}$ for the RH problem for $\hh B$.
   Note the reversion of the orientation of some of the rays.}
   \label{fig:ContourB:min}
\end{center}
\end{figure}
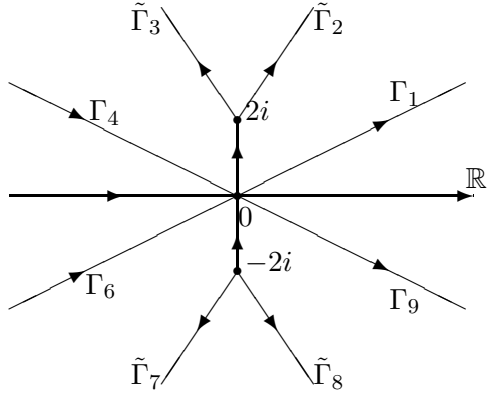

In the second transformation we again apply contour deformations, although in a
different way than in Section~\ref{section:asyM}. The four rays $\Gamma_k$,
$k=2,3,7,8$, emanating from the origin are replaced by their parallel lines
emanating from some special points on the imaginary axis. More precisely, we
replace $\Gamma_2$ and $\Gamma_3$ by their parallel rays $\tilde\Gamma_2$ and
$\tilde\Gamma_3$ emanating from the point~$2i$, and we replace $\Gamma_7$ and
$\Gamma_8$ by their parallel rays $\tilde\Gamma_7$ and $\tilde\Gamma_8$
emanating from the point~$-2i$. See Figure~\ref{fig:ContourB:min}.

We define
\begin{equation}
     \hh B(\zeta)=\left\{
                \begin{array}{ll}
                  \hh A(\zeta)J_2^{-1}, & \text{for $\zeta$ between $\Gamma_2$, $[0,2i]$ and $\til\Gamma_2$}, \\
                  \hh A(\zeta)J_3, & \text{for $\zeta$  between $\Gamma_3$, $[0,2i]$ and $\til\Gamma_3$}, \\
                  \hh A(\zeta)J_7^{-1}, & \text{for $\zeta$ between $\Gamma_7$, $[-2i,0]$ and $\til\Gamma_7$}, \\
                  \hh A(\zeta)J_8, & \text{for $\zeta$  between $\Gamma_8$, $[-2i,0]$ and $\til\Gamma_8$}, \\
                  \hh A(\zeta), & \text{elsewhere},
                \end{array}.
              \right.
\end{equation}

Now $\hh B$ is analytic in $\cee\setminus \Sigma_{\hh B}$, where $\Sigma_{\hh
B}$ is the contour shown in Figure~\ref{fig:ContourB:min}. Note that in this
figure we reverse the orientation on some of the rays. Then $\hh B$ satisfies
the following RH problem.
\begin{rhp} We look for a $4\times 4$ matrix valued function $\hh B$ satisfying
\begin{enumerate}
  \item[(1)] $\hh B(\zeta)$ is analytic for $\zeta\in\cee\setminus
  \Sigma_{\hh B}$.
  \item[(2)] $\hh B$ has the following jumps on $\Sigma_{\hh B}$:
  \begin{equation*}
  \hh B_{+}(\zeta)=\hh B_-(\zeta)J_{\hh B}(\zeta),
  \end{equation*}
  where $J_{\hh B}$ is defined by
  \begin{align*}
  J_{\hh B}(\zeta)&=
  E_{2,2}+E_{4,4}+E_{1,3}-E_{3,1},
\qquad \text{for $\zeta \in (0,+\infty)$},
  \\
  J_{\hh B}(\zeta)&=(I+E_{3,1}),
\qquad \text{for $\zeta\in \Gamma_1\cup
  \Gamma_9$},\\
  J_{\hh B}(\zeta)&=(I+E_{4,2}),
  \qquad \text{for $\zeta\in \Gamma_4
  \cup \Gamma_6$},
  \\
  J_{\hh B}(\zeta)&=
  E_{1,1}+E_{3,3}+E_{2,4}-E_{4,2},
  \qquad \text{for $\zeta \in (-\infty, 0)$}\\
  J_{\hh B}(\zeta)&=
  \diag\left(\begin{pmatrix}
  1 & e^{-\nu\pi i}\\ -e^{\nu\pi i} & 0\end{pmatrix},\begin{pmatrix}
  0 & e^{\nu\pi i}\\ -e^{-\nu\pi i} & 1\end{pmatrix}\right),
  \qquad \text{for $\zeta\in(0,2i)$},\\
  J_{\hh B}(\zeta)&= (I + e^{\nu\pi i}(E_{3,4}-E_{2,1})), \qquad \text{for $\zeta\in\til\Gamma_2$},
  \\
  J_{\hh B}(\zeta)&= (I - e^{-\nu\pi i}(E_{4,3}-E_{1,2})), \qquad \text{for
  $\zeta\in\til\Gamma_3$},\\
  J_{\hh B}(\zeta)&=
\diag\left(\begin{pmatrix}
  1 & e^{\nu\pi i}\\ -e^{-\nu\pi i} & 0\end{pmatrix},\begin{pmatrix}
  0 & e^{-\nu\pi i}\\ -e^{\nu\pi i} & 1\end{pmatrix}\right),
    \qquad \text{for $\zeta\in(-2i,0)$},\\
  J_{\hh B}(\zeta)&= (I + e^{\nu\pi i}(E_{4,3}-E_{1,2})), \qquad \text{for
  $\zeta\in\til\Gamma_7$},\\
  J_{\hh B}(\zeta)&= (I - e^{-\nu\pi i}(E_{3,4}-E_{2,1})), \qquad \text{for
  $\zeta\in\til\Gamma_8$}.
  \end{align*}

  \item[(3)] As $\zeta\to\infty$, we have
  \begin{align}\label{asy of B:min}
  \hh B(\zeta)=&\left(I+\frac{\hh B_1}{\zeta}+
  O\left(\frac{1}{\zeta^2}\right)\right)\diag((-\zeta)^{-1/4},\zeta^{-1/4},
  (-\zeta)^{1/4},\zeta^{1/4}) \nonumber \\
  &\times \Aa\diag(e^{-\hh\lambda \hh\theta_1(\zeta)},e^{-\hh\lam\hh\theta_2(\zeta)} ,e^{\hh\lambda \hh\theta_1(\zeta)},
  e^{\hh\lambda \hh\theta_2(\zeta)}),
  \end{align}
  where $\hh\lambda = (-s)^{3/2}$ and $\hh\theta_{1}(\zeta)$, $\hh\theta_{2}(\zeta)$ are
  given in \eqref{tilde theta:min}.
\end{enumerate}
\end{rhp}

Now \eqref{d in A:min} transforms as follows:
\begin{equation}\label{d in B:min}
d =-i\sqrt{-s}(\hh B_1)_{1,4}.
\end{equation}

\subsection{Third transformation: $\hh B \mapsto \hh C$}

In this transformation we normalize the RH problem at infinity by means of \lq
$g$-functions\rq. The construction of the $g$-functions will be markedly
different from Section~\ref{section:asyM}.

Using the principal branches of the square root, we first define two functions
$\xi_1(\zeta)$ and $\xi_2(\zeta)$ by
\begin{align}\label{xi12a:min}
\xi_1(\zeta)=-\sqrt{-2-\sqrt{4+\zeta^2}},\quad
\xi_2(\zeta)=\sqrt{-2+\sqrt{4+\zeta^2}},\qquad \textrm{if }\Re\zeta>0,
\\ \label{xi12b:min} \xi_1(\zeta)=-\sqrt{-2+\sqrt{4+\zeta^2}},\quad
\xi_2(\zeta)=\sqrt{-2-\sqrt{4+\zeta^2}},\qquad \textrm{if }\Re\zeta<0.
\end{align}
Note that
\begin{align}\begin{array}{ll}
\xi_1(z) = -\sqrt{-z}\left(1+\frac
1z+\frac{1}{2z^2}-\frac{1}{2z^3}+O(z^{-4})\right),& \quad
z\to\infty,\\
\xi_2(z) = \sqrt{z}\left(1-\frac
1z+\frac{1}{2z^2}+\frac{1}{2z^3}+O(z^{-4})\right),& \quad z\to\infty.
\end{array}
\end{align}

The $g$-functions are defined as the following anti-derivatives of the
$\xi$-functions:
\begin{equation}
g_1(\zeta) = \int_{0^+}^{\zeta} \xi_1(x)\ \ud x,\qquad g_2(\zeta) =
\int_{0^-}^{\zeta} \xi_2(x)\ \ud x,
\end{equation}
where $0^+$ is the origin reached from the first quadrant of the complex plane
and $0^-$ from the second quadrant, and where the integration path of $g_1$ (or
$g_2$) is not allowed to cross $\er_+\cup [-2i,2i]$ (or $\er_-\cup [-2i,2i]$
respectively).
We have
\begin{align}\begin{array}{ll}
\label{g1:hat:asy} g_1(z) = \frac 23
(-z)^{3/2}-2(-z)^{1/2}-(-z)^{-1/2}+O((-z)^{-3/2}),& \quad z\to\infty,\\
g_2(z) = \frac 23 z^{3/2}-2z^{1/2}-z^{-1/2}+O(z^{-3/2}),& \quad z\to\infty.
\end{array}
\end{align}
Observe that there is no integration constant in \eqref{g1:hat:asy}. Indeed by
taking the limits $z\to +\infty$ and $z\to -\infty$ along the real line we see
that the integration constant must be simultaneously real and purely imaginary
and therefore it is zero.

We need the following relations:
\begin{itemize}
\item $g_{1,\pm}(z)=g_{2,\mp}(z)$,\quad for $z\in(-2i,2i)$,
\item $g_{1,+}(x)+g_{1,-}(x)=0$,\quad  for $x\in\er_+$,
\item $g_{2,+}(x)+g_{2,-}(x)=0$,\quad for $x\in\er_-$,
\end{itemize}
and we also need the following inequalities for the real parts of the
$g$-functions:
\begin{itemize}
\item $\Re(g_{1}(z))<0$,\quad for $z\in \Gamma_1\cup\Gamma_9\setminus\{0\}$,
\item $\Re(g_{2}(z))<0$,\quad for $z\in \Gamma_4\cup\Gamma_6\setminus\{0\}$,
\item $\Re(g_{1,+}(z)-g_{1,-}(z))<0$,\quad for $z\in (-2i,2i)$,
\item $\Re(g_{2,+}(z)-g_{2,-}(z))>0$,\quad for $z\in (-2i,2i)$,
\item $\Re(g_{1}(z)-g_{2}(z))<0$,\quad for $z\in \til\Gamma_2\cup\til\Gamma_8\setminus\{0\}$,
\item $\Re(g_{1}(z)-g_{2}(z))>0$,\quad for $z\in \til\Gamma_3\cup\til\Gamma_7\setminus\{0\}$.
\end{itemize}
Each of these relations can be straightforwardly checked from the definition of
the functions $g_1,g_2$.


Now we define
\begin{equation}\label{B to C:min}
\hh C(\zeta)= (I - i \hh\lambda (E_{3,1}-E_{4,2})) \hh
B(\zeta)\diag(e^{\hh\lambda g_1(\zeta)}, e^{\hh\lambda g_2(\zeta)},e^{-\hh
\lambda g_1(\zeta)},e^{-\hh\lambda g_2(\zeta)}).
\end{equation}
Then $\hh C$ satisfies the following RH problem.

\begin{rhp} We look for a $4\times 4$ matrix valued function $\hh C$ satisfying
\begin{enumerate}
  \item[(1)] $\hh C(\zeta)$ is analytic for $\zeta\in\cee\setminus
  \Sigma_{\hh B}$.
  \item[(2)] $\hh C$ has the following jumps on $\Sigma_{\hh B}$:
  \begin{equation*}
  \hh C_{+}(\zeta)=\hh C_-(\zeta)J_{\hh C}(\zeta),
  \end{equation*}
  where $J_{\hh C}$ is defined by
  \begin{align*}
  J_{\hh C}(\zeta)&=
  E_{2,2}+E_{4,4}+E_{1,3}-E_{3,1},
  \qquad \text{for $\zeta \in (0,+\infty)$},
  \\
  J_{\hh C}(\zeta)&=(I+E_{3,1}e^{2\hh\lam g_1(\zeta)}),
\qquad \text{for $\zeta\in \Gamma_1\cup
  \Gamma_9$},\\
  J_{\hh C}(\zeta)&=(I+E_{4,2}e^{2\hh\lam g_2(\zeta)}),
  \qquad \text{for $\zeta\in \Gamma_4
  \cup \Gamma_6$},
  \\
  J_{\hh C}(\zeta)&=
  E_{1,1}+E_{3,3}+E_{2,4}-E_{4,2},
  \qquad \text{for $\zeta \in (-\infty, 0)$}\\
  J_{\hh C}(\zeta)&=\begin{pmatrix}
  e^{\hh\lam (g_{1,+}(\zeta)-g_{1,-}(\zeta))} & e^{-\nu\pi i} & 0 & 0\\
  -e^{\nu\pi i} & 0 & 0 & 0\\
  0 & 0 & 0 & e^{\nu\pi i}\\
  0 & 0 & -e^{-\nu\pi i} & e^{-\hh\lam (g_{2,+}(\zeta)-g_{2,-}(\zeta))}
  \end{pmatrix},
  \qquad \text{for $\zeta\in(0,2i)$},\\
  J_{\hh C}(\zeta)&= (I + e^{\nu\pi i}e^{\hh\lam (g_1(\zeta)-g_2(\zeta))}(E_{3,4}-E_{2,1})), \qquad \text{for $\zeta\in\til\Gamma_2$},
  \\
  J_{\hh C}(\zeta)&= (I - e^{-\nu\pi i}e^{-\hh\lam (g_1(\zeta)-g_2(\zeta))}(E_{4,3}-E_{1,2})), \qquad \text{for $\zeta\in\til\Gamma_3$},\\
  J_{\hh C}(\zeta)&=\begin{pmatrix}
  e^{\hh\lam (g_{1,+}(\zeta)-g_{1,-}(\zeta))} & e^{\nu\pi i} & 0 & 0\\
  -e^{-\nu\pi i} & 0 & 0 & 0\\
  0 & 0 & 0 & e^{-\nu\pi i}\\
  0 & 0 & -e^{\nu\pi i} & e^{-\hh\lam (g_{2,+}(\zeta)-g_{2,-}(\zeta))}
  \end{pmatrix},
    \qquad \text{for $\zeta\in(-2i,0)$},\\
  J_{\hh C}(\zeta)&= (I + e^{\nu\pi i}e^{-\hh\lam (g_1(\zeta)-g_2(\zeta))}(E_{4,3}-E_{1,2})), \qquad \text{for
  $\zeta\in\til\Gamma_7$},\\
  J_{\hh C}(\zeta)&= (I - e^{-\nu\pi i}e^{\hh\lam (g_1(\zeta)-g_2(\zeta))}(E_{3,4}-E_{2,1})), \qquad \text{for
  $\zeta\in\til\Gamma_8$}.
  \end{align*}

  \item[(3)] As $\zeta\to\infty$, we have
  \begin{equation*}\label{asy of C:min}
  \hh C(\zeta)=\left(I+\frac{\hh C_1}{\zeta}+
  O\left(\frac{1}{\zeta^2}\right)\right)\diag((-\zeta)^{-1/4},\zeta^{-1/4},
  (-\zeta)^{1/4},\zeta^{1/4}) \Aa.
  \end{equation*}
\end{enumerate}
\end{rhp}

Now \eqref{d in B:min} transforms as follows:
\begin{equation}\label{d in C:min}
d =-i\sqrt{-s}(\hh C_1)_{1,4}.
\end{equation}

\subsection{Global and local parametrices}

\subsubsection*{Global parametrix}

The global parametrix is obtained from the RH problem for $\hh C$ by ignoring
all the exponentially decaying entries in the jump matrices:

\begin{rhp} We look for a $4\times 4$ matrix valued function $\hh C^{(\infty)}$ satisfying
\begin{enumerate}
  \item[(1)] $\hh C^{(\infty)}(\zeta)$ is analytic for $\zeta\in\cee\setminus
  (\er\cup [-2i,2i])$.

  \item[(2)] $\hh C^{(\infty)}(\zeta)$ has the jumps
  \begin{equation}
  \label{jumps:Chatinfty}
  \begin{array}{l}
  \hh C^{(\infty)}_{+}(\zeta)=\hh C^{(\infty)}_-(\zeta)
  (E_{2,2}+E_{4,4}+E_{1,3}-E_{3,1}),
  \qquad \text{for $\zeta>0$},
  \\
  \hh C^{(\infty)}_{+}(\zeta)=\hh C^{(\infty)}_-(\zeta)
  (E_{1,1}+E_{3,3}+E_{2,4}-E_{4,2}),
  \qquad \text{for $\zeta<0$},
  \\
  \hh C^{(\infty)}_{+}(\zeta)=\hh C^{(\infty)}_-(\zeta)
  \diag\left(\begin{pmatrix}0 & e^{-\nu\pi i} \\ -e^{\nu\pi i} & 0 \end{pmatrix},
  \begin{pmatrix}0 & e^{\nu\pi i} \\ -e^{-\nu\pi i} & 0 \end{pmatrix}
  \right),
  \qquad \text{for $\zeta\in(0,2i)$},\\
  \hh C^{(\infty)}_{+}(\zeta)=\hh C^{(\infty)}_-(\zeta)
    \diag\left(\begin{pmatrix}0 & e^{\nu\pi i} \\ -e^{-\nu\pi i} & 0 \end{pmatrix},
  \begin{pmatrix}0 & e^{-\nu\pi i} \\ -e^{\nu\pi i} & 0 \end{pmatrix}
  \right),
  \qquad \text{for $\zeta\in(-2i,0)$}.
  \end{array}\end{equation}

  \item[(3)] As $\zeta\to\infty$, we have
  \begin{equation}\label{Cinfty:asy:min}
  \hh C^{(\infty)}(\zeta)=\left(I+\frac{\hh C^{(\infty)}_1}{\zeta}+O\left(\frac{1}{\zeta^2}\right)
  \right)\diag((-\zeta)^{-1/4},\zeta^{-1/4},
  (-\zeta)^{1/4},\zeta^{1/4})
  \Aa.
  \end{equation}

  \item[(4)] We have \begin{equation}\begin{array}{ll}
  \hh C^{(\infty)}(\zeta)= O(|\zeta\mp 2i|^{-1/4}),&\qquad \textrm{as $\zeta\to \pm
  2i$},\\
  \label{Chatinfty:zero}\hh
  C^{(\infty)}(\zeta)\diag(1,\zeta^{\nu},1,\zeta^{-\nu})=O(1),&\qquad \textrm{as $\zeta\to 0$ with
  $\Re\zeta>0$.}\end{array}\end{equation}

\end{enumerate}
\end{rhp}

\begin{lemma}\label{lemma:smininfty:global}
There exists a solution $\hh C^{(\infty)}$ to the above RH problem. Moreover,
the $(1,4)$ entry of the matrix $\hh C_1^{(\infty)}$ in \eqref{Cinfty:asy:min}
is given by
\begin{equation}\label{Cinfty:14}
(\hh C^{(\infty)}_1)_{1,4} = i.
\end{equation}
\end{lemma}

\begin{proof}
We will give an explicit construction of $\hh C^{(\infty)}$. Define the
functions
\begin{equation}\label{xibis12:min}
v_1(z) = \frac{\xi_1(z)+\sqrt{\xi_1^2(z)+4}}{2},\qquad v_2(z) =
\frac{\xi_2(z)+\sqrt{\xi_2^2(z)+4}}{2},
\end{equation}
where we recall the definition of $\xi_1,\xi_2$ in
\eqref{xi12a:min}--\eqref{xi12b:min}. Note that
\begin{eqnarray}
v_1(z) &=& (-z)^{-1/2}+O((-z)^{-5/2}),
\quad z\to\infty,\\
v_2(z) &=& z^{1/2}+O(z^{-3/2}),\qquad\qquad\ \ z\to\infty.
\end{eqnarray}
Also define
$$ v_3(z) = -v_1(z),\qquad v_4(z) = -v_2(z).
$$

We construct a compact Riemann surface $\err$ with four sheets $\err_j$,
$j=1,\ldots,4$. The sheets are defined by
$\err_1=\err_3=\cee\setminus(\er_+\cup [-2i,2i])$ and
$\err_2=\err_4=\cee\setminus(\er_-\cup [-2i,2i])$. We glue the sheets $\err_1$
and $\err_3$ in the usual crosswise way along the cut $\er_+$. Similarly we
glue $\err_2$ and $\err_4$ along $\er_-$, and we glue $\err_1$ and $\err_2$
(and also $\err_3$ and $\err_4$) along the cut $(-2i,2i)$. We add points at
infinity and at the finite branch points $-2i$ and $2i$ to make $\err$ a
compact Riemann surface.

We consider the function $v_k(z)$ to be living on the $k$th sheet $\err_k$ of
the Riemann surface, $k=1,\ldots,4$. Together these functions provide a
bijective holomorphic mapping of the Riemann surface $\err$ into the Riemann
sphere $\overline\cee$. The images of the four sheets of the Riemann surface
under this mapping are given by
\begin{equation*}\begin{array}{cccc}
v_1(\err_1) &= \{z\in\overline\cee\mid |z|<1\textrm{  and  }\Re z>0\},\\
v_2(\err_2) &= \{z\in\overline\cee\mid |z|>1\textrm{  and  }\Re z>0\},\\
v_3(\err_3) &= \{z\in\overline\cee\mid |z|<1\textrm{  and  }\Re z<0\},\\
v_4(\err_4) &= \{z\in\overline\cee\mid |z|>1\textrm{  and  }\Re z<0\}.
\end{array}\end{equation*} The cuts of the Riemann surface are mapped into $i\er\cup
\mathbb T$ with $\mathbb T$ the unit circle. The images of the branch points
are
\begin{equation*}\begin{array}{cc} v_1(\pm 2i) = v_2(\pm 2i) = e^{\pm\pi i/4},& \qquad v_3(\pm 2i) =
v_4(\pm 2i) = -e^{\pm\pi i/4}, \\ v_1(\infty) = v_3(\infty) = 0,& \qquad
v_2(\infty) = v_4(\infty) = \infty.\end{array}\end{equation*} The points at the
origin of the four sheets also play a special role; they are mapped to the
points $\pm 1$ and $\pm i$ in the $v$-plane.

Define the polynomial
\begin{equation}\label{def:P}
    P(v)=2v(v^4+1),
\end{equation}
and its square root
\[ \sqrt{P(v)}, \]
defined as an analytic function in the $v$-plane, with cuts along the union of
arcs $$i\er_+\cup \{e^{i\theta}\mid
\theta\in[-\pi/4,\pi/4]\cup[3\pi/4,5\pi/4]\},$$ and also satisfying
$\sqrt{P(v)}>0$ as $v\to\infty$ along the positive real axis.

For $\nu=0$, we define the global parametrix $\hh C^{(\infty)}(\zeta;\nu)=\hh
C^{(\infty)}(\zeta;0)$ by
\begin{equation}\label{globalparametrix}
\hh C^{(\infty)}(\zeta;0) =
\begin{pmatrix} f_j(v_k(\zeta)) \end{pmatrix}_{j,k=1}^{4},\qquad \textrm{for }\zeta \in \cee\setminus \Sigma_{\hh
B},
\end{equation}
where the functions $f_j(v)$, $j=1,\ldots,4$ are defined by
\begin{equation}\label{def:f1234}
f_1(v) = -\frac{v}{\sqrt{P(v)}},\quad f_2(v) = \frac{v^2}{\sqrt{P(v)}},\quad
f_3(v) = \frac{i}{\sqrt{P(v)}},\quad f_4(v) = \frac{iv^3}{\sqrt{P(v)}},
\end{equation}
with the above discussed branch of $\sqrt{P(v)}$.

For general $\nu$, we define the global parametrix $\hh
C^{(\infty)}(\zeta;\nu)$ by
\begin{equation}\label{Cinfty:nu}\hh C^{(\infty)}(\zeta;\nu) = \mathcal K \hh C^{(\infty)}(\zeta;0) \diag\left(
e^{h(v_1(\zeta))},e^{h(v_2(\zeta))},e^{h(v_3(\zeta))},e^{h(v_4(\zeta))}\right)
\end{equation} where $\mathcal K$ is a matrix of the form \begin{equation}\label{Cinfty:nu:bis}
\mathcal K=I +\alpha E_{3,1}+\beta E_{4,2}\end{equation} for suitable constants
$\alpha,\beta$. The function $h$ in \eqref{Cinfty:nu} serves to create the jump
entries $e^{\nu\pi i}$ and $e^{-\nu\pi i}$ in \eqref{jumps:Chatinfty}. Hence
$h$ must be analytic in $\cee\setminus\mathbb T$ and with the counterclockwise
orientation of $\mathbb T$ we must have $$h_+(v)- h_-(v) = \pm\nu\pi i,\qquad
\textrm{for }v\in\mathbb T_{\pm},$$ where
\begin{equation*}\mathbb T_+ := \mathbb T\cap\{v\mid \Im v>0\},\qquad
\mathbb T_- := \mathbb T\cap\{v\mid \Im v<0\}.\end{equation*} We construct the
function $h$ explicitly in a Cauchy integral form:
\begin{equation} \label{Cinfty:h}h(v)=-\frac{\nu}{2}\int_{\mathbb{T}_+}
\frac{1}{v-x}\ \ud x + \frac{\nu}{2}\int_{\mathbb T_-} \frac{1}{v-x}\ \ud x.
\end{equation}
The Plemelj formula guarantees that $h$ has the desired jumps on $\mathbb T_+$
and $\mathbb T_-$. Moreover $h$ tends to zero for both $v\to 0$ and
$v\to\infty$. In fact, the dominant term of $e^{h(v_j(\zeta))}$ as
$\zeta\to\infty$ is of the order $O(\zeta^{-1/2})$ and it is then routine to
find the constant matrix $\mathcal K$ in
\eqref{Cinfty:nu}--\eqref{Cinfty:nu:bis} so that \eqref{Cinfty:asy:min} holds.

From the construction \eqref{Cinfty:nu}--\eqref{Cinfty:nu:bis} we see that the
$(1,4)$ entry of $\hh C_1^{(\infty)}$ in \eqref{Cinfty:asy:min} is independent
of $\nu$. In fact, it is not too hard to calculate this entry explicitly,
leading to the value $i$ in \eqref{Cinfty:14}. Finally, \eqref{Chatinfty:zero}
is also clearly satisfied. \end{proof}

\subsubsection*{Local parametrices at $-2i$ and $2i$}

Near the points $-2i$ and $2i$ one can again construct local parametrices $\hh
C^{(-2i)}$, $\hh C^{(2i)}$ using Airy functions. We omit the details.

\subsubsection*{Local parametrix at the origin}

At the origin we construct a local parametrix $\hh C^{(0)}$ in a similar way as
in Section~\ref{subsection:localPII}. We require $\hh C^{(0)}$ to have the same
jumps as $\hh C$ on the contour
$B_{\rho}\cap(\Gamma_1\cup\Gamma_4\cup\Gamma_6\cup\Gamma_9)$. For
$\zeta\in (0,\rho i)$ we impose the jump
\begin{equation}\label{jumps:Chatinfty:mod1}
\hh C^{(0)}_+(\zeta) = \hh C^{(0)}_-(\zeta)\times\left\{
\begin{array}{ll}
J_{\hh C}(\zeta),& \textrm{if $\nu-1/2\not\in\zet_{\geq 0}$},\\
e^{\nu \pi i}(E_{3,4}-E_{2,1})+e^{-\nu\pi i}(E_{1,2}-E_{4,3}),& \textrm{if
$\nu-1/2\in\zet_{\geq 0}$},
\end{array}\right. \end{equation}
and for $\zeta\in (-\rho i,0)$ we put
\begin{equation}\label{jumps:Chatinfty:mod2}
\hh C^{(0)}_+(\zeta) = \hh C^{(0)}_-(\zeta)\times\left\{
\begin{array}{ll}
J_{\hh C}(\zeta),& \textrm{if $\nu-1/2\not\in\zet_{\geq 0}$},\\
e^{-\nu \pi i}(E_{3,4}-E_{2,1})+e^{\nu\pi i}(E_{1,2}-E_{4,3}),& \textrm{if
$\nu-1/2\in\zet_{\geq 0}$}.
\end{array}\right. \end{equation}
We again use transformations of the type \eqref{squareroot:transfo} to move to
a \lq square root version\rq\ of the RH problems for $\hh C^{(0)}$ and $\hh
C^{(\infty)}$. The local parametrix in the square root setting can be
constructed as in \cite[Sec.~5.6]{DKRZ}, see also \cite[Sec.~5.6]{DGZ}, and it can then be lifted back to the
original setting by using \eqref{squareroot:transfo:inverse}. This is very
similar to Section~\ref{subsection:localPII}. The main difference is that we
are now dealing with \lq Case~III\rq\ rather than \lq Case~I\rq\ in the
terminology of \cite[Sec.~5.6]{DKRZ}; see also Fig.~\ref{fig:phasediagram}. The
details are irrelevant for our purposes and are omitted.

\subsection{Fourth transformation: $\hh C \mapsto \hh D$}
We define the final transformation
\begin{equation}\label{def:R:Lun:min}
\hh D(\zeta)=\left\{
       \begin{array}{ll}
         \hh C(\zeta)(\hh C^{(2i)})^{-1}(\zeta), & \text{if  $|\zeta-2i|<\rho$,} \\
         \hh C(\zeta)(\hh C^{(-2i)})^{-1}(\zeta), & \text{if $|\zeta+2i|<\rho$,} \\
         \hh C(\zeta)(\hh C^{(0)})^{-1}(\zeta), & \text{if $|\zeta|<\rho$,} \\
         \hh C(\zeta)(\hh C^{(\infty)})^{-1}(\zeta), & \text{elsewhere.}
       \end{array}
     \right.
\end{equation}

\begin{figure}[t]
\vspace{-5mm}
\begin{center}
   \setlength{\unitlength}{1truemm}
   \begin{picture}(100,70)(-5,2)
       \put(40,42.5){\line(0,1){5}}
       \put(40,37.5){\line(0,-1){5}}
       \put(42.3,41.15){\line(2,1){28}}
       \put(42.3,38.85){\line(2,-1){28}}
       \put(37.7,41.15){\line(-2,1){28}}
       \put(37.7,38.85){\line(-2,-1){28}}
       \put(41.4,52.1){\line(2,3){8.5}}
       \put(41.4,27.9){\line(2,-3){8.5}}
       \put(38.6,52.1){\line(-2,3){8.5}}
       \put(38.6,27.9){\line(-2,-3){8.5}}
       \put(40,40){\thicklines\circle*{1}}
       \put(40,30){\thicklines\circle*{1}}
       \put(40,50){\thicklines\circle*{1}}
       \put(42,35.5){$0$}
       \put(43,29){$-2i$}
       \put(44,49){$2i$}
       \put(40,47){\thicklines\vector(0,1){.0001}}
       \put(40,36){\thicklines\vector(0,1){.0001}}
       \put(60,50){\thicklines\vector(2,1){.0001}}
       \put(60,30){\thicklines\vector(2,-1){.0001}}
       \put(20,50){\thicklines\vector(2,-1){.0001}}
       \put(20,30){\thicklines\vector(2,1){.0001}}
       \put(45.2,58){\thicklines\vector(2,3){.0001}}
       \put(45.3,22){\thicklines\vector(2,-3){.0001}}
       \put(34.8,58){\thicklines\vector(-2,3){.0001}}
       \put(34.7,22){\thicklines\vector(-2,-3){.0001}}

       \put(41.2,32.3){\thicklines\vector(1,0){.0001}}
       \put(41.2,42.3){\thicklines\vector(1,0){.0001}}
       \put(41.2,52.3){\thicklines\vector(1,0){.0001}}

       \put(40,40){\circle{5}}
       \put(40,30){\circle{5}}
       \put(40,50){\circle{5}}

       \put(60,52.5){$\Gamma_1$}
       \put(50,62){$\til\Gamma_2$}
       \put(26,62){$\til\Gamma_3$}
       \put(20.5,50){$\Gamma_4$}
       \put(20,27){$\Gamma_6$}
       \put(26,15){$\til\Gamma_7$}
       \put(50,15){$\til\Gamma_8$}
       \put(60,25){$\Gamma_{9}$}
  \end{picture}
  \vspace{-14mm}
   \caption{Contour  $\Sigma_{\hh D}$ for the RH problem for $\hh D$ in the case $\nu-1/2\not\in\zet_{\geq 0}$.
   If $\nu-1/2\in\zet_{\geq 0}$ then there is an additional jump contour on the vertical line segment $(-\rho i,\rho i)$.}
   \label{fig:ContourD:min}
\end{center}
\end{figure}

Then $\hh D$ satisfies the following RH problem.
\begin{rhp} We look for a $4\times 4$ matrix valued function $\hh D$ satisfying
\begin{enumerate}
\item[(1)] $\hh D$ is analytic in $\cee\setminus \Sigma_{\hh D}$, where $\Sigma_{\hh D}$ is
shown in Figure \ref{fig:ContourD:min}.

\item[(2)] $\hh D$ has jumps $\hh D_+(\zeta)=\hh D_-(\zeta)J_{\hh D}(\zeta)$ for
 $\zeta\in \Sigma_{\hh D}$, where
\begin{equation}\label{jump for D:min}
J_{\hh D}=\left\{
      \begin{array}{ll}
        \hh C^{(0)} \left(\hh C^{(\infty)} \right)^{-1}, & \text{on $|\zeta| = \rho$,} \\
        \hh C^{(2i)} \left(\hh C^{(\infty)} \right)^{-1}, & \text{on $|\zeta - 2i| = \rho$,} \\
        \hh C^{(-2i)} \left(\hh C^{(\infty)} \right)^{-1}, & \text{on $|\zeta + 2i| = \rho$,} \\
        \hh C^{(\infty)} J_{\hh C} \left(\hh C^{(\infty)}\right)^{-1}, & \text{on the rest of $\Sigma_{\hh D}$}.
      \end{array}
    \right.
\end{equation}
\item[(3)] As $\zeta \to
\infty$, we have
\begin{equation}\label{eq: D asy:min}
\hh D(\zeta)=I+\frac{\hh D_1}{\zeta}+O\left(\frac{1}{\zeta^2}\right).
\end{equation}
\item[(4)] $\hh D(\zeta)$ is bounded in a neighborhood of $\zeta=0$.
\end{enumerate}
\end{rhp}

The jump matrix for $\hh D$ satisfies
\begin{equation*}
J_{\hh D}(\zeta)=I+O(1/\hh\lambda), \qquad \text{as $\hh\lambda \to +\infty$},
\end{equation*}
uniformly for $z$ on the circles $|z| = \rho$, $|z+2i| = \rho$ and $|z-2i| =
\rho$, and the jumps on the remaining contours of $\Sigma_{\hh D}$ are
uniformly bounded and converging exponentially fast to the identity matrix as
$\hh\lam\to\infty$. We then again have estimates like in \eqref{RLun:small}--\eqref{R1Lun:small}.

It follows from \eqref{d in C:min}, $\hh\lam=(-s)^{3/2}$ and the above
constructions that
\begin{equation}\label{cctild:R:min} d(s)\ =\ -i\sqrt{-s}(\hh C^{(\infty)}_1)_{1,4}+O(s^{-1})\ =\ \sqrt{-s}+O(s^{-1}),
\end{equation}
where we used \eqref{Cinfty:14}. This ends the proof of \eqref{eq:asy of d:bis}
and thereby of Proposition~\ref{prop:large s solvability}. $\bol$

\section{Proof of Theorem~\ref{theorem:Painleve2modelrhp}}
\label{section:proofTheoremP2}

In this section we prove Theorem~\ref{theorem:Painleve2modelrhp} on the
connection of the residue matrix $M_1$ with the Hastings-McLeod solution to
Painlev\'e~II. We proceed in several steps.

\subsection{Symmetries of the Riemann-Hilbert problem for $M(\zeta)$}

First we will check that the residue matrix $M_1$ can be written as in
\eqref{M1:explicit}. To this end we use the symmetry relations of the model RH
problem; some of these relations will also be used in the next sections. In
what follows we use the elementary permutation matrix
\begin{equation}\label{J:matrix}
J = \begin{pmatrix} 0&1\\ 1&0\end{pmatrix}
\end{equation}
and we use $I$ for the $2\times 2$
identity matrix.

\begin{lemma}\label{lemma:symmetries} (Symmetries).
For any fixed $\nu>-1/2$ and $r_1,r_2,s,\tau$, we have the symmetry relations
\begin{align}\label{symmetry:conjugate}
\overline{M(\overline{\zeta};\overline{r_1},\overline{r_2},\overline{s},\overline{\tau})}
=
\begin{pmatrix} I & 0 \\ 0 & -I \end{pmatrix}
M(\zeta;r_1,r_2,s,\tau) \begin{pmatrix} I & 0 \\ 0 & -I
\end{pmatrix},
\\ \label{symmetry:special}
M(-\zeta;r_1,r_2,s,\tau) =
\begin{pmatrix} J & 0 \\ 0 & -J \end{pmatrix}
M(\zeta;r_2,r_1,s,\tau)  \begin{pmatrix} J & 0 \\ 0 & -J
\end{pmatrix},
\\ \label{symmetry:inversetranspose}
M^{-T}(\zeta;r_1,r_2,s,\tau) =
\begin{pmatrix} 0 & I \\ -I & 0 \end{pmatrix}
M(\zeta;r_1,r_2,s,-\tau) \begin{pmatrix} 0 & -I \\ I & 0 \end{pmatrix},
\end{align}
where the bar denotes the complex conjugation.
\end{lemma}

\begin{proof} One checks that the left and right hand sides of
\eqref{symmetry:conjugate} satisfy the same RH problem. Then
\eqref{symmetry:conjugate} follows from the uniqueness of the solution to this
RH problem. The same argument applies to \eqref{symmetry:special} and
\eqref{symmetry:inversetranspose}.
\end{proof}

\begin{corollary} (Symmetries of the residue matrix $M_1$).
For any fixed $\nu>-1/2$, $r_1=r_2=:r$, $s,\tau\in\er$, we have
\begin{align*}  \overline{M_1} = \begin{pmatrix} I & 0 \\ 0 & -I \end{pmatrix}
M_1 \begin{pmatrix} I & 0 \\ 0 & -I
\end{pmatrix},
\\ -M_1 =
\begin{pmatrix} J & 0 \\ 0 & -J \end{pmatrix}
M_1  \begin{pmatrix} J & 0 \\ 0 & -J
\end{pmatrix}.
\end{align*}
Consequently, $M_1$ takes the form \eqref{M1:explicit}
where $a,b,c,d,e,f,g,h$ are real valued constants depending parametrically on
$\nu,r,s,\tau$. If $\tau=0$ then we have additionally
$$  -M_1^T = \begin{pmatrix} 0 & I \\ -I & 0 \end{pmatrix} M_1 \begin{pmatrix}
0 & -I
\\ I & 0 \end{pmatrix},
$$
and then in \eqref{M1:explicit} we have $g=a$ and $h=b$.
\end{corollary}

\subsection{Behavior of $M(\zeta)$ near the origin}

For the proof of Theorem~\ref{theorem:Painleve2modelrhp}, we will need detailed
asymptotics of $M(\zeta)$ near the origin. This is provided in the following
proposition.

\begin{proposition}\label{prop:behaviorMzero} (Behavior of $M$ near the
origin). Let $M(\zeta)$ be the solution to the model RH
problem~\ref{rhp:modelM}. Then we have, with all the branches being principal:
\textrm{}\begin{itemize}
\item If $\nu-1/2\not\in\zet_{\geq 0}$, then there exist an analytic matrix
valued function $E$ and constant matrices $A_k$ such that, for the regions
$\Omega_k$ in Fig.~\ref{fig:modelRHP},
\begin{equation}\label{singularbehzero:1}
M(\zeta) = E(\zeta)\diag(\zeta^{\nu},\zeta^{\nu},\zeta^{-\nu},\zeta^{-\nu})A_k,
\qquad \zeta\in\Omega_k,\quad k=0,\ldots,9.
\end{equation}
Letting $J_k$ denote the jump matrix for $M$ on
$\Gamma_k$, we have
\begin{equation}\label{singularbehzero:2} A_{k+1} = A_{k}J_{k+1}, \qquad k=0,\ldots,3,5,\ldots,9,
\end{equation}
where we identify $A_{10}\equiv A_0$, $J_{10}\equiv J_0$ and with, among
others,
\begin{equation}\label{singularbehzero:2bis}
A_1 =
D\begin{pmatrix} 0&0&-1 & 2\cos \nu\pi\\ 2\cos\nu\pi&1&e^{-\nu\pi i}&0 \\ 0&1&0&0 \\
0&0&1&0\end{pmatrix},\quad A_5 = \begin{pmatrix} 0&1&-2\cos \nu\pi&1\\ 1&e^{\nu\pi i}&0&-e^{-\nu\pi i} \\ 1&0&0&0 \\
0&1&0&1\end{pmatrix},
\end{equation}
with $D=\diag(e^{-\nu\pi i},e^{-\nu\pi i},-e^{\nu\pi i},e^{\nu\pi i})$.
\item If $\nu-1/2\in\zet_{\geq 0}$, then $M$ has logarithmic behavior at the
origin: There exist an analytic matrix-valued function $E$ and constant
matrices $A_k$ such that
\begin{equation}\label{singularbehzero:1zet}
M(\zeta) = E(\zeta)K(\zeta)A_k, \qquad \zeta\in\Omega_k,\quad k=0,\ldots,9,
\end{equation}
with \begin{equation}\label{singularbehzero:1zet:K}
K(\zeta):=\begin{pmatrix}\zeta^{\nu}&0&\frac{1}{\pi}\zeta^{\nu}\log \zeta&0\\
0&\zeta^{\nu}&0&\frac{1}{\pi}\zeta^{\nu}\log \zeta\\ 0&0&\zeta^{-\nu}&0\\
0&0&0&\zeta^{-\nu}\end{pmatrix}. \end{equation} The matrices $A_k$ still
satisfy \eqref{singularbehzero:2} and we have, among others,
\begin{equation}\label{singularbehzero:2biszet}
A_1 = \begin{pmatrix} \mp i&1&1\mp i&\pm i \\
0&0&1&\pm i\\ 0&i&\pm 1+i&0\\ 0&0&i&0 \end{pmatrix},\quad
A_5 = \begin{pmatrix} 0&1&1&0 \\
0&0&1&0\\ \mp 1&\pm 1-i&0&\pm 1-i\\ 0&\pm 1&0&\pm 1 \end{pmatrix},
\end{equation}
in the case where $e^{\nu\pi i}=\pm i$ respectively.
\end{itemize}
\end{proposition}

\begin{remark} The factorization \eqref{singularbehzero:1} is not unique.
Indeed, for any non-singular $4\times 4$ block diagonal matrix
$$ K = \begin{pmatrix}*&*&0&0 \\
*&*&0&0 \\
0&0&*&* \\
0&0&*&*
\end{pmatrix},
$$
one can substitute $E(\zeta)$ and $A_k$ in \eqref{singularbehzero:1} by
$E(\zeta)\mathcal K$ and $\mathcal K^{-1}A_k$, respectively.
\end{remark}

\begin{proof} The proof uses ideas from \cite{FIKN}, see also \cite{Claeys2,IK,IKO} among others, but the details will be
more involved since we are now dealing with $4\times 4$ matrices.

First we consider the case where $\nu-\frac 12\not\in\zet_{\geq 0}$. Then the
matrices $A_k$ in \eqref{singularbehzero:2}--\eqref{singularbehzero:2bis} are
invertible. We can then \emph{define} $E$ by \eqref{singularbehzero:1}, i.e.,
we put
\begin{equation}\label{E:twisting} E(\zeta)
=M(\zeta)A_k^{-1}\diag(\zeta^{-\nu},\zeta^{-\nu},\zeta^{\nu},\zeta^{\nu}),\qquad
\zeta\in\Omega_k.
\end{equation}
By construction, $E$ is analytic in $\cee\setminus\bigcup_{k=0}^{9}\Gamma_k$.
We now show that $E$ is indeed entire. The relations \eqref{singularbehzero:2}
show that $E$ is analytic also on $\bigcup_{k=0}^{9}\Gamma_k\setminus
\Gamma_5$. Moreover, on $\Gamma_5=\er_-$ oriented from left to right we have
\begin{equation}\label{singularbehzero:3}
E_+^{-1}(\zeta)E_-(\zeta) =
\diag(\zeta^{\nu},\zeta^{\nu},\zeta^{-\nu},\zeta^{-\nu})_+ A_4J_5A_5^{-1}
\diag(\zeta^{-\nu},\zeta^{-\nu},\zeta^{\nu},\zeta^{\nu})_-.
\end{equation}
Recall the matrix $J$ in \eqref{J:matrix}. By a straightforward calculation,
\begin{eqnarray}
\nonumber A_4J_5 A_5^{-1} &=& A_5 \left(J_6 J_7 J_8 J_9 J_0 J_1 J_2 J_3 J_4
J_5\right) A_5^{-1}
\\
\nonumber &=& A_5 \left(\diag(J,-J) J_1 J_2 J_3 J_4 J_5\diag(J,-J)J_1 J_2 J_3
J_4 J_5\right) A_5^{-1}
\\
\label{singularbehzero:4}&=& \left(A_5 \diag(J,-J)J_1 J_2 J_3 J_4 J_5
A_5^{-1}\right)^2,
\end{eqnarray}
where in the first step we used the defining relation \eqref{singularbehzero:2}
for $A_k$, and in the second step we used that
\begin{equation}\label{Jk:sym}J_k=\diag(J,-J)J_{k+5}\diag(J,-J),\qquad
k=0,\ldots,4.\end{equation} To further evaluate \eqref{singularbehzero:4}, one
calculates
\begin{equation}\label{singularbehzero:4bis}
\begin{pmatrix}J&0\\0&-J\end{pmatrix} J_1 J_2 J_3 J_4 J_5 =
\begin{pmatrix} -e^{\nu\pi i}&0&0&0\\ 1&0&0&-e^{-\nu\pi i}\\
0&-1&e^{-\nu\pi i}&-1 \\ -1&-e^{\nu\pi i}&0&e^{-\nu\pi i}-e^{\nu\pi
i}\end{pmatrix}.
\end{equation}
It is easy to see that this matrix has eigenvalues $e^{-\nu\pi i}$ and
$-e^{\nu\pi i}$, both with multiplicity two. Moreover, the matrix $A_5$ in
\eqref{singularbehzero:2bis} was chosen in such a way that its first, second,
third and fourth row are left eigenvectors corresponding to the eigenvalues
$e^{-\nu\pi i},e^{-\nu\pi i},-e^{\nu\pi i},-e^{\nu\pi i}$, respectively. Using
this in \eqref{singularbehzero:4} we get
\begin{eqnarray*}
A_4J_5A_5^{-1} &=& \left(\diag(e^{-\nu\pi i},e^{-\nu\pi i},-e^{\nu\pi i},-e^{\nu\pi i})\right)^2\\
&=& \diag(e^{-2\nu\pi i},e^{-2\nu\pi i},e^{2\nu\pi i},e^{2\nu\pi i})\\
&=&
\diag(\zeta^{-\nu},\zeta^{-\nu},\zeta^{\nu},\zeta^{\nu})_+\diag(\zeta^{\nu},\zeta^{\nu},\zeta^{-\nu},\zeta^{-\nu})_-.
\end{eqnarray*}
Inserting this result in \eqref{singularbehzero:3} we see that $E$ is analytic
also on $\er_-$, and therefore in $\cee\setminus\{0\}$.

It remains to show that the singularity at $0$ is removable. If $\nu<0$ then we
see from the definition \eqref{E:twisting} of $E(\zeta)$ and from the behavior
\eqref{Mzero:smaller} of $M$ near the origin that $E(\zeta) = O(\zeta^{2\nu})$
as $\zeta\to 0$, so (since $2\nu>-1$) the isolated singularity at $0$ is indeed
removable. 

If $\nu>0$ and $\zeta\to 0$ in $\Omega_1$ then we find in a similar way (using
\eqref{Mzero:greater} and \eqref{E:twisting}) that
$$ E(\zeta)
=C(\zeta)\diag(\zeta^{\nu},\zeta^{-\nu},\zeta^{-\nu},\zeta^{\nu})\begin{pmatrix}
* & *
& *& * \\ 0 & 0 & *& * \\ 0 & 0 & *& * \\
* & * & *& *
\end{pmatrix} \diag(\zeta^{-\nu},\zeta^{-\nu},\zeta^{\nu},\zeta^{\nu}),\qquad
\zeta\in\Omega_1,
$$
where $C(\zeta)=O(1)$ for $\zeta\to 0$ and we used the particular form of the
matrix $A_1^{-1}$, cf.~\eqref{singularbehzero:2bis}. It follows that $E(\zeta)$
is bounded near the origin in $\Omega_1$ and so $0$ cannot be a pole. Since $0$
cannot be an essential singularity either, the singularity is indeed removable.
This ends the proof of the proposition if $\nu-\frac 12\not\in\zet_{\geq 0}$.

Next we consider the case where $\nu-\frac 12\in\zet_{\geq 0}$. In this case,
\eqref{E:twisting} and \eqref{singularbehzero:3} are replaced by
\begin{equation}\label{E:twisting:zet} E(\zeta)
=M(\zeta)A_k^{-1}K(\zeta)^{-1},\qquad \zeta\in\Omega_k,
\end{equation}
and
\begin{equation}\label{singularbehzero:3zet} E_+^{-1}(\zeta)E_-(\zeta) =
K_+(\zeta) A_4J_5A_5^{-1} K_-(\zeta)^{-1},
\end{equation}
for $\zeta\in\er_-$. Now \eqref{singularbehzero:4bis} reduces to
\begin{equation}\label{singularbehzero:4biszet}
\begin{pmatrix}J&0\\0&-J\end{pmatrix} J_1 J_2 J_3 J_4 J_5 =
\begin{pmatrix} \mp i&0&0&0\\ 1&0&0&\pm i\\
0&-1&\mp i&-1 \\ -1&\mp i&0&\mp 2i\end{pmatrix},
\end{equation} in the case where $e^{\nu\pi i}=\pm i$ respectively. All the eigenvalues of this
matrix have the same value $e^{-\nu\pi i}=-e^{\nu\pi i}=\mp i$. The eigenvector
space corresponding to this eigenvalue is two-dimensional. The matrix $A_5$ in
\eqref{singularbehzero:2biszet} was chosen in such a way that its third and
first row, and similarly its fourth and second row, form Jordan chains in the
sense that
\begin{equation*}A_5 \begin{pmatrix}J&0\\0&-J\end{pmatrix} J_1 J_2 J_3 J_4 J_5 A_5^{-1} =
\mp\begin{pmatrix} i&0&1&0\\ 0&i&0&1\\
0&0&i&0 \\ 0&0&0&i\end{pmatrix}.
\end{equation*}
Hence by \eqref{singularbehzero:4},
$$ A_4J_5A_5^{-1} = \begin{pmatrix} i&0&1&0\\ 0&i&0&1\\
0&0&i&0 \\ 0&0&0&i\end{pmatrix}^2 = \begin{pmatrix} -1&0&2i&0\\ 0&-1&0&2i\\
0&0&-1&0 \\ 0&0&0&-1\end{pmatrix}.
$$
A straightforward calculation shows that this matrix can be written in the form
$$ A_4J_5A_5^{-1}=K_+(\zeta)^{-1}K_-(\zeta),
$$
recall \eqref{singularbehzero:1zet:K}. Inserting this in
\eqref{singularbehzero:3zet} shows that $E$ is analytic on $\er_-$. By
definition, $E$ is also analytic in $\cee\setminus\er_-$. Finally, one shows as
before that the singularity at the origin is removable, so $E$ is analytic in
the entire complex plane.
\end{proof}

We need the following symmetry relation for $E(\zeta)$.

\begin{lemma} (Symmetry of $E$). Assume that $r_1=r_2=:r$ and define $J$ in \eqref{J:matrix}.
With the notations in Proposition~\ref{prop:behaviorMzero}, we have the
symmetry relation
\begin{equation}\label{singularbehzero:9}
E(\zeta) = \begin{pmatrix}J&0\\0&-J\end{pmatrix}
E(-\zeta)\diag(1,1,-1,-1),\qquad \zeta\in\cee.
\end{equation}
\end{lemma}

\begin{proof} We will give the proof for $\nu-\frac 12\not\in\zet_{\geq 0}$; the other
case follows from similar considerations.

From the definition \eqref{singularbehzero:1} of $E(\zeta)$ and the symmetry
relation \eqref{symmetry:special} of the RH problem for $M(\zeta)$ it follows
that
\begin{multline}\label{singularbehzero:8} E(\zeta) =
\diag(J,-J) E(-\zeta)
\diag((-\zeta)^{\nu},(-\zeta)^{\nu},(-\zeta)^{-\nu},(-\zeta)^{-\nu})
\\ \times A_{k}\diag(J,-J) A_{k+5}^{-1}\diag(\zeta^{-\nu},\zeta^{-\nu},\zeta^{\nu},\zeta^{\nu}),
\end{multline}
if $\zeta\in\Omega_{k+5}$, for any $k=0,\ldots,4$. We will use this relation
for $k=0$. Then
\begin{eqnarray}\nonumber A_{0}\diag(J,-J)A_{5}^{-1} &=& A_5 \left( J_6J_7J_8J_9J_0\diag(J,-J)\right) A_5^{-1}\\
\nonumber &=&
A_5 \left( \diag(J,-J)J_1J_2J_3J_4J_5\right) A_5^{-1}\\
\label{hulpsym} &=& \diag(e^{-\nu\pi i},e^{-\nu\pi i},-e^{\nu\pi i},-e^{\nu\pi
i}),
\end{eqnarray}
where we used again the relations \eqref{singularbehzero:2}, \eqref{Jk:sym} and the fact that the  first,
second, third and fourth row of $A_5$ are left eigenvectors corresponding to
the eigenvalues $e^{-\nu\pi i},e^{-\nu\pi i}$, $-e^{\nu\pi i},-e^{\nu\pi i}$ of
the matrix \eqref{singularbehzero:4bis}, respectively. Inserting
\eqref{hulpsym} in \eqref{singularbehzero:8} (with $k=0$), and using the
principal branches of $\zeta^{\nu}$ and $\zeta^{-\nu}$ with $\zeta\in\Omega_5$,
we obtain the desired relation \eqref{singularbehzero:9} for
$\zeta\in\Omega_5$. From \eqref{Jk:sym} and \eqref{singularbehzero:8} it
follows that the relation \eqref{singularbehzero:9} remains valid in the other
sectors $\Omega_{k+5}$, $k=0,\ldots,4$ and hence in the entire complex plane.
\end{proof}

We are now ready to prove the following lemma.

\begin{lemma}\label{cor:residue} Assume that $r_1=r_2=:r$.
With the notations in Proposition~\ref{prop:behaviorMzero}, we have
\begin{equation}\label{singularbehzero:5} \frac{\partial M}{\partial \zeta}M^{-1}(\zeta) = \frac{\partial
E}{\partial \zeta}E^{-1}(\zeta)+
\frac{\nu}{\zeta}E(\zeta)\diag(1,1,-1,-1)E^{-1}(\zeta), \qquad \zeta\in\cee,
\end{equation}
if $\nu-\frac 12\not\in\zet_{\geq 0}$. The first term in the right hand side of \eqref{singularbehzero:5} is analytic,
whereas the second term behaves as
\begin{equation}\label{singularbehzero:6}
\frac{1}{\zeta} \diag\left(\begin{pmatrix}0&\nu \\ \nu&0 \end{pmatrix},
\begin{pmatrix}0&-\nu \\ -\nu&0 \end{pmatrix}\right)+O(1),\qquad \zeta\to 0.
\end{equation}
If $\nu-\frac 12\in\zet_{\geq 0}$ then the same conclusions hold except that,
in the right hand side of \eqref{singularbehzero:5}, we have
an extra term of order $O(z^{2\nu-1})$ as $z\to 0$.
\end{lemma}

\begin{proof} We restrict ourselves to the case where $\nu-\frac 12\not\in\zet_{\geq 0}$; the other case is similar.
The equation \eqref{singularbehzero:5} follows immediately from
\eqref{singularbehzero:1}. Next, we will show that the second term of
\eqref{singularbehzero:5} behaves as \eqref{singularbehzero:6} near $\zeta=0$.
To this end, we must show that \begin{equation}\label{singularbehzero:7}
E(0)\diag(1,1,-1,-1)E^{-1}(0) = \left(\begin{pmatrix}0&1 \\ 1&0
\end{pmatrix},
\begin{pmatrix}0&-1 \\ -1&0 \end{pmatrix}\right).
\end{equation}
This follows by substituting $\zeta=0$ in \eqref{singularbehzero:9}.
\end{proof}

\begin{remark}\label{remark:E0:pattern}
Substituting $\zeta=0$ in \eqref{singularbehzero:9} we get a number of
equations for the entries of $E(0)$. It is then straightforward to check that
$E(0)$ must be of the form \eqref{E0:pattern:0}.
\end{remark}

\subsection{Lax pair and compatibility conditions}

Now we proceed with the proof of Theorem~\ref{theorem:Painleve2modelrhp}. We
will derive the Lax pair corresponding to the RH problem for $M(\zeta)$.
Throughout this section we take a fixed $\nu>-1/2$, $\tau\in\er$ and we let
$r_1=r_2=1$. First we obtain a differential equation with respect to $\zeta$.

\begin{proposition}\label{prop:diffeq1}
We have the differential equation
\begin{equation}\label{diffeq1}
\frac{\partial M}{\partial\zeta}=UM
\end{equation}
where (recall \eqref{M1:explicit})
\begin{equation}\label{def:U:Lax}
U:=\begin{pmatrix} -c+\tau & d+\nu\zeta^{-1} & i & 0 \\ -d+\nu\zeta^{-1} & c-\tau & 0 & i \\
-i(-\zeta+g+a+s) & -i(b+h) & c+\tau & d-\nu\zeta^{-1} \\
-i(b+h) & -i(\zeta+g+a+s) & -d-\nu\zeta^{-1} & -c-\tau
\end{pmatrix}.
\end{equation}
\end{proposition}

\begin{proof}
On account of \eqref{M:asymptotics}, we have for $\zeta\to\infty$,
\begin{multline}\label{diffeq1:RH}
\frac{\partial M}{\partial\zeta}M^{-1} =
\left(I+\frac{M_1}{\zeta}+\frac{M_2}{\zeta^2}+\ldots\right)
\begin{pmatrix} \tau & 0 & i(1-s\zeta^{-1}) & 0 \\ 0 & -\tau & 0 & i(1+ s\zeta^{-1}) \\
i(\zeta-s) & 0 & \tau & 0 \\
0 & -i( \zeta+ s) & 0 & -\tau
\end{pmatrix}\\
\left(I-\frac{M_1}{\zeta}+\frac{M_1^2-M_2}{\zeta^2}+\ldots\right)+O(\zeta^{-1}).
\end{multline}
Since the jump matrices in the RH problem for $M$ are independent of $\zeta$,
the matrix $\frac{\partial M(\zeta)}{\partial \zeta}M(\zeta)^{-1}$ is analytic
for $\zeta\in\cee\setminus\{0\}$, and it has a first order pole at $\zeta=0$
given by \eqref{singularbehzero:6}. The terms in \eqref{diffeq1:RH} which are
polynomial in $\zeta$ are as in \eqref{def:U:Lax} and so by a standard
application of Liouville's theorem we obtain
\eqref{diffeq1}--\eqref{def:U:Lax}.
\end{proof}

From the proof of Proposition~\ref{prop:diffeq1} we can deduce more. As already
mentioned, we know that the $\zeta^{-1}$-coefficient of \eqref{diffeq1:RH} must
have the form in \eqref{singularbehzero:6}.
Evaluating the top rightmost $2\times 2$ block of this $\zeta^{-1}$-coefficient
we find after some calculations,
$$ i\begin{pmatrix}
a+c^2-d^2+g-s & b-h-2\tau d \\ -b+h+2\tau d & -a-c^2+d^2-g+s
\end{pmatrix} = 0.
$$
Hence we obtain the following result.

\begin{lemma}\label{lemma:threerelations}
The numbers $a,b,c,d,\ldots$ in \eqref{M1:explicit} satisfy the relations
\begin{align}\label{threeextrarelations:symm} & g=-a-c^2+d^2+s,
\\ \label{threeextrarelations:symm2} & h=b-2\tau d.
\end{align}
\end{lemma}

\begin{remark} By subtracting the $(1,2)$ and $(3,4)$ entries of the
$\zeta^{-1}$-coefficient in \eqref{diffeq1:RH}, we obtain the additional
relation \begin{equation}\label{threeextrarelations:symm3}
f = -2bc+c^2 d+2\tau
cd+2\tau^2 d-d^3-2sd+\nu.
\end{equation}
\end{remark}

Next we obtain a differential equation with respect to $s$:

\begin{proposition}\label{prop:diffeq2}
With the notations \eqref{M1:explicit}, we have the differential equation
\begin{equation}\label{diffeq2}
\frac{\partial M}{\partial s}=VM,
\qquad\quad V:=2\begin{pmatrix} c & d & -i & 0 \\ d & c & 0 & i \\
i(-\zeta+a+g) & i(b-h) & -c & -d \\
i(h-b) & -i(\zeta+a+g) & -d & -c
\end{pmatrix}.
\end{equation}
\end{proposition}

\begin{proof} The matrix
$\frac{\partial M(\zeta)}{\partial s}M(\zeta)^{-1}$ is again
analytic for $\zeta\in\cee\setminus\{0\}$. It is
also analytic at $\zeta=0$, due to the fact that the matrices $A_k$ in \eqref{singularbehzero:1} and
\eqref{singularbehzero:1zet} are independent of $s$.
Moreover, \eqref{M:asymptotics} yields
\begin{equation}\label{diffeq2:RH}
\frac{\partial M}{\partial s}M^{-1} = \left(I+\frac{M_1}{\zeta}+\ldots\right)
\begin{pmatrix} 0 & 0 & -2i & 0 \\ 0 & 0 & 0 & 2i \\
-2i\zeta & 0 & 0 & 0 \\
0 & -2i\zeta & 0 & 0
\end{pmatrix}
\left(I-\frac{M_1}{\zeta}+\ldots\right),
\end{equation}
as $\zeta\to\infty$. The terms which are polynomial in $\zeta$ are given by $V$ in \eqref{diffeq2}
and so by Liouville's theorem we obtain \eqref{diffeq2}.
\end{proof}

Now we turn to the compatibility condition of the two differential equations
\eqref{diffeq1} and \eqref{diffeq2}. By computing the partial derivative
$\frac{\partial^2 M}{\partial\zeta
\partial s}= \frac{\partial^2 M}{\partial s \partial \zeta}$ in two different
ways, thereby making use of the relations \eqref{diffeq1} and \eqref{diffeq2},
one obtains the matrix relation \begin{equation}\label{compat:laxpair}
\frac{\partial U}{\partial s} = \frac{\partial V}{\partial\zeta} +VU- UV.
\end{equation}
Writing this matrix relation in entrywise form leads to the following lemma.

\begin{lemma}\label{lemma:compatibility}
The numbers $b,c,d$ in \eqref{M1:explicit}, viewed as functions of $s$, satisfy
the following system of coupled first order differential equations:
\begin{eqnarray}
\label{compat:1} c' &=& 4d^2+2s, \\
\label{compat:3} d' &=& 4(-b+cd+\tau d),\\
\label{compat:4}b'&=& -4b(c+\tau)+4d(c^2+2\tau c+2\tau^2)-4d^3-6sd
+2\nu,
\end{eqnarray}
where the prime denotes the derivative with respect to $s$.
\end{lemma}

\begin{proof} One can write \eqref{compat:laxpair} in its entrywise form with the help of
\eqref{def:U:Lax} and \eqref{diffeq2}. This leads to a system of $4\times 4=16$
coupled first order differential equations. After some straightforward
calculations, one checks that these $16$ relations reduce to $4$ independent
conditions. The $(1,1)$ matrix entry yields \eqref{compat:1}, while the $(1,2)$
and $(3,2)$ entries yield
\begin{eqnarray}
\label{compat:3bis} d' &=& 4(c d-\tau d-h),\\
\label{compat:4bis}b'+h'&=&-4(2d(a+g)+bc+(h-b)\tau +ch+sd-\nu).
\end{eqnarray} The relations \eqref{compat:3}--\eqref{compat:4} now
follow from \eqref{compat:3bis}--\eqref{compat:4bis} and
\eqref{threeextrarelations:symm}--\eqref{threeextrarelations:symm2}.
\end{proof}

\begin{lemma}\label{lemma:compat:derb}
The number $b$ in \eqref{M1:explicit} satisfies
\begin{equation}\label{compat:bpd:6} b' = (c+\tau)d'+4\tau^2 d
-4d^3-6sd+2\nu,
\end{equation}
where the prime denotes the derivative with respect to $s$.
\end{lemma}

\begin{proof} This follows from substituting 
\eqref{compat:3} in \eqref{compat:4}. \end{proof}

\begin{lemma}\label{lemma:dmodp2}
The number $d$ in \eqref{M1:explicit} satisfies
\begin{equation}\label{compat:f} d'' =
32d^3+32sd-16\tau^2 d-8\nu,
\end{equation}
where the prime denotes the derivative with respect to $s$.
\end{lemma}

\begin{proof} This follows from \eqref{compat:1}, \eqref{compat:3} and
\eqref{compat:bpd:6}. \end{proof}

\textsc{Proof of Theorem~\ref{theorem:Painleve2modelrhp}.} It is easily seen
that any solution $d=d(s)$ to \eqref{compat:f} is of the form $$d = 2^{-1/3}
q\left(2^{2/3} (2s-\tau^2)\right),$$ where $q$ is a solution to the Painlev\'
e~II equation~\eqref{def:Painleve2}. This yields \eqref{d:Painleve2}.

Next we must show that $q(s)$ in \eqref{d:Painleve2} is indeed the
\emph{Hastings-McLeod solution} to the Painlev\'e~II equation, i.e., we must
establish \eqref{def:HastingMcLeod1}--\eqref{def:HastingMcLeod2}. If $\tau=0$
then these relations follow from the asymptotic behavior of $d$ in
\eqref{eq:asy of d} and \eqref{eq:asy of d:bis}. If $\tau\in\er\setminus\{0\}$
one could perform a similar analysis as in Sections~\ref{section:asyM} and
\ref{section:asyM:minusinfty} to show that the same asymptotics hold true. An
alternative proof follows from Section~\ref{subsection:exist:2}. There we will
use the Lax pair relations, with $q$ the Hastings-McLeod solution, to
\emph{define} the matrix $M(\zeta)$ if $\tau\neq 0$, and we will then prove
that the resulting matrix $M(\zeta)$ indeed satisfies the RH
problem~\ref{rhp:modelM}.

Finally, we need to prove the formula \eqref{c:Hamiltonean} for $c$. By
comparing \eqref{compat:1} with \eqref{Hamiltonean:der}, we obtain the
following relation for the \emph{derivative} of $c=c(s)$ with respect to $s$:
\begin{equation*}
c'(s) = -2^{-1/3}\frac{\ud}{\ud s}(u(2^{2/3} (2s-\tau^2)))+2s.
\end{equation*}
Then \eqref{c:Hamiltonean} follows by integrating this relation with respect to
$s$. The fact that the integration constant in \eqref{c:Hamiltonean} is zero,
follows from \eqref{eq:asy of c} if $\tau=0$. The case where $\tau\neq 0$ can
be obtained as in the previous paragraph. $\bol$

As in \cite{DG}, we can also obtain a differential equation with respect to the
variable~$\tau$.

\begin{proposition}\label{prop:diffeq3}
With the notations \eqref{M1:explicit}, we have the differential equation
\begin{equation}\label{diffeq3}
\frac{\partial M}{\partial\tau}=WM,\qquad
W:=\begin{pmatrix} \zeta & -2b & 0 & -2id \\ -2b & -\zeta & 2id & 0 \\
0 & -2if & \zeta & -2h \\
2if & 0 & -2h & -\zeta
\end{pmatrix}.
\end{equation}
\end{proposition}

The compatibility condition of the differential equations \eqref{diffeq1} and
\eqref{diffeq3} now reads
\begin{equation}\label{compat:laxpair2}
\frac{\partial U}{\partial \tau} = \frac{\partial W}{\partial\zeta} +WU- UW.
\end{equation}
For later use, we note that the $(1,2)$ entry of \eqref{compat:laxpair2}
yields, with the help of \eqref{threeextrarelations:symm3},
\begin{equation}\label{compat:laxpair2:b} \frac{1}{4\tau}\frac{\partial d}{\partial \tau}=b-cd-\tau d.
\end{equation}

\section{Proof of Theorem~\ref{theorem:solvability}}
\label{section:vanishing}

In this section we prove Theorem~\ref{theorem:solvability} on the solvability
of the RH problem~\ref{rhp:modelM} for $M(\zeta)$ for real values of
$r_1=r_2=:r$, $s,\tau$. The proof uses Fredholm operator theory and the
technique of a \lq vanishing lemma\rq\ if $\tau=0$. For $\tau\neq 0$ we will
follow \cite[Sec.~5]{DG}.

\subsection{Fredholm property}

In this section we show that for any $\nu>-1/2$ and $r_1,r_2,s,\tau\in\cee$
with $r_1,r_2$ having positive real part, the singular integral operator
associated to the RH problem~\ref{rhp:modelM} is Fredholm with Fredholm index
zero. This follows from a powerful general procedure that has been applied in
different settings in the literature \cite{FZ,Zhou,Zhou1}. For completeness we
give a brief account of those arguments that involve the particular structure
of our $4\times 4$ RH problem. Hereby we closely follow the exposition in
\cite[Sec.~2.3]{IKO}.

Let $\dee:=\{\zeta\in\cee\mid|\zeta|<1\}$. Denote the matrix in
\eqref{M:asymptotics} by
$$ H(\zeta):=\diag((-\zeta)^{-1/4},\zeta^{-1/4},(-\zeta)^{1/4},\zeta^{1/4})
\Aa\diag\left(e^{-\theta_1(\zeta)+\tau\zeta},e^{-\theta_2(\zeta)-\tau\zeta},e^{\theta_1(\zeta)+\tau\zeta},e^{\theta_2(\zeta)-\tau\zeta}\right).
$$
In what follows we assume for convenience that $\nu-1/2\not\in\zet_{\geq 0}$.
We define a new $4\times 4$ matrix valued function $M^{(1)}(\zeta)$ by
$$ M^{(1)}(\zeta) = \left\{\begin{array}{ll} M(\zeta)A_j^{-1}\diag(\zeta^{-\nu},\zeta^{-\nu},\zeta^{\nu},\zeta^{\nu}),
& \quad
\zeta\in\Omega_j\cap\dee, \\
M(\zeta) H^{-1}(\zeta), & \quad
\zeta\in\Omega_j\cap\overline{\dee}^{\small{c}},\end{array}\right.
$$
$j=0,\ldots,9$, with $A_j$ the matrices in \eqref{singularbehzero:1}. Here
$\overline{\dee}^{\small{c}}$ stands for the complement of the closed unit
disk.

By Proposition~\ref{prop:behaviorMzero} we have that $M^{(1)}(\zeta)$ is
analytic in $\dee$. Let $\Sigma:=\bigcup_{j=0}^9\Gamma_j\cup\partial \dee$ and
orient it as in Figure~\ref{fig:RHP:m}. Then $\Sigma$ is a \emph{complete}
contour, in the sense that $\cee\setminus \Sigma$ allows a decomposition as the
disjoint union of two sets: $\cee\setminus \Sigma=\Omega_+\cup\Omega_-$,
$\Omega_+\cap\Omega_-=\emptyset$, such that $\Sigma$ is the positively oriented
boundary of $\Omega_+$ and the negatively oriented boundary of $\Omega_-$.
Denote $\Sigma_j=\Omega_j\cap\partial\dee$ as shown in Figure~\ref{fig:RHP:m}.

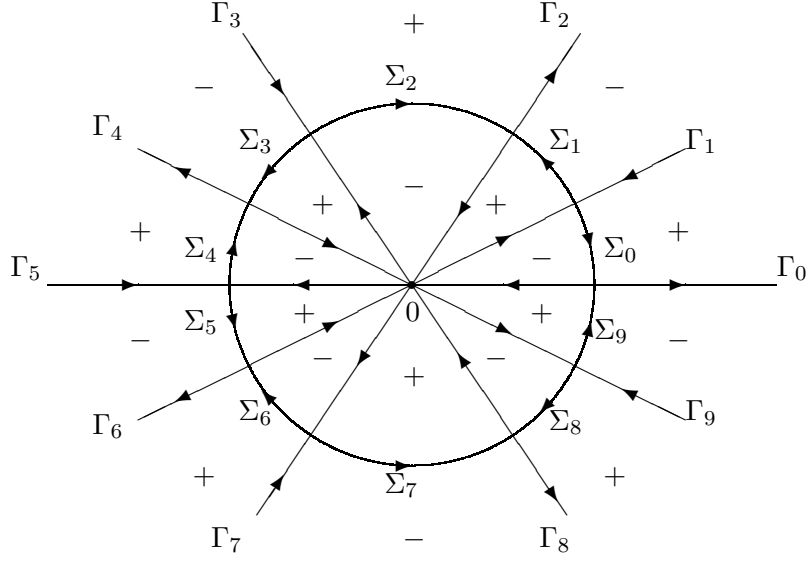
\begin{figure}[t]
\begin{center}
   \setlength{\unitlength}{1.2truemm}
   \begin{picture}(100,70)(-5,2)
       \put(40,40){\line(1,0){40}}
       \put(40,40){\line(-1,0){40}}
       \put(40,40){\line(2,1){30}}
       \put(40,40){\line(2,-1){30}}
       \put(40,40){\line(-2,1){30}}
       \put(40,40){\line(-2,-1){30}}
       \put(40,40){\line(2,3){18.5}}
       \put(40,40){\line(2,-3){17}}
       \put(40,40){\line(-2,3){18.5}}
       \put(40,40){\line(-2,-3){17}}
       \put(40,40){\thicklines\circle*{1}}
       \put(39.3,36){$0$}
       \put(70,40){\thicklines\vector(1,0){.0001}}
       \put(50,40){\thicklines\vector(-1,0){.0001}}
       \put(10,40){\thicklines\vector(1,0){.0001}}
       \put(27,40){\thicklines\vector(-1,0){.0001}}
       \put(63,51.5){\thicklines\vector(-2,-1){.0001}}
       \put(51,45.5){\thicklines\vector(2,1){.0001}}

       \put(63,28.5){\thicklines\vector(-2,1){.0001}}
       \put(51,34.5){\thicklines\vector(2,-1){.0001}}

       \put(45,47.5){\thicklines\vector(-2,-3){.0001}}
       \put(56,64){\thicklines\vector(2,3){.0001}}

       \put(45,32.5){\thicklines\vector(-2,3){.0001}}
       \put(56,16){\thicklines\vector(2,-3){.0001}}

       \put(14,53){\thicklines\vector(-2,1){.0001}}
       \put(32,44){\thicklines\vector(2,-1){.0001}}

       \put(14,27){\thicklines\vector(-2,-1){.0001}}
       \put(32,36){\thicklines\vector(2,1){.0001}}

       \put(34,49){\thicklines\vector(-2,3){.0001}}
       \put(26,61){\thicklines\vector(2,-3){.0001}}

       \put(34,31){\thicklines\vector(-2,-3){.0001}}
       \put(26,19){\thicklines\vector(2,3){.0001}}

       \put(80,41){$\Gamma_0$}
       \put(70,55.0){$\Gamma_1$}
       \put(54,69){$\Gamma_2$}
       \put(18,69){$\Gamma_3$}
       \put(5,56.5){$\Gamma_4$}
       \put(-4,41){$\Gamma_5$}
       \put(5,23.5){$\Gamma_6$}
       \put(18,11){$\Gamma_7$}
       \put(54,11){$\Gamma_8$}
       \put(70,25){$\Gamma_{9}$}

       \put(61,43){$\small{\Sigma_0}$}
       \put(59.7,44){\thicklines\vector(1,-4){.0001}}
       \put(53,42){$-$}
       \put(68,45){$+$}
       \put(55,55){$\small{\Sigma_1}$}
       \put(54,54.2){\thicklines\vector(-1,1){.0001}}
       \put(48,48){$+$}
       \put(61,61){$-$}
       \put(37,62){$\small{\Sigma_2}$}
       \put(40,60){\thicklines\vector(1,0){.0001}}
       \put(39,50){$-$}
       \put(39,68){$+$}
       \put(21,55){$\small{\Sigma_3}$}
       \put(23.5,51.7){\thicklines\vector(-1,-1){.0001}}
       \put(29,48){$+$}
       \put(16,61){$-$}
       \put(15,43){$\small{\Sigma_4}$}
       \put(20.7,45){\thicklines\vector(1,4){.0001}}
       \put(27,42){$-$}
       \put(9,45){$+$}
       \put(15,35){$\small{\Sigma_5}$}
       \put(20.7,35){\thicklines\vector(1,-4){.0001}}
       \put(27,36){$+$}
       \put(9,33){$-$}
       \put(21,25){$\small{\Sigma_6}$}
       \put(23.5,28.3){\thicklines\vector(-1,1){.0001}}
       \put(29,31){$-$}
       \put(16,18){$+$}
       \put(37,17){$\small{\Sigma_7}$}
       \put(40,20){\thicklines\vector(1,0){.0001}}
       \put(39,29){$+$}
       \put(39,11){$-$}
       \put(55,24){$\small{\Sigma_8}$}
       \put(54,25.8){\thicklines\vector(-1,-1){.0001}}
       \put(48,31){$-$}
       \put(61,18){$+$}
       \put(60,34){$\small{\Sigma_{9}}$}
       \put(59.7,36){\thicklines\vector(1,4){.0001}}
       \put(53,36){$+$}
       \put(68,33){$-$}

       \bigcircle{40}{40}{20}

  \end{picture}
   \vspace{-10mm}
   \caption{Contour $\Sigma=\bigcup_{j=0}^9 (\Gamma_j\cup\Sigma_j)$
   for the RH problem for $M^{(1)}(\zeta)$. The figure also shows the decomposition of
   $\cee\setminus\Sigma$ into two disjoint regions $\Omega_+$ and $\Omega_-$.}
   \label{fig:RHP:m}
\end{center}
\end{figure}

Now $M^{(1)}(\zeta)$ satisfies the following RH problem.

\begin{rhp}\label{rhp:mFredholm} We look for a $4\times 4$ matrix valued function
$M^{(1)}(\zeta)$ (which also depends parametrically on $\nu>-1/2$ and on
$r_1,r_2,s,\tau\in\cee$) satisfying
\begin{itemize}
\item[(1)] $M^{(1)}(\zeta)$ is analytic for $\zeta\in\cee\setminus\Sigma$.
\item[(2)] $M^{(1)}_+(\zeta) = M^{(1)}_-(\zeta)J_{M^{(1)}}(\zeta)$ for $\zeta\in\Sigma$, where
\begin{align*}
J_{M^{(1)}}(\zeta)= & I,\qquad\qquad\qquad\qquad
\zeta\in(\Gamma_j\cap\dee)\cup\er,
\\
 J_{M^{(1)}}(\zeta)= &  H(\zeta)J_{M,j} H^{-1}(\zeta),\qquad\qquad
\zeta\in\Gamma_j\cap\overline{\dee}^c,\ j\not\in\{0,5\},
\\
J_{M^{(1)}}(\zeta)= &
\diag(\zeta^{\nu},\zeta^{\nu},\zeta^{-\nu},\zeta^{-\nu})A_j(\zeta)
H^{-1}(\zeta),\qquad \zeta\in\Sigma_j,\quad\textrm{$j$ even},
\\ J_{M^{(1)}}(\zeta)= &  H(\zeta)
A_j^{-1}(\zeta)\diag(\zeta^{-\nu},\zeta^{-\nu},\zeta^{\nu},\zeta^{\nu}),\qquad
\zeta\in\Sigma_j,\quad\textrm{$j$ odd},
\end{align*}
\item[(3)] $M^{(1)}(\zeta)=I+O(\zeta^{-1})$ for $\zeta\to\infty$.
\end{itemize}
\end{rhp}

We observe that $J_{M^{(1)}}-I$ decays exponentially as $\zeta\to\infty$ along
$\Sigma$. Also observe that the contour $\Sigma$ has $11$ points of
self-intersection, of which one is the origin and the other $10$ lie on the
unit circle. At each fixed point of self-intersection lying on the unit circle,
say $P$, order the contours that meet at $P$ counterclockwise, starting from
any contour that is oriented outwards from $P$. If we denote the limiting value
of the jump matrices over the $j$th contour at $P$ by $J_{M^{(1)},j}(P)$,
$j=1,\ldots,4$, then a little calculation shows that we have the cyclic
relation
$$ J_{M^{(1)},1}(P)J_{M^{(1)},2}^{-1}(P)J_{M^{(1)},3}(P)J_{M^{(1)},4}^{-1}(P) = I.
$$
The situation at the origin is trivial since all jump matrices there are the
identity.

Now we will obtain a factorization
$$ J_{M^{(1)}}(\zeta) =: (v_-(\zeta))^{-1}v_+(\zeta).
$$
Outside small neighborhoods of the $10$ points of self-intersection on
$\partial\dee$ we choose the trivial factorization $v_+=J_{M^{(1)}}$, $v_-=I$.
Using the cyclic relations we are then able \cite{FZ,Zhou} to choose a
factorization of $J_{M^{(1)}}(\zeta)$ in the remaining neighborhoods in such a
way that $v_+$ (or $v_-$) is continuous along the boundary of each connected
component of $\Omega_+$ (or $\Omega_-$, respectively). Standard arguments
\cite{FZ,Zhou} then imply the Fredholm property with Fredholm index zero.


\subsection{Existence of $M(\zeta)$ if $\tau=0$}
\label{subsection:exist:1}

Now we show the existence of $M(\zeta)$ if $r_1,r_2>0$, $s\in\er$ and $\tau=0$.
Thanks to the Fredholm property in the previous section, the existence follows
if we can prove that the \lq homogeneous version\rq\ of the RH
problem~\ref{rhp:modelM}, obtained by replacing the $I+O(\zeta^{-1})$ series in
\eqref{M:asymptotics} by $O(\zeta^{-1})$, has only the trivial solution
$M\equiv 0$. This approach is known as a \emph{vanishing lemma}
\cite{DKMVZ1,FZ,Zhou}. In the present setting it can be established in
essentially the same way as in \cite[Sec.~4]{DKZ}. There are a few minor
differences due to the $\nu$-dependence of the jump matrices, and the fact that
there is singular behavior at the origin.

For example, defining the products of jump matrices
$$ J_{M^{(4)}}:=J_8J_9J_0J_1J_2 = \begin{pmatrix}
1 & 0 & 1 & e^{\nu\pi i} \\
e^{-\nu\pi i}-e^{\nu\pi i} & 1 & e^{-\nu\pi i} & 1\\
0 & 0 & 1 & e^{\nu\pi i}-e^{-\nu\pi i} \\
0 & 0 & 0 & 1
\end{pmatrix},\quad x\in\er_+,
$$
$$ J_{M^{(4)}}:=(J_3J_4J_5J_6J_7)^{-1} = \begin{pmatrix}
1 & e^{\nu\pi i}-e^{-\nu\pi i} & 1 & e^{\nu\pi i} \\
0 & 1 & e^{-\nu\pi i} & 1\\
0 & 0 & 1 & 0 \\
0 & 0 & e^{-\nu\pi i}-e^{\nu\pi i} & 1
\end{pmatrix},\quad x\in\er_-,
$$
then a key property in the proof of the vanishing lemma is that, with
\begin{align*} \Theta(\zeta) &:=\diag(e^{\theta_1(\zeta)},e^{\theta_2(\zeta)},e^{-\theta_1(\zeta)},e^{-\theta_2(\zeta)}),& \zeta\in\cee,\\
J(x)&:=\Theta_-^{-1}(x)\, J_{M^{(4)}}\,\Theta_+(x)\begin{pmatrix}0 & -I_2\\
I_2 & 0\end{pmatrix},& x\in\er,
\end{align*}
we have the rank-one property
$$ J(x)+J^H(x) = 2\Theta_-^{-1}(x)\begin{pmatrix}1 & e^{-\nu\pi i} & 0 & 0\end{pmatrix}^T
\begin{pmatrix}1 & e^{\nu\pi i} & 0 & 0\end{pmatrix}\Theta_-^{-H}(x),\quad x\in\er,
$$
where the superscripts ${}^H$ and ${}^{-H}$ denote the conjugate transpose and
the inverse conjugate transpose, respectively. Using these formulas, it is
straightforward to adapt the proof of the vanishing lemma in
\cite[Sec.~4]{DKZ}.$\bol$

\begin{remark}
An alternative approach for proving the existence of $M$ if $r_1=r_2=1$,
$s\in\er$ and $\tau =0$, avoiding the vanishing lemma, is to use an analogue of
the approach from \cite[Sec.~5]{DG} that is outlined in
Section~\ref{subsection:exist:2} below. We should then keep $\tau=0$ fixed and
work with the system of differential equations \eqref{system:DG}, with the
second equation replaced by \eqref{diffeq2}. From Proposition~\ref{prop:large s
solvability} we already know that the RH matrix $M$ exists for $s\in\er$ large
enough. The outlined approach then allows to extend this conclusion to all
$s\in\er$.
\end{remark}

\subsection{Existence of $M(\zeta)$ if $\tau\neq 0$}
\label{subsection:exist:2}

Finally we prove the existence of $M(\zeta)$ if $r_1=r_2=:r$ and $s,\tau\in\er$
with $\tau\neq 0$. To this end we use the approach of Duits and Geudens
\cite[Sec.~5]{DG}. By rescaling we may assume that $r_1=r_2=1$. We consider the
system of two differential equations
\begin{equation}\label{system:DG}
\left\{\begin{array}{l}\frac{\partial M}{\partial\zeta}(\zeta,\tau) = U(\zeta,\tau)M(\zeta,\tau),\\
\frac{\partial M}{\partial\tau}(\zeta,\tau) =
W(\zeta,\tau)M(\zeta,\tau),\end{array}\right.
\end{equation}
with $U$ and $W$ as in \eqref{def:U:Lax} and \eqref{diffeq3} respectively,
where we now \emph{define}
\begin{equation}\label{DG:entries} \left\{\begin{array}{l}
d = 2^{-1/3}q(2^{2/3}(2s-\tau^2)),\\
c = -2^{-1/3}u(2^{2/3}(2s-\tau^2))+s^2,\\
b = \frac{1}{4\tau}\frac{\partial d}{\partial \tau}+cd+\tau d,\\
f=2\tau^2 d-\frac{c}{2\tau}\frac{\partial d}{\partial \tau}-c^2d-d^3-2sd+\nu,\\
h=\frac{1}{4\tau}\frac{\partial d}{\partial \tau}+cd-\tau d,\\
g+a=-c^2+d^2+s,
\end{array}\right.
\end{equation}
where $q$ is \emph{defined} as the Hastings-McLeod solution to Painlev\'e~II
and $u$ is the corresponding Hamiltonian. The formulas in \eqref{DG:entries}
are well-defined for every $\tau\in\er$ since the Hastings-McLeod function $q$
has no poles on the real line \cite{Claeys2}. These formulas are compatible
with the ones we already established. Indeed, for the entries $c$ and $d$ this
follows from Theorem~\ref{theorem:Painleve2modelrhp}.
The formula for $b$ follows from \eqref{compat:laxpair2:b}, and the formulas
for $f,g,h$ in \eqref{DG:entries} are then obtained from
\eqref{threeextrarelations:symm}--\eqref{threeextrarelations:symm3}.

We note that \eqref{DG:entries} is exactly the same as \cite[Eq.~(5.5)]{DG}
except for the extra term $\nu$ in the formula for $f$ and in
\eqref{def:Painleve2}.

The next two lemmas are proved in exactly the same way as in \cite{DG}.

\begin{lemma}\label{lemma:DG1}
(cf.~\cite[Lemma~5.1]{DG}). The matrices $U$ and $W$ defined above satisfy the
compatibility relation \eqref{compat:laxpair2}.
\end{lemma}

\begin{lemma}\label{lemma:DG2}
(cf.~\cite[Lemma~5.2]{DG}). Fix $\tau\in\er$ and let $U$ be as defined above.
Let $\Sigma$ be one of the four complex sectors
\begin{align*} \Sigma_1 = \{\zeta\in\cee\mid 0<\arg\zeta<2\pi/3+\epsilon\},\quad \Sigma_2=-\overline{\Sigma_1},
\quad \Sigma_3=-\Sigma_1,\quad \Sigma_4=\overline{\Sigma_1},
\end{align*}
where $\epsilon>0$ is fixed and small and where the bar denotes the complex
conjugation. Then the equation $\frac{\partial N}{\partial\zeta}(\zeta,\tau) =
U(\zeta,\tau)N(\zeta,\tau)$ has a unique fundamental solution $N$ in the sector
$\Sigma$ with the following asymptotics as $\zeta\to\infty$ within this sector,
\begin{multline}
\label{M:asymptotics:DG} N(\zeta) =
\left(I+\frac{N_1}{\zeta}+\frac{N_2}{\zeta^2}+O\left(\frac{1}{\zeta^3}\right)\right)
\diag((-\zeta)^{-1/4},\zeta^{-1/4},(-\zeta)^{1/4},\zeta^{1/4})
\\
\times
\Aa\diag\left(e^{-\theta_1(\zeta)+\tau\zeta},e^{-\theta_2(\zeta)-\tau\zeta},e^{\theta_1(\zeta)+\tau\zeta},e^{\theta_2(\zeta)-\tau\zeta}\right),
\end{multline}
with $\Aa,\theta_1,\theta_2$ given in
\eqref{mixing:matrix}--\eqref{def:theta1}, and with $N_1$ of the form
$$ N_1 = \begin{pmatrix}
* & b&*&id \\
-b & * & id & * \\
* & if & * & h \\
if & * & -h & *
\end{pmatrix}.
$$
The asymptotics in \eqref{M:asymptotics:DG} are uniform within any closed
sector of $\Sigma$ (away from $\er$).
\end{lemma}

Following \cite[Sec.~5]{DG}, the fundamental solutions to the differential
equation $\frac{\partial N}{\partial\zeta} = UN$ in the four sectors in
Lemma~\ref{lemma:DG2} can be used to construct a matrix function
$M=M(\zeta,\tau)$ satisfying the correct jumps and asymptotics for
$\zeta\to\infty$ in the RH problem~\ref{rhp:modelM}, and satisfying also the
second differential equation in \eqref{system:DG}. Using the latter
differential equation, we find that the matrices $E(\zeta)$ in
Proposition~\ref{prop:behaviorMzero} satisfy  $\frac{\partial E}{\partial\tau}
= WE$. Since $W$ is analytic and uniformly bounded and we already know that
$E(\zeta)$ is analytic for $\tau=0$ (due to
Proposition~\ref{prop:behaviorMzero} and the existence established in
Section~\ref{subsection:exist:1}), we then find that $E(\zeta)$ is analytic for
all $\tau\in\er$. This shows that $M$ has the correct behavior at the origin in
the RH problem~\ref{rhp:modelM}. $\bol$


\section{Proof of Theorem~\ref{theorem:kernelpsi}}
\label{section:steepest:DKRZ}

In this section we prove Theorem~\ref{theorem:kernelpsi} on the critical
behavior of the non-intersecting squared Bessel paths near the hard-edge
tacnode. The proof will follow from a Deift-Zhou steepest descent analysis of
the RH problem~\ref{rhp:Y} for $Y(z)$. The main technical step will be the
construction of the local parametrix near the origin and this is where the
model RH problem~\ref{rhp:modelM} for $M(\zeta)$ will be used.

In what follows we consider $n$ non-intersecting squared Bessel paths with
fixed endpoints $a,b$ satisfying \eqref{doublescaling:ab}. We assume that $n$
is even. We study these paths at an $n$-dependent time $t$ and temperature $T$
which vary according to the triple scaling limit
\eqref{doublescaling:t}--\eqref{doublescaling:T}.

\subsection{Modified $\lam$-functions} \label{section:Riemannsurface}

First we define some auxiliary objects that will be needed during the steepest
descent analysis. Inspired by \cite[Sec.~6.3]{DKZ}, define
\begin{align}\label{modlam:alpha}\alpha:=(1-t)\sqrt{a}+t\sqrt{b}-\sqrt{2t(1-t)T},\\
\label{modlam:beta}\beta:=(1-t)\sqrt{a}+t\sqrt{b}+\sqrt{2t(1-t)T},
\end{align}
and
\begin{align}
\label{modlam:gamma}\gamma := \left(\alpha+\beta+2\sqrt{\alpha^2+\beta^2-\alpha\beta}\right)/3,\\
\label{modlam:delta}\delta :=
\left(\alpha+\beta-\sqrt{\alpha^2+\beta^2-\alpha\beta}\right)/3.
\end{align}
A little calculation shows that under the triple scaling assumptions
\eqref{doublescaling:ab}--\eqref{doublescaling:T}, we have for $n\to\infty$
that
\begin{align}\begin{array}{l}
\label{asy:alpha:delta}
\alpha = \frac{K^2(\sqrt{a}+\sqrt{b})^4
-L}{2(\sqrt{a}+\sqrt{b})}n^{-2/3}+O(n^{-1}),
\\
\beta=\frac{2}{\sqrt{a}+\sqrt{b}}+O(n^{-1/3}),\\
\gamma=\frac{2}{\sqrt{a}+\sqrt{b}}+O(n^{-1/3}),\\
\delta = \frac{K^2(\sqrt{a}+\sqrt{b})^4
-L}{4(\sqrt{a}+\sqrt{b})}n^{-2/3}+O(n^{-1}).
\end{array}
\end{align}

\begin{lemma}\label{lemma:lambda:DKZ} (\cite[Sec.~6.3]{DKZ}).
There exists an analytic function $\wtil\lam:\cee\setminus [0,\infty)\to\cee$
such that
\begin{equation} \label{var:ineq1}
    \Re \wtil\lambda_{+}(x)  =\Re \wtil\lambda_{-}(x)
        \begin{cases} =0, & x\in (0,\gamma], \\
        > 0, & x\in(\gamma,\infty), \end{cases}
        \end{equation}
\begin{equation} \label{var:ineq3}
    \Im \wtil\lambda_{+}(x)=-\Im \wtil\lambda_{-}(x)\begin{cases} >0, & x\in (\epsilon,\gamma], \\
        = \pi/2, & x\in(\gamma,\infty), \end{cases}
    \end{equation}
with $\epsilon$ a number that goes to zero as $O(n^{-1/3})$ as $n\to\infty$;
moreover
\begin{align}\label{lambda12:asy}
    \wtil\lambda(z)&= \frac{z^2}{4t(1-t)T}-\frac{\sqrt{a}}{2tT}z-\frac{\sqrt{b}}{2(1-t)T}z
      -\frac 12 \log (-z)+\ell+O(z^{-1}),\quad z\to\infty,
\end{align}
for a certain constant $\ell$, and
\begin{equation}\label{lambdatil:zero} \wtil\lam(z)= \frac{\delta\sqrt{\gamma}}{t(1-t)T}
(-z)^{1/2}+ (-z)^{3/2}G(z),
\end{equation}
where $G(z)$ is an analytic function in a neighborhood of $z=0$ which satisfies
\begin{equation}\label{G:zero} G(0) = \frac{\sqrt{\gamma}}{3t(1-t)}+O(n^{-2/3}),\qquad n\to\infty.
\end{equation}
\end{lemma}

\begin{proof} Take $\wtil \lam(z)$ to be the function which is called $\lam_1(z)$ in \cite[Sec.~6.3]{DKZ}
with $a_1:=\sqrt{a}$, $b_1:=\sqrt{b}$ and $p_1:=\frac 12$. Note that the
temperature $T\equiv 1$ in \cite{DKZ} but the extension to a nontrivial
temperature is straightforward.
\end{proof}

Now we construct the functions \begin{equation}\begin{array}{l}
\label{lambda:1234}\lam_1(z) :=
2\wtil\lam(-\sqrt{z})+\frac{z}{2t(1-t)T}+\frac{\sqrt{az}}{tT}-\frac{\sqrt{bz}}{(1-t)T}+2\ell,\\
 \lam_2(z) :=
2\wtil\lam(\sqrt{z})+\frac{z}{2t(1-t)T}-\frac{\sqrt{az}}{tT}+\frac{\sqrt{bz}}{(1-t)T}+2\ell\mp\pi i,\\
\lam_3(z) :=  -2\wtil\lam(\sqrt{z})+\frac{z}{2t(1-t)T}-\frac{\sqrt{a z}}{tT}+\frac{\sqrt{bz}}{(1-t)T}+2\ell\pm\pi i,\\
\lam_4(z) := -2\wtil\lam(-\sqrt{z})+\frac{z}{2t(1-t)T}+\frac{\sqrt{a
z}}{tT}-\frac{\sqrt{bz}}{(1-t)T}+2\ell,
\end{array}\end{equation}
for $\pm \Im z>0$.

\begin{lemma}\label{lemma:lambda:asy}
The $\lam$-functions have the following asymptotics for $z\to\infty$:
$$\begin{array}{l} \lam_1(z) =
\frac{z}{t(1-t)T}+\frac{2\sqrt{az}}{tT}+4\ell-\frac{1}{2}\log z-\frac{2c_1}{\sqrt{z}}+O(z^{-1}),\\
\lam_2(z) =
\frac{z}{t(1-t)T}-\frac{2\sqrt{az}}{tT}+4\ell-\frac{1}{2}\log z+\frac{2c_1}{\sqrt{z}}+O(z^{-1}),\\
\lam_3(z) =
\frac{2\sqrt{bz}}{(1-t)T}+\frac{1}{2}\log z-\frac{2c_3}{\sqrt{z}}+O(z^{-1}),\\
\lam_4(z) = -\frac{2\sqrt{bz}}{(1-t)T} + \frac{1}{2}\log z+
\frac{2c_3}{\sqrt{z}}+O(z^{-1}),
\end{array}$$
for certain constants $c_1$ and $c_3$. 
\end{lemma}

\begin{proof} Immediate from the definitions.
\end{proof}

Note that the asymptotics in Lemma~\ref{lemma:lambda:asy} are exactly the same
as for the $\lam$-functions in \cite[Lemma~4.8]{DKRZ}. Also note that
\begin{equation}\label{t:zero} \frac{\sqrt{a}}{tT}-\frac{\sqrt{b}}{(1-t)T} = -2K(\sqrt{a}+\sqrt{b})^3 n^{-1/3}+O(n^{-2/3}),\qquad
n\to\infty,
\end{equation}
due to our triple scaling assumptions
\eqref{doublescaling:ab}--\eqref{doublescaling:T}.

In what follows we write, with \eqref{modlam:gamma},
\begin{equation}\label{q:endpoint:mod}q:=\gamma^2.\end{equation}

\begin{lemma}\label{lemma:lambda:var} We have
$$
\begin{array}{ll}
\lam_{1,\pm}(x)=\lam_{2,\mp}(x)\mp\pi i,\qquad & x\in \er_-,\\
\lam_{2,\pm}(x)=\lam_{3,\mp}(x),\qquad& x\in [0,q],\\
\lam_{3,\pm}(x)=\lam_{4,\mp}(x)\pm\pi i,\qquad& x\in \er_-,
\end{array}$$
and $$\begin{array}{ll}
\Re(\lam_{2}(x)-\lam_{3}(x))>0,\qquad& x\in (q,\infty).
\end{array}$$
There is a Jordan curve $\Delta_2^+$ running from $0$ to $q$ in the upper half place so that
$$ \Re(\lam_2(z)-\lam_3(z))<0,\qquad z\in \Delta_{2}^{\pm}\setminus B_{\rho},$$
where $\Delta_2^-=\overline{\Delta_2^+}$ and $B_{\delta}$ denotes the disk
around the origin with radius $\rho>0$ sufficiently small. Similarly, there is
a Jordan curve $\Delta_1^+$ running from $\infty$ to $0$ in the second quadrant
of the plane so that, with $\Delta_1^-=\overline{\Delta_1^+}$ and $\rho>0$
sufficiently small,
$$\Re(\lam_1(z)-\lam_2(z))>0,\ \ \Re(\lam_3(z)-\lam_4(z))>0,\qquad z\in \Delta_1^{\pm}\setminus B_{\rho}.$$
\end{lemma}

\begin{proof} The equalities follow from the definitions. The inequalities are
intrinsic in our definition of $\wtil\lam$ in the proof of
Lemma~\ref{lemma:lambda:DKZ}, with the help of \cite[Lemmas~7.8, 7.9]{DKZ}.
\end{proof}

An illustration of the Jordan curves $\Delta_1^{\pm}$ and $\Delta_2^{\pm}$ in shown in Figure~\ref{fig:squaredBessel:rhpT}.

\subsection{The transformations $Y\mapsto X\mapsto U\mapsto T$}
\label{section:transfoYT}

In \cite{DKRZ}, it is shown how to apply a steepest descent analysis to the RH
problem~\ref{rhp:Y} for $Y(z)$. The authors apply there a series of
transformations
$Y\mapsto X\mapsto U\mapsto T$. 
We can apply exactly the same transformations here except that we replace the
$\lam$-functions in \cite{DKRZ} by the modified $\lam$-functions
\eqref{lambda:1234}.
%
%
%
The resulting matrix $T$ satisfies the following RH problem (see Case~I in
\cite[Proposition~5.4]{DKRZ}):

\begin{figure}[t]
\vspace{-13mm}
\begin{center}
   \setlength{\unitlength}{1truemm}
   \begin{picture}(100,70)(-5,2)
       \put(40,40){\line(1,0){55}}
       \put(40,40){\line(-1,0){40}}
       \put(93,36.6){$\er$}
       \put(95,40){\thicklines\vector(1,0){1}}
       \put(40,40){\thicklines\circle*{1}}
       \put(39.3,36){$0$}
       \qbezier(40,40)(30,50)(0,55)
       \qbezier(40,40)(30,30)(0,25)
       \put(18,53){$\Delta_1^{+}$}
       \put(18,25){$\Delta_1^{-}$}
       \put(20,50.3){\thicklines\vector(4,-1){1}}
       \put(20,29.6){\thicklines\vector(4,1){1}}
       \put(20,40){\thicklines\vector(1,0){1}}
       \put(60,40){\thicklines\circle*{1}}
       \put(59.3,36){$q$}
       \qbezier(40,40)(50,55)(60,40)
       \qbezier(40,40)(50,25)(60,40)
       \put(48,49){$\Delta_2^+$}
       \put(48,28){$\Delta_2^-$}
       \put(50,47.5){\thicklines\vector(1,0){1}}
       \put(50,32.5){\thicklines\vector(1,0){1}}
       \put(50,40){\thicklines\vector(1,0){1}}

  \end{picture}
   \vspace{-21mm}
   \caption{Jump contours in the RH problem for $T(z)$ in the analysis of the non-intersecting squared Bessel
   paths in the multi-critical case.}
   \label{fig:squaredBessel:rhpT}
\end{center}
\end{figure}
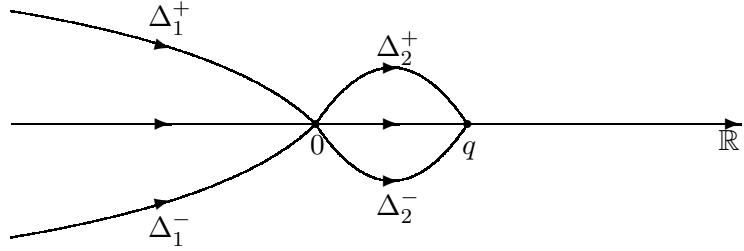

\begin{rhp}\label{rhp:T} We look for a $4\times 4$ matrix valued function $T(z)$ satisfying
\begin{itemize}
\item[(1)] $T(z)$ is analytic for $z\in\cee\setminus\left(\er\cup\Delta_1^{\pm}\cup\Delta_2^{\pm}\right)$,
where the contours are as in Figure~\ref{fig:squaredBessel:rhpT}.
\item[(2)] For each of the oriented contours in Figure~\ref{fig:squaredBessel:rhpT},
$T$ has a jump $T_+(z) = T_-(z)J_T(z)$ where \begin{align*} J_T(x) &=
\diag\left(1,\begin{pmatrix}0&1\\-1&0\end{pmatrix},1\right), \qquad \textrm{
for }x\in (0,q),
\\ J_T(x) &= I+e^{-n(\lam_{2,+}(x)-\lam_{3,+}(x))}E_{2,3},\qquad \textrm{ for } x\in\er_+\setminus [0,q],
\\ J_T(x) &=
\diag\left(\begin{pmatrix} 0&1 \\ -1&0\end{pmatrix},\begin{pmatrix} 0&1 \\
-1&0\end{pmatrix}\right), \qquad \textrm{ for }x\in (-\infty,0),
\\ J_T(z) &= I+e^{n(\lam_2(z)-\lam_3(z))}E_{3,2},\qquad \textrm{ for }z\in\Delta_2^{\pm},
\\ J_T(z) &= I-e^{\pm \alpha\pi i}e^{-n(\lam_1(z)-\lam_2(z))}E_{1,2}-e^{\pm \alpha\pi
i}e^{-n(\lam_3(z)-\lam_4(z))}E_{3,4},\qquad \textrm{ for } z\in\Delta_1^{\pm},
\end{align*}
where the notation $E_{j,k}$ is defined in Section~\ref{subsection:AtoB}.
\item[(3)] As $z\to\infty$ we have
\begin{equation*} T(z) =
\left(I+O\left(\frac{1}{z}\right)\right)\diag(z^{-1/4},z^{1/4},z^{-1/4},z^{1/4})
\frac{1}{\sqrt{2}}\diag\left(\begin{pmatrix} 1&i \\ i &
1\end{pmatrix},\begin{pmatrix} 1&i \\ i & 1\end{pmatrix}\right)
\end{equation*}
uniformly for $z\in\cee\setminus\er$.
\item[(4)] $T(z)$ behaves for $z\to 0$ as
 \begin{equation}\label{zero behavior of T} \begin{array}{ll}
  T(z)\diag(|z|^{\alpha/2},|z|^{-\alpha/2},|z|^{\alpha/2},|z|^{-\alpha/2}))=O(1),&
\qquad \textrm{if $\alpha>0$},\\
  T(z)\diag((\log |z|)^{-1},1,(\log |z|)^{-1},1)=O(1),&\qquad \textrm{if $\alpha=0$},
  \\
  T(z)=O(z^{\alpha/2}),~~T^{-1}(z)=O(z^{\alpha/2}),&\qquad \textrm{if $-1<\alpha<0$,}
  \end{array}\end{equation}
  where we assume that $z\to 0$ in the region between $\Delta_{1}^{+}$ and $\Delta_{2}^{+}$.
\end{itemize}
\end{rhp}

\subsection{Global parametrix} \label{section:global}

The global parametrix $P^{(\infty)}(z)$ is the solution to the RH problem
obtained by setting all the exponentially small entries in the RH problem for
$T(z)$ equal to zero (see Lemma~\ref{lemma:lambda:var}):
\begin{rhp}\label{rhp:Pinfty:alpha} We look for a $4\times 4$ matrix valued function $P^{(\infty)}(z)$ satisfying
\begin{itemize}
\item[(1)] $P^{(\infty)}(z)$ is analytic for $z\in\cee\setminus(\er_-\cup [0,q])$.
\item[(2)] For $x\in\er_-\cup [0,q]$ we have the jump
\begin{align*} P^{(\infty)}_+(x) = P^{(\infty)}_-(x)
\diag\left(1,\begin{pmatrix}0&1\\-1&0\end{pmatrix},1\right), \quad x\in (0,q),
\\ P^{(\infty)}_+(x) = P^{(\infty)}_-(x)\diag\left(\begin{pmatrix} 0&1 \\ -1&0\end{pmatrix},\begin{pmatrix} 0&1 \\
-1&0\end{pmatrix}\right),\quad x\in\er_-.
\end{align*}
\item[(3)] As $z\to\infty$ we have uniformly for $z\in\cee\setminus\er$,
\begin{equation*} P^{(\infty)}(z) =
\left(I+O\left(\frac{1}{z}\right)\right)\diag(z^{-1/4},z^{1/4},z^{-1/4},z^{1/4})
\frac{1}{\sqrt{2}}\diag\left(\begin{pmatrix} 1&i \\ i & 1\end{pmatrix},
\begin{pmatrix} 1&i \\ i & 1\end{pmatrix}\right).
\end{equation*}
\item[(4)] $P^{(\infty)}(z)$ has at most fourth-root singularities at the special points
$q$ and $0$.
\end{itemize}
\end{rhp}

The global parametrix was constructed in a more general setting in
\cite[Sec.~5.4]{DKRZ}. In the present case the construction can be made fully
explicit:

\begin{lemma}
The solution to the RH problem \ref{rhp:Pinfty:alpha} for $P^{(\infty)}$ is
given explicitly by
\begin{multline}\label{Pinfty:explicit:tilde}
P^{(\infty)}(z) = C\diag(z^{-1/4},z^{1/4},z^{-1/4},z^{1/4}) \frac{1}{\sqrt{2}}
\diag\left(\begin{pmatrix} 1&i \\ i & 1\end{pmatrix},
\begin{pmatrix} 1&i \\ i & 1\end{pmatrix}\right)
\\
\times \wtil \Aa^{-1}
\diag(\gamma_2^{-1}(\sqrt{z}),\gamma_1^{-1}(\sqrt{z}),\gamma_1(\sqrt{z}),\gamma_2(\sqrt{z}))
\wtil A,
\end{multline}
where
\begin{equation}\label{mixing:matrix2}
\wtil \Aa:=\frac{1}{\sqrt{2}}\begin{pmatrix} 1 & 0 & 0 & i \\
0 & 1 & -i & 0 \\
0 & -i & 1 & 0 \\
i & 0 & 0 & 1 \\
\end{pmatrix},
\end{equation}
\begin{equation}\label{Pinfty:gammas}
 C := I+\frac{i\sqrt{q}}{4}(E_{2,3}-E_{4,1}),\qquad \gamma_1(z):=\left(\frac{z}{z-\sqrt{q}}\right)^{1/4},\quad
\gamma_2(z):=\left(\frac{z}{z+\sqrt{q}}\right)^{1/4}.
\end{equation}
\end{lemma}

\begin{proof}
Straightforward verification.
\end{proof}

\subsection{Local parametrix at the point $q$} \label{subsection:localAiry}

Inside a fixed small disk around the point $q$ we construct a local parametrix
$P^{(q)}(z)$ to the RH problem with the help of Airy functions. Such a
construction is standard and well-known and we do not go into the details.

\subsection{Local parametrix at the origin} \label{section:PhaseII}

In this section we construct a local parametrix near the origin. To this end we
will use the model RH problem for $M(\zeta)$. The construction will also use
the \lq squaring trick\rq\ of Its et al. \cite{IKO}.

\subsubsection*{Transformation of the Riemann-Hilbert problem for $M(\zeta)$}

Recall the model RH problem~\ref{rhp:modelM} for $M(\zeta)$. We put
$\nu:=\alpha+1/2$ and we set \begin{equation}\label{N:def} N(\zeta) :=
D^{-1}\diag\left(\begin{pmatrix}0&1\\
1&0\end{pmatrix},1,1\right)M(\zeta)\diag\left(\begin{pmatrix}0&1\\
1&0\end{pmatrix},1,1\right)D,
\end{equation}
where
$$ D = \diag(-i,1,1,-i).
$$
The jumps for $N$ are shown in Figure~\ref{fig:modelRHP:2}.

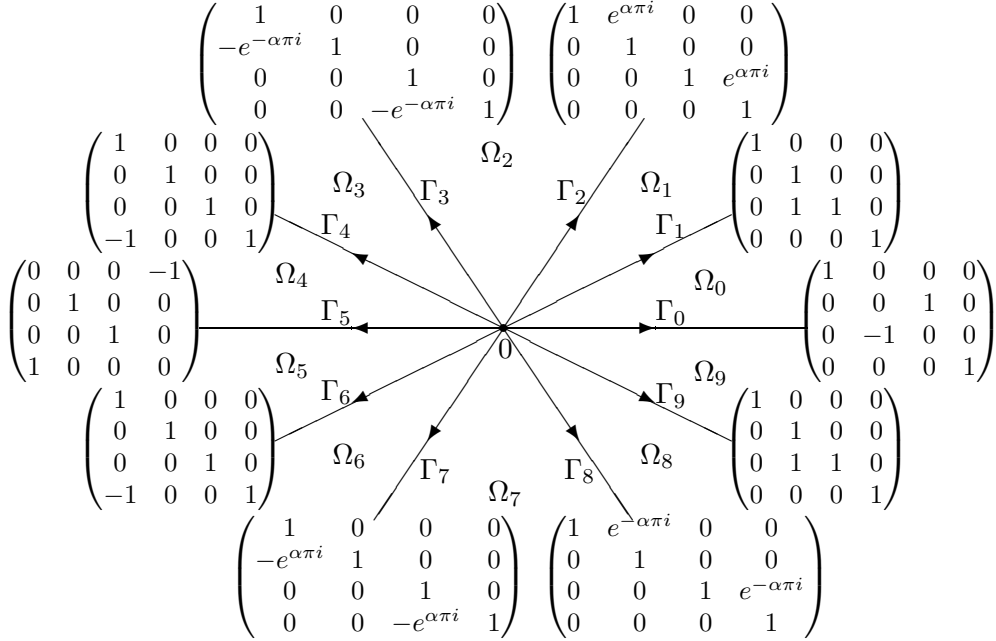
\begin{figure}[t]
\vspace{14mm}
\begin{center}
   \setlength{\unitlength}{1truemm}
   \begin{picture}(100,70)(-5,2)
       \put(40,40){\line(1,0){40}}
       \put(40,40){\line(-1,0){40}}
       \put(40,40){\line(2,1){30}}
       \put(40,40){\line(2,-1){30}}
       \put(40,40){\line(-2,1){30}}
       \put(40,40){\line(-2,-1){30}}
       \put(40,40){\line(2,3){18.5}}
       \put(40,40){\line(2,-3){17}}
       \put(40,40){\line(-2,3){18.5}}
       \put(40,40){\line(-2,-3){17}}
       \put(40,40){\thicklines\circle*{1}}
       \put(39.3,36){$0$}
       \put(60,40){\thicklines\vector(1,0){.0001}}
       \put(20,40){\thicklines\vector(-1,0){.0001}}
       \put(60,50){\thicklines\vector(2,1){.0001}}
       \put(60,30){\thicklines\vector(2,-1){.0001}}
       \put(20,50){\thicklines\vector(-2,1){.0001}}
       \put(20,30){\thicklines\vector(-2,-1){.0001}}
       \put(50,55){\thicklines\vector(2,3){.0001}}
       \put(50,25){\thicklines\vector(2,-3){.0001}}
       \put(30,55){\thicklines\vector(-2,3){.0001}}
       \put(30,25){\thicklines\vector(-2,-3){.0001}}

       \put(60,41){$\Gamma_0$}
       \put(60,52.5){$\Gamma_1$}
       \put(47,57){$\Gamma_2$}
       \put(29,57){$\Gamma_3$}
       \put(16,52.5){$\Gamma_4$}
       \put(16,41){$\Gamma_5$}
       \put(16,30.5){$\Gamma_6$}
       \put(29,20){$\Gamma_7$}
       \put(48,20){$\Gamma_8$}
       \put(60,30){$\Gamma_{9}$}

       \put(65,45){$\small{\Omega_0}$}
       \put(58,58){$\small{\Omega_1}$}
       \put(37,62){$\small{\Omega_2}$}
       \put(17.5,58){$\small{\Omega_3}$}
       \put(10,46){$\small{\Omega_4}$}
       \put(10,34){$\small{\Omega_5}$}
       \put(17.5,22){$\small{\Omega_6}$}
       \put(38,17){$\small{\Omega_7}$}
       \put(58,22){$\small{\Omega_8}$}
       \put(65,33){$\small{\Omega_{9}}$}

       \put(78.5,40){$\small{\begin{pmatrix}1&0&0&0\\ 0&0&1&0\\ 0&-1&0&0\\ 0&0&0&1 \end{pmatrix}}$}
       \put(69,57){$\small{\begin{pmatrix}1&0&0&0\\ 0&1&0&0\\ 0&1&1&0\\ 0&0&0&1 \end{pmatrix}}$}
       \put(45,74){$\small{\begin{pmatrix}1& e^{\al\pi i}&0&0\\ 0&1&0&0\\ 0&0&1&e^{\al\pi i}\\ 0&0&0&1 \end{pmatrix}}$}
       \put(-1,74){$\small{\begin{pmatrix}1&0&0&0\\ - e^{-\al\pi i}&1&0&0\\ 0&0&1&0\\ 0&0& -e^{-\al\pi i}&1 \end{pmatrix}}$}
       \put(-16,57){$\small{\begin{pmatrix}1&0&0&0\\ 0&1&0&0\\ 0&0&1&0\\ -1&0&0&1\end{pmatrix}}$}
       \put(-26,40){$\small{\begin{pmatrix}0&0&0&-1\\ 0&1&0&0\\ 0&0&1&0\\ 1&0&0&0 \end{pmatrix}}$}
       \put(-16,23){$\small{\begin{pmatrix}1&0&0&0\\ 0&1&0&0\\ 0&0&1&0\\ -1&0&0&1 \end{pmatrix}}$}
       \put(4,6){$\small{\begin{pmatrix}1&0&0&0\\ - e^{\al\pi i}&1&0&0\\ 0&0&1&0\\ 0&0& -e^{\al\pi i}&1 \end{pmatrix}}$}
       \put(45,6){$\small{\begin{pmatrix}1& e^{-\al\pi i}&0&0\\ 0&1&0&0\\ 0&0&1& e^{-\al\pi i}\\ 0&0&0&1 \end{pmatrix}}$}
       \put(69,23){$\small{\begin{pmatrix}1&0&0&0\\ 0&1&0&0\\ 0&1&1&0\\ 0&0&0&1\end{pmatrix}}$}

  \end{picture}
   \vspace{0mm}
   \caption{Jump
   matrices in the RH problem for  $N = N(\zeta)$.}
   \label{fig:modelRHP:2}
\end{center}
\end{figure}

The asymptotics of $N$ as $\zeta \to \infty$ is given by
\begin{multline}
\label{N:asymptotics} N(\zeta) =
\left(I+\frac{N_1}{\zeta}+\frac{N_2}{\zeta^2}+O\left(\frac{1}{\zeta^3}\right)\right)
\diag(\zeta^{-1/4},(-\zeta)^{-1/4},(-\zeta)^{1/4},\zeta^{1/4})
\\
\times \wtil
\Aa\diag\left(e^{-\theta_2(\zeta)-\tau\zeta},e^{-\theta_1(\zeta)+\tau\zeta},e^{\theta_1(\zeta)+\tau\zeta},e^{\theta_2(\zeta)-\tau\zeta}\right)
\end{multline}
with $\wtil \Aa$ as in \eqref{mixing:matrix2}. The behavior around infinity can
be rewritten as
\begin{multline}
\label{N:asymptotics2} N(\zeta) =
\diag(\zeta^{-1/4},(-\zeta)^{-1/4},(-\zeta)^{1/4},\zeta^{1/4}) \wtil \Aa
\left(I+\frac{\wtil N_{1,\pm}}{\zeta^{1/2}}+\frac{\wtil
N_{2,\pm}}{\zeta}+O\left(\frac{1}{\zeta^{3/2}}\right)\right) \\ \times
\diag\left(e^{-\theta_2(\zeta)-\tau\zeta},e^{-\theta_1(\zeta)+\tau\zeta},e^{\theta_1(\zeta)+\tau\zeta},e^{\theta_2(\zeta)-\tau\zeta}\right)
\end{multline}
with the $\pm$ sign as $\zeta \to \infty$ within the upper/lower half plane.
Here
\[
\wtil N_{1,\pm}=\wtil \Aa^{-1} \diag \left(1,e^{\mp \pi i/4},0,0 \right) N_1
\diag \left(0,0,e^{\mp \pi i/4},1 \right) \wtil \Aa.
\]

For further use, we record the symmetry relation
\begin{multline}\label{symmetry:N}
N(-\zeta;r_1,r_2,s,\tau) = \diag\left(\begin{pmatrix} 0 & 1 \\ -1 & 0
\end{pmatrix},\begin{pmatrix} 0 & 1 \\ -1 & 0
\end{pmatrix}\right)\\
\times N(\zeta; r_2,r_1,s,\tau)\diag\left(\begin{pmatrix} 0 & -1 \\ 1 & 0
\end{pmatrix},\begin{pmatrix} 0 & -1 \\ 1 & 0
\end{pmatrix}\right).
\end{multline}
This easily follows from \eqref{symmetry:special}. Note that the order of $r_1$
and $r_2$ differs on both sides of the equality.

\subsubsection*{Construction of the local parametrix}

Now we construct the local parametrix $P^{(0)}$ around the origin. In \cite{DKZ} we constructed
$P^{(0)}$ inside a shrinking disk $B_{\rho}$ of radius $\rho=n^{-1/3}$ around the origin, see also
\cite{DG}. In the present setting, we will be able to work inside a \emph{fixed} disk $B_{\rho}$ with radius $\rho>0$
fixed but sufficiently small. The fact that we have a fixed (rather than a
shrinking) disk will greatly simplify some of the technical details.

\begin{rhp}\label{rhp:P0} We look for a $4\times 4$ matrix valued function $P^{(0)}$ such that
\begin{itemize}
\item[(1)] $P^{(0)}(z)$ is analytic for $z\in B_\rho\setminus
(\er\cup\Delta_1^{\pm}\cup \Delta_2^{\pm})$.
\item[(2)] For $z\in B_\rho\cap (\er\cup\Delta_1^{\pm}\cup \Delta_2^{\pm})$,
$P^{(0)}$ has the same jumps as $T$, see RH problem~\ref{rhp:T}.
\item[(3)] Uniformly on the circle $|z|=\rho$ we have for $n\to
\infty$ that
\begin{equation}\label{matching:P0}
P^{(0)}(z)=P^{(\infty)}(z)(I+O(n^{-1/3})).\end{equation}
 \item[(4)] The behavior of $P^{(0)}(z)$ for $z\to 0$ is the same as for
 $T(z)$, see \eqref{zero behavior of T}.
\end{itemize}
\end{rhp}

We need some functions $f(z)$, $r_1(z)$, $r_2(z)$ and constants $s$, $\tau$. We will
use a significantly simpler construction that in \cite{DKZ} and \cite{DG}. Note that there
are many parameters $f,r_1,r_2,s,\tau$ and there is a certain freedom in how to define them.

We set (recall \eqref{def:kappa})
\begin{align}\label{r12stau}
& \kappa=2(\sqrt{a}+\sqrt{b}),\qquad f(z)=\kappa\sqrt{z},\qquad r_1(z)=3 G(\sqrt{z})\kappa^{-3/2},\\
\nonumber & r_2(z)=3 G(-\sqrt{z})\kappa^{-3/2},\qquad
s=\frac{\delta\sqrt{\gamma}}{t(1-t)T}\kappa^{-1/2},\qquad \tau =
\left(\frac{\sqrt{a}}{tT}-\frac{\sqrt{b}}{(1-t)T}\right)\kappa^{-1},
\end{align}
recall \eqref{lambdatil:zero}. Thanks to \eqref{lambdatil:zero} and
\eqref{lambda:1234} we have
\begin{align}\label{lambdas:modelrhp}\begin{array}{rl} -\frac{2}{3}r_2(z)f(z)^{3/2}-2sf(z)^{1/2} -\tau f(z) &=
-\lam_1(z)+h(z),\\
-\frac{2}{3}r_1(z)(-f(z))^{3/2}-2s(-f(z))^{1/2} +\tau f(z)&=-\lam_2(z)+h(z)\mp\pi i,\\
\frac{2}{3}r_1(z)(-f(z))^{3/2}+2s(-f(z))^{1/2} +\tau f(z)&=-\lam_3(z)+h(z)\pm\pi i,\\
\frac{2}{3}r_2(z)f(z)^{3/2}+2sf(z)^{1/2} -\tau
f(z)&=-\lam_4(z)+h(z),\end{array}\end{align} with
$h(z):=\frac{z}{2t(1-t)T}+2\ell$.

The following limits exist:
\begin{align}\label{sn:limit}
&s^*:=\lim_{n\to\infty} n^{2/3}s = \lim_{n\to\infty}
n^{2/3}\frac{\delta\sqrt{\gamma}\kappa^{-1/2}}{t(1-t)}=\frac{K^2(\sqrt{a}+\sqrt{b})^4
-L}{2},\\
\nonumber &\tau^*:=\lim_{n\to\infty} n^{1/3}\tau = -K(\sqrt{a}+\sqrt{b})^2.
\end{align}
The first formula follows from
\eqref{doublescaling:ab}--\eqref{doublescaling:T} and \eqref{asy:alpha:delta}.
The second one follows from \eqref{t:zero}. From \eqref{G:zero} we also find
that for $z$ in a neighborhood of the origin,
$$ \lim_{n\to\infty} r_j(z)= \frac{\sqrt{\gamma}\kappa^{-3/2}}{t(1-t)}+O(z^{1/2})= 1+O(z^{1/2}),\qquad j=1,2.
$$

Now we define the local parametrix $P^{(0)}(z)$ at the origin by
\begin{equation}\label{def:P0} P^{(0)}(z) = E(z)N(n^{2/3}f(z);r_1(z),r_2(z),n^{2/3}s,n^{1/3}\tau)\Lam(z),
\end{equation}
with $N(\zeta)$ defined in \eqref{N:def}, and with
\begin{align}\label{def:Ez}
E(z) = P^{(\infty)}(z) \wtil
\Aa^{-1}\diag(\zeta^{1/4},(-\zeta)^{1/4},(-\zeta)^{-1/4},\zeta^{-1/4}),\qquad
\zeta:= n^{2/3}f(z),\\
\label{def:Lam}
\Lam(z)=\exp\left(-n\left(\frac{z}{2t(1-t)T}+2\ell\right)\right)\diag(e^{n\lam_1(z)},e^{n\lam_2(z)},e^{n\lam_3(z)},e^{n\lam_4(z)}).
\end{align}

\begin{lemma}\label{lemma:jumpsE}
The matrix-valued function $E(z)$ in \eqref{def:Ez} is analytic for
$z\in\cee\setminus\er_-$. For $x<0$ it has the jump
\begin{equation}\label{jump:E} E_+(x) = E_-(x)\diag\left(\begin{pmatrix} 0&1 \\ -1&0\end{pmatrix},\begin{pmatrix} 0&1 \\
-1&0\end{pmatrix}\right),\qquad x<0.
\end{equation}
\end{lemma}

\begin{proof}
It follows from the definition \eqref{def:Ez} of $E(z)$ and the fact that
$P^{(\infty)}(z)$ is analytic for $z\in\cee\setminus\er$ that $E(z)$ is also
analytic for $z\in\cee\setminus\er$. For $x\in\er_+$, one checks that the jumps
of $P^{(\infty)}$ and of the rightmost factor in \eqref{def:Ez} cancel each
other out so that $E(x)$ is analytic for $x\in\er_+$. A similar calculation
yields the jump \eqref{jump:E} of $E(x)$ for $x\in\er_-$. (Alternatively, the
lemma could be proved with the help of the explicit formula
\eqref{Pinfty:explicit:tilde}.)

\end{proof}

We claim that the matrix $P^{(0)}(z)$ in \eqref{def:P0} satisfies the RH problem~\ref{rhp:P0}.
The jumps of $P^{(0)}$ on $\er_+\cup\Delta_1^{\pm}\cup \Delta_2^{\pm}$ are
easily checked. We are assuming here that the curves $\Delta_2^+$,
$\Delta_1^+$, $\Delta_1^-$, $\Delta_2^{-}$ are chosen near the origin so that
they are mapped to the rays $\Gamma_1$, $\Gamma_2$, $\Gamma_8$, $\Gamma_9$
respectively under the map $z\mapsto \sqrt{z}$. For $x\in\er_-$ we have
\begin{eqnarray*}
P^{(0)}_{+}(x) &=& E_+(x)N_+(n^{2/3}f(x);r_1(x),r_2(x),n^{2/3}s,n^{1/3}\tau)\Lam_+(x) \\
&=& E_-(x)\left(\begin{pmatrix} 0&1 \\ -1 & 0\end{pmatrix},
\begin{pmatrix} 0&1 \\ -1 & 0\end{pmatrix}\right) N_+(n^{2/3}f(x);r_1(x),r_2(x),n^{2/3}s,n^{1/3}\tau)\Lam_+(x)
\\
&=& E_-(x)N_-(n^{2/3}f(x);r_1(x),r_2(x),n^{2/3}s,n^{1/3}\tau) \left(\begin{pmatrix} 0&1 \\ -1 & 0\end{pmatrix},
\begin{pmatrix} 0&1 \\ -1 & 0\end{pmatrix}\right)\Lam_+(x)\\
&=& P^{(0)}_{-}(x)\diag\left(\begin{pmatrix} 0&1 \\ -1 & 0\end{pmatrix},
\begin{pmatrix} 0&1 \\ -1 & 0\end{pmatrix}\right),
\end{eqnarray*}
where in the second line we used \eqref{jump:E} and in the third line we used
\eqref{symmetry:N} and the fact that $r_1(z)=r_2(-z)$. This yields the required
jump on $\er_-$.

The matching condition \eqref{matching:P0} follows from \eqref{def:P0},
\eqref{lambdas:modelrhp} and \eqref{N:asymptotics2}. The fact that $P^{(0)}(z)$
has the correct behavior for $z\to 0$ follows from $\alpha=\nu-1/2$,
\eqref{def:P0}, \eqref{N:def} and a careful inspection of
\eqref{singularbehzero:1}--\eqref{singularbehzero:2biszet} and
\eqref{E0:pattern:0}.

\subsection{Final transformation of the Riemann-Hilbert problem}
\label{section:finaltransfo}

\begin{figure}[t]
\vspace{-13mm}
\begin{center}
   \setlength{\unitlength}{1truemm}
   \begin{picture}(100,70)(-5,2)
       \put(62.5,40){\line(1,0){32.5}}
       \put(93,36.6){$\er$}
       \put(79,40){\thicklines\vector(1,0){1}}
       \put(40,40){\thicklines\circle*{1}}
       \put(39.3,34.5){$0$}
       \qbezier(38.5,42)(30,49)(0,55)
       \qbezier(38.5,38)(30,31)(0,25)
       \put(18,53){$\Delta_1^{+}$}
       \put(18,25){$\Delta_1^{-}$}
       \put(20,50.3){\thicklines\vector(4,-1){1}}
       \put(20,29.6){\thicklines\vector(4,1){1}}
       \put(60,40){\thicklines\circle*{1}}
       \put(59.3,34.5){$q$}
       \qbezier(41.5,42)(50,53)(58.5,42)
       \qbezier(41.5,38)(50,27)(58.5,38)
       \put(48,49){$\Delta_2^+$}
       \put(48,28){$\Delta_2^-$}
       \put(50,47.5){\thicklines\vector(1,0){1}}
       \put(50,32.5){\thicklines\vector(1,0){1}}

       \put(40,40){\circle{5}}
       \put(60,40){\circle{5}}

  \end{picture}
   \vspace{-21mm}
   \caption{Jump contour $\Sigma_R$ in the RH problem for $R(z)$.}
   \label{fig:squaredBessel:rhpR}
\end{center}
\end{figure}
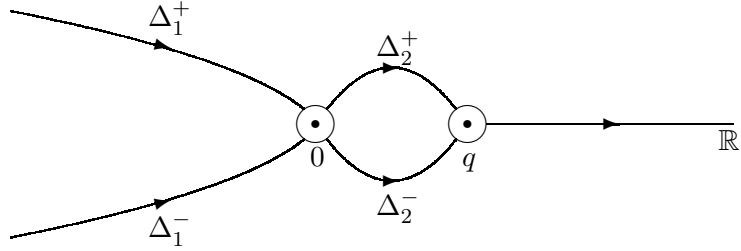

Using the global parametrix $P^{(\infty)}$ and the local parametrices $P^{(q)}$
and $P^{(0)}$, we define the final transformation $T \mapsto R$ by
\begin{equation}\label{def:R}
R(z) = \left\{
    \begin{array}{ll}
    T(z)(P^{(0)})^{-1}(z),& \quad \textrm{if }|z|<\rho, \\
    T(z)(P^{(q)})^{-1}(z),& \quad \textrm{if }|z-q|<\rho, \\
    T(z)(P^{(\infty)})^{-1}(z),& \quad \textrm{elsewhere}.
\end{array}
\right.
\end{equation}

\begin{lemma}\label{lemma:R}
With the contour $\Sigma_R$ as in Figure~\ref{fig:squaredBessel:rhpR}, we have
$R_+=R_-J_R$ on $\Sigma_R$ where, for a suitable constant $c>0$, the jump
matrices behave as
\begin{align*}
    J_R(z) & = I + O(n^{-1}), \qquad \text{on the circle $|z-q|=\rho$}, \\
    J_R(z) & = I + O(n^{-1/3}), \qquad \text{on the circle $|z|=\rho$}, \\
    J_R(z) & = I + O\left(\frac{e^{-c n^{1/3}}}{1+|z|}\right), \qquad \text{uniformly for $z$
    on the other parts of $\Sigma_R$}.
    \end{align*}
\end{lemma}

\begin{proof}
Only the statement about the circle $|z|=\rho$ requires further clarification.
This follows from \eqref{matching:P0}. Recall that we have here a disk with
\emph{fixed} radius $\rho$, in contrast to the shrinking disk in \cite{DKZ,DG}.
%
\end{proof}

Note that $R(z)$ has no jumps in the disk around the origin. From \eqref{def:R}
and the behavior of $T(z)$ and $P^{(0)}(z)$ for $z\to 0$, we also find that
$R(z)$ can have no pole at $z=0$ and therefore the singularity at the origin is
removable. Hence $R(z)$ is analytic at $z=0$.

By standard arguments, see e.g.\ \cite{Dei}, we obtain from Lemma~\ref{lemma:R} the estimate
\begin{equation}\label{R:estimate}
R(z)=I+ O\left( \frac{1}{n^{1/3}(1+|z|)}\right),
\end{equation}
as $n \to \infty$, uniformly for $z \in\cee \setminus \Sigma_R$.

\subsection{Proof of Theorem~\ref{theorem:kernelpsi}}
\label{section:proofTheoremKernel}

In this section we prove Theorem~\ref{theorem:kernelpsi}. We start from the
expression for the kernel in \eqref{kernel:Y:0}. Unfolding the transformations
$Y\mapsto X\mapsto U\mapsto T$ we obtain from \cite[Eq.~(6.1)]{DKRZ} that
\begin{multline}\label{kernel in terms of T}
K_n(x,y)=\frac{x^{\alpha/2}y^{-\alpha/2}}{2\pi i(x-y)}\begin{pmatrix} 0 &
-e^{n\lam_{2,-}(y)} & e^{n\lam_{3,-}(y)} & 0
\end{pmatrix}
T_{+}^{-1}(y)T_{+}(x)
\\ \times\begin{pmatrix}
0 & e^{-n\lam_{2,+}(x)}  & e^{-n\lam_{3,+}(x)} & 0
\end{pmatrix}^T.
\end{multline}

If $x,y$ lie within radius $\rho$ of the origin,  then it follows from
$T(z)=R(z)P^{(0)}(z)$ and the definition \eqref{def:P0} of $P^{(0)}$ that
\begin{multline}\label{kernel:R}
K_{n}(x,y)=\frac{x^{\alpha/2}y^{-\alpha/2}}{2\pi i(x-y)}\begin{pmatrix}0 & -1 &
1 & 0\end{pmatrix}
N_+^{-1}(n^{2/3}f(y);n^{2/3}s,n^{1/3}\tau) E_+^{-1}(y)R_{+}^{-1}(y)\\
\times R_{+}(x)E_+(x) N_+(n^{2/3}f(x);n^{2/3}s,n^{1/3}\tau)
\begin{pmatrix}
0 & 1 & 1 & 0
\end{pmatrix}^T.
\end{multline}
Now we put \begin{equation}\label{xy:uv}
x=\frac{u}{\kappa^2 n^{4/3}},\qquad y=\frac{v}{\kappa^2 n^{4/3}},\end{equation}
for fixed $u,v>0$. Using \eqref{r12stau} it then follows that
$$ n^{2/3}f(x)\to \sqrt{u},\qquad n^{2/3}f(y)\to \sqrt{v},
$$
and we also recall from \eqref{sn:limit} that
$$ n^{2/3}s\to s^*,\qquad n^{1/3}\tau\to \tau^*.
$$
Furthermore, we have from \eqref{R:estimate} and Cauchy's formula that
\begin{align}\label{RR:estimate}
R^{-1}(y)R(x)=I+O\left(\frac{x-y}{n^{1/3}}\right) =
I+O\left(\frac{u-v}{n^{5/3}}\right).\end{align} The constants implied by the $O$-symbols are independent of
$u$ and $v$ when $u,v$ are restricted to compact subsets of $\er_+$. 

Next we estimate the matrix $E(z)$ in \eqref{def:Ez}. We claim that the transformed matrix \begin{equation}\label{En:tilde}
\wtil E(z) := E(z)\diag\left(\begin{pmatrix}1&i\\ i& 1\end{pmatrix},\begin{pmatrix}1&i\\
i& 1\end{pmatrix}\right)\diag(\zeta^{-1/4},\zeta^{1/4},\zeta^{-1/4},\zeta^{1/4}),\qquad \zeta=n^{4/3}z,
\end{equation}
is analytic near $z=0$. Indeed it has no jumps,
by virtue of Lemma~\ref{lemma:jumpsE}, and moreover it behaves as $\mathcal O(z^{-3/4})$ as $z\to 0$ so there is no pole at $z=0$. 

From \eqref{En:tilde}, \eqref{xy:uv} and \eqref{def:Ez} we get the estimates $\wtil E(x)=O(n^{1/2})$ and $\wtil E^{-1}(y)=O(n^{1/2})$. 
By the analyticity of $\wtil E(z)$ and Cauchy's formula we then find
\begin{align*} \wtil E^{-1}(y)\wtil E(x)
= I+O\left(n(x-y)\right) =
I+O\left(\frac{u-v}{n^{1/3}}\right)
\end{align*}
as $n\to\infty$, uniformly for $u,v\in\er^+$. Combining this with \eqref{RR:estimate} we find
$$
\lim_{n\to\infty} \wtil E^{-1}(y) R_{+}^{-1}(y)R_{+}(x)\wtil E(x) = I.
$$

To use the above estimate, we should first express the matrix $E$ in \eqref{kernel:R}
in terms of its transformed counterpart $\wtil E$ in \eqref{En:tilde}.
This substitution releases an extra factor which multiplies from the left or right the matrix $N$ or $N^{-1}$, respectively, in \eqref{kernel:R}.
By combining this with the above estimates we find
\begin{multline}\label{kernel:R:bis}
\lim_{n\to\infty} \frac{1}{\kappa^2n^{4/3}} K_{n}\left(\frac{u}{\kappa^2
n^{4/3}},\frac{v}{\kappa^2 n^{4/3}}\right) =
\frac{u^{\alpha/2}v^{-\alpha/2}}{2\pi i(u-v)}\begin{pmatrix}0 & -1 & 1 &
0\end{pmatrix}
\what N_+^{-1}(\sqrt{v};s^*,\tau^*) \\
\times \what N_+(\sqrt{u};s^*,\tau^*)
\begin{pmatrix}
0 & 1 & 1 & 0
\end{pmatrix}^T,
\end{multline}
with
\begin{equation*}
\what N\left(\zeta^{1/2}\right) := \diag(\zeta^{1/4},\zeta^{-1/4},\zeta^{1/4},\zeta^{-1/4})
\diag\left(\begin{pmatrix}1&-i\\ -i& 1\end{pmatrix},\begin{pmatrix}1&-i\\
-i& 1\end{pmatrix}\right)N\left(\zeta^{1/2}\right),
\end{equation*}
for $\zeta=u,v$. Equivalently, by \eqref{N:def} and \eqref{Mhat},

$$ \what N\left(\zeta^{1/2}\right) = \diag(-i,1,1,i)\what M(\zeta)\diag\left(\begin{pmatrix}0& 1\\ -i&0\end{pmatrix},1,-i\right).
$$
Inserting this in \eqref{kernel:R:bis} we get
\begin{multline*}
\lim_{n\to\infty} \frac{1}{\kappa^2n^{4/3}} K_{n}\left(\frac{u}{\kappa^2
n^{4/3}},\frac{v}{\kappa^2 n^{4/3}}\right) =
\frac{u^{\alpha/2}v^{-\alpha/2}}{2\pi i(u-v)}\begin{pmatrix}-1 & 0 & 1 &
0\end{pmatrix}
\what M_+^{-1}(v;s^*,\tau^*) \\
\times \what M_+(u;s^*,\tau^*)
\begin{pmatrix}
1 & 0 & 1 & 0
\end{pmatrix}^T.
\end{multline*}
This finally proves \eqref{kernel at tacnode}.


\end{document}